\documentclass
{amsart}
\usepackage{amssymb}
\usepackage{amsmath}
\usepackage{amsfonts}
\usepackage{epsfig}
\usepackage{color}
\usepackage{epstopdf}
\usepackage{graphicx}
\usepackage{amsthm}
\usepackage{enumerate}
\usepackage[mathscr]{eucal}
\usepackage{verbatim}
\usepackage{setspace}
cm \hoffset -1.25 true cm \marginparwidth= 2 true cm
\newcommand{\vbf}{\bf}

\newcommand{\Be}{\begin{equation}}
\newcommand{\Ee}{\end{equation}}
\newcommand{\Bea}{\begin{eqnarray}}
\newcommand{\Eea}{\end{eqnarray}}
\newcommand{\Beas}{\begin{eqnarray*}}
\newcommand{\Eeas}{\end{eqnarray*}}
\newcommand{\Benu}{\begin{enumerate}}
\newcommand{\Eenu}{\end{enumerate}}
\newcommand{\Bi}{\begin{itemize}}
\newcommand{\Ei}{\end{itemize}}

\def\intslash{\rlap{\kern  .32em $\mspace {.5mu}\backslash$ }\int}
\def\qsl{{\rlap{\kern  .32em $\mspace {.5mu}\backslash$ }\int_{Q_x}}}

\def\emph#1{{\it #1 }}

\def\dist{{\text{\it dist\,}}}

\def\supp{{\text{\rm supp}}}

\def\lc{\lesssim}

\def\eps{\varepsilon}

\def\zet{\zeta}

\def\fC{{\mathfrak {C}}}

\def\fG{{\mathfrak {G}}}

\def\fK{{\mathfrak {K}}}

\def\fM{{\mathfrak {M}}}
\def\fN{{\mathfrak {N}}}

\def\fP{{\mathfrak {P}}}
\def\fQ{{\mathfrak {Q}}}

\def\fS{{\mathfrak {S}}}
\def\fT{{\mathfrak {T}}}

\def\fd{{\mathfrak {d}}}

\def\fm{{\mathfrak {m}}}

\def\fq{{\mathfrak {q}}}

\def\bbR{{\mathbb {R}}}

\def\cA{{\mathcal {A}}}
\def\cB{{\mathcal {B}}}

\def\cE{{\mathcal {E}}}
\def\cF{{\mathcal {F}}}
\def\cG{{\mathcal {G}}}

\def\cP{{\mathcal {P}}}
\def\cQ{{\mathcal {Q}}}
\def\cR{{\mathcal {R}}}

 at 10 true pt

\def\be#1{\begin{equation}\label{ #1}}
\def\endeq{\end{equation}}
\def\endal{\end{align}}
\def\bas{\begin{align*}}
\def\eas{\end{align*}}
\def\bi{\begin{itemize}}
\def\ei{\end{itemize}}

\def\eps{\varepsilon}

\def\emph#1{{\it #1}}
\def\textbf#1{{\bf #1}}

\newtheorem{thm}{Theorem}[section]
\newtheorem{cor}[thm]{Corollary}
\newtheorem{lem}[thm]{Lemma}
\newtheorem{prop}[thm]{Proposition}
\theoremstyle{definition}
\newtheorem{defn}[thm]{Definition}
\theoremstyle{remark}
\newtheorem{rem}[thm]{Remark}

\newcommand{\mbr}{{\mathbf r}}
\newcommand{\mbq}{\mathbf q}
\newcommand{\ball}[2]{B(#1,#2)}

\newcommand{\ltwo}{\|_{L^2_t(I)}}

\newcommand{\Lp}[2]{\Big\|_{L^{#1}{({#2})}}}

\newcommand{\tdelf}[2]{ \tdel f_{\!\qqq{#1}{#2}} }
\newcommand{\tdelg}[2]{ \tdel f_{\!\qqq{#1}{#2}} }
\newcommand{\tdeltg}[2]{  \tdel (\tau_{#2} f_{\!\qqq{#1}{#2}}) }

\newcommand{\ffqq}[1]{f_{\fq^{#1}}}
\newcommand{\qq}[1]{{\fq^{{#1}}}}
\newcommand{\qqq}[2]{{\fq^{#1}_{#2}}}

\newcommand{\tdff}[3]{[\tdel f_{\!\qqq{#1}{#2}}]^{#3}}
\newcommand{\tdfff}[2]{\tdel f_{\!\qqq{#1}{#2}}}

\newcommand{\tdfg}[3]{[\tdel f_{\!\qqq{#1}{#2}}]^{#3}}
\newcommand{\tdtfg}[2]{\tdel (\tau_{#2} f_{\!\qqq{#1}{#2}})}

\newcommand{\tdel}{T_\delta}

\newcommand{\fge}{\fG(\epsilon_\circ, N)}

\newcommand{\psiaep}{\psi_a^\varepsilon}

\newcommand{\mbn}{\mathbf n}

\newcommand{\ffqi}{f_{\!\fq^i}}

\newcommand{\tdtaf}[2]{\tdel (\tau_{#2} f_{\!\qqq{#1}{#2}})}
\newcommand{\Lpi}[1]{\Big\|_{L^{#1}_t(I)}}

\newcommand{\my}{\mathbf y}
 \newcommand{\mba}{\mathbf a}

\newcommand{\mbc}{\mathbf{c}}

\newcommand{\mbK}{\mathbf K}

\newcommand{\fgee}{\overline{\mathfrak G}(\epsilon_\circ,N)}
\newcommand{\qqs}[1]{\fq^{#1}}
\newcommand{\vpi}[2]{\fq^{#1}_{#2}}

\newcommand{\sdel}{S_\delta}

\newcommand{\ffqqs}[1]{f_{\fq^{#1}}}

\newcommand{\sphere}{\mathbb S^{d-1}}

\newcommand{\sctp}[2]{{\fS^{#1}_{\fd^{#2}}}}

\newcommand{\tti}{\fd^i}

\newcommand{\ti}[1]{\fd^#1}
\newcommand{\tii}[2]{\fd^{#1}_{#2}}

\newcommand{\sctffp}[2]{\fS^{}_{\tii{#1}{#2}}f}

\newcommand{\vpp}[2]{\fq^{#1}_{#2}}

\newcommand{\vppf}[2]{f_{\vpp{#1}{#2}}}
\newcommand{\vppff}[2]{f_{\vpp{#1}{#2}}}

\newcommand{\sqfrq}[3]{\fS^{}_{\fd^{#1}_{#2}} F_{\fq^{#3}_{#2}}}
\newcommand{\sqfr}[4]{\fS^{}_{\fd^{#1}_{#2}} f_{\fq^{#3}_{#4}}}

\newcommand{\smathbbz}{\sigma^{-1} \mathbb Z}
\begin{document}

\subjclass[2000]{42B15, 35B65} \keywords{Square function,
Bochner--Riesz means}

\author{Sanghyuk Lee}
\address{Sanghyuk Lee\\ School of Mathematical Sciences, Seoul National University, Seoul 151-742, Korea} \email{shklee@snu.ac.kr}

\title[Square function estimates]
{Square function estimates  for \\ the Bochner-Riesz means}

\begin{abstract} We consider the  square function (known as Stein's square function) estimate associated with the  Bochner-Riesz means.  The previously known range of sharp estimate is improved. Our results
are based on vector valued extensions of  Bennett-Carbery-Tao's multilinear (adjoint)
restriction estimate and adaptation of an induction  argument due to Bourgain-Guth.   Unlike the previous work by Bourgain-Guth on $L^p$ boundedness of the Bochner-Riesz means  in which  oscillatory operators  associated to the kernel were studied,  we take more direct approach by working  on Fourier transform side.  This enables us to obtain the correct order of smoothing which is essential for obtaining  the sharp estimates for the square functions. 
\end{abstract}

\maketitle

\maketitle
\section{Introduction}
\noindent We consider the Bochner-Riesz mean of order $\alpha$ which is defined by
\[\widehat{\cR^\alpha_t f}(\xi)=\Big(1-\frac {|\xi|^2}{t^2}\Big)_+^\alpha\,
\widehat f(\xi), \  \ t>0,\  \   \xi\in \mathbb R^d,  \ \   d\ge 2\,.\]
Let $1\le p\le \infty$. The Bochner-Riesz conjecture  is that the estimate 
\begin{equation}\label{lp}\| \cR^\alpha_t  f\|_p\le C\|f\|_p
\end{equation} holds (except $p=2$) if and only if
\Be\label{exponent}\alpha>\alpha(p)=\max\Big(d\,\Big|\frac12-\frac1p\Big|-\frac12,
0\Big).\Ee
The Bochner-Riesz mean which is a kind of summability method has been studied in order to understand  convergence properties  of Fourier series and integrals.  In fact, for $1\le p<
\infty$,  $L^p$  boundedness of $\cR^\alpha_t$ implies  
$\cR^\alpha_t f\to  f$ in $L^p$ as $t\to \infty$.  The necessary condition \eqref{exponent} has been known for a long time (\cite{fe2}, \cite[p.\,389]{st2}).

When $d=2$, the conjecture  was verified by Carleson and Sj\"olin \cite{cs} (also
see \cite{fe2}). In higher dimensions $d\ge 3$  the problem is still
open and partial results are known.  The conjecture was
shown to be true for $\max(p,p')\ge 2(d+1)/(d-1)$ by the argument due to
Stein \cite{fe1} (also see \cite[Ch.\! 9]{st2}) and the sharp $L^2\to
L^{2(d+1)/(d-1)}$ restriction estimate (the Stein-Tomas theorem) for the sphere \cite{tomas,stein84}. It was Bourgain \cite{b1, b2} who first made progress
beyond this result when $d=3$.
Since then,  subsequent progress had been paralleled  with those
of restriction problem.
Bilinear or multilinear generalizations  under transversality  assumptions have turned out to be most effective and fruitful tools.  These results have propelled progresses in this area and there is a large body of literature on restriction estimates and related problems. 
See \cite{tvv, tv1, w2, t5, l-v, lee1, lee2, lee3} for bilinear restriction estimates and related results,   \cite{becata, bogu, l-v1, bo-point, temur, bd, bennett, bbfl, jramos} for multilinear restriction estimates and their applications, and \cite{ gu1, shayya, gu2, li-guth, rzhang, ou-wang} (also, references therein) for most recent developments  related to polynomial partitioning method.

Concerning improved $L^p$ boundedness of the Bochner-Riesz means in higher dimensions, the sharp $L^p$ bounds for the Bochner-Riesz operator on the range $\max(p, p')\ge 2(d+2)/d$ were established by the author \cite{lee1} making use of the sharp bilinear  restriction estimate due to Tao \cite{t5}. When $d\ge 5$ further progress was recently  made by Bourgain and Guth  \cite{bogu}. They improved the range of the sharp (linear) estimates for the oscillatory integral operators of  Carleson-Sj\"olin type of which phases additionally  satisfy  elliptic condition (see \cite{stein84, b4, lee2} for earlier results) by using the multilinear estimates for oscillatory integral operators  due to Bennett, Carbery  and Tao \cite{becata} and a factorization theorem.  Also see \cite{cs, hor, stein84} and \cite[Ch 11] {st2} for the relation between the Bochner-Riesz problem and the oscillatory integral operators of Carleson-Sj\"olin type.

The following is currently the best known
result for the sharp $L^p$ boundedness of the Bochner-Riesz operator.

\begin{thm}[\cite{cs, lee1, bogu}]
\label{main} Let $d\ge 2$, $p\in [1,\infty]$, and  $p_{\circ}$ be defined by \Be
\label{p0} p_{\circ}=p_\circ(d)=2+ \frac{12}{4d-3-k} \quad \text{ if
} d\equiv k \,\, (\hspace{-3.5mm} \mod 3),\ k=-1, 0, 1. \footnote{For 
the sharp bound for $\max(p,p')\ge p_\ast$  the numerology are related as follows:  (bilinear)
$ p_\ast= 2+4/d$; (multilinear) $p_\ast= 3+3/d+O(d^{-2})$; (conjecture)
$p_\ast=2+2/d+O(d^{-2})$.} \Ee If $\max(p, p')\ge p_{\circ}$, then
\eqref{lp} holds for $\alpha> \alpha(p)$.
\end{thm}

There are also results concerning the endpoint estimates at the critical exponent $\alpha=\alpha(p)$
( for example, see  \cite{christ-rough, christ-wtBR, se1, tao-weak0}).
It was shown by  Tao  \cite{tao-weak} that the sharp $L^p$ bounds of $\cR^\alpha_t$ for  $1<p< p_\circ<2d/(d-1)$ imply 
  the weak type bounds of $\cR^{\alpha(p)}_t$ for $1<p<p_\circ$. We
  refer interested readers to \cite{lrs2} and references therein  for variants and related problems.

\subsubsection*{Square function estimate} We now consider the square function $\mathcal G^\alpha\! f$ which is defined by
$$\mathcal G^\alpha\! f(x)=
\Big(\int_0^\infty \Big|\frac{\partial}{\partial t} \cR^\alpha_t
f(x) \Big|^2 t\,dt \Big)^{1/2}.$$ It was introduced by
Stein~\cite{stein58} to study almost everywhere summability of
Fourier series.  Due to derivative in $t$ the square function behaves as if it is a multiplier of order $(\alpha-1)$ and the derivative ${\partial}/{\partial t}$ makes $L^p$ estimate possible by mitigating bad behavior near the origin.  
 In this paper we are concerned with the estimate \Be
 \label{square}
  \|\cG^\alpha
\!f\|_p\le C \|f\|_p\,.
\Ee

The $L^p$ estimate for the square function has various consequences and applications. First of all,  it is related to smoothing  estimates for solutions to dispersive equations associated to radial symbols such as wave and Schr\"odinger operators. See \cite{lrs, lrs1} for the details (also, {\it Remark} \ref{spherical}).  The sharp square function estimate implies the sharp maximal bounds for Bochner-Riesz means, which is to be discussed below in connection to pointwise convergence. It also gives   $L^p$ and maximal $L^p$ boundedness of general radial Fourier multipliers, especially the  sharp $L^p$ boundedness result of H\"ormander-Mikhlin type (see, Corollary \ref{radialmultiplier} below, \cite{cgt, caesc} and \cite{lrs}).

For  $1<p\le 2$, the inequality \eqref{square} is 
well understood. In this range of $p$, $\cG^\alpha$  is bounded on $L^p$
 if and only if  $\alpha>d(1/p-1/2)+1/2$
 (see~\cite{su1} and \cite{lrs2}).  Sufficiency can be shown by using the vector valued Calder\'on-Zygmund theory.    In contrast with the case $1<p\le 2$, if $p>2$, due to smoothing effect  resulting from averaging in
time the problem has more interesting features and  may be considered as a
vector valued extension of the Bochner-Riesz conjecture in that its sharp
$L^p$  bound also implies that of Bochner-Riesz operator.  The condition $\alpha>
\max\{1/2, d(1/2-1/p)\}$ is known to be necessary for \eqref{square} (see, for example \cite{lrs2})
and it is natural to conjecture that this is also sufficient for $p>2$.  This conjecture
in two dimensions was proven by Carbery~\cite{ca},  and in higher dimensions, $d\ge 3$,  sharp estimates  for $p> {2(d+1)}/{(d-1)}$  were obtained by
Christ~\cite{ch} and Seeger~\cite{se0} and it was later improved to the
range of $p\ge {2(d+2)}/{d}$ by the author,  Rogers, and Seeger \cite{lrs}. There are also endpoint estimates at the critical
exponent $\alpha=d/2-d/p$ and  weaker $L^{p,2}\to L^p$ endpoint
estimates were  obtained in \cite{lrs2} for $2(d+1)/(d-1)<p<\infty$.

There are two notable approaches for the study of Bochner-Riesz problem. The one which may be called the {\it spatial side approach}   is to prove the sharp estimates for the oscillatory integral operators of Carleson-Sj\"olin type \cite{cs, hor, stein84}.   These operators are natural
variable coefficient generalizations of the adjoint restriction
operators (\cite{b4, lee2, wise}) for  hypersurfaces with nonvanishing Gaussian curvature such as spheres,  paraboloids, and hyperboloids.  The other which we may call {\it frequency side approach} is more related to Fourier transform side, based on suitable decomposition in frequency side and orthogonality between the  decomposed  pieces \cite{fef2, ca, ch,christ-wtBR, se1, tao-weak, lee1}.  As has been demonstrated  in  related works the latter approach makes it possible to carry out finer analysis and to obtain refined results such as  the sharp maximal bounds, square function estimates,  and various endpoint estimates.

The  recently  improved bound for the Bochner-Riesz operator in \cite{bogu} was obtained from the
sharp estimate for the oscillatory integral operators of  Carleson-Sj\"olin type with additional elliptic  assumption.  
However,  this approach doesn't seem appropriate for the study of
the square function. Especially, there is an obvious difficulty when one tries to make use of disjointness of the
singularity of Fourier transform of $\cR_t^\alpha f$ which occurs as $t$ varies (for
example, see \eqref{sphericalsq}). This is where comes in the extra
smoothing of order $1/2$ for the square function estimate, 
which is most important for the sharp estimates for   $\cG^\alpha f$ (\cite{ca, ch, lee1, lrs}).   This kind of smoothing can be seen clearly in the Fourier transforms of Bochner-Riesz means  but  is not easy to exploit in the oscillatory kernel side.  As is already known \cite{b4, wise, lee2, bogu},
the behavior of the oscillatory integral operators of Carleson-Sj\"olin type  are more subtle 
and generally considered to be difficult to analyze when compared to
their constant coefficient  counterparts, the adjoint restriction
operators. So, we take {\it frequency side approach} in which we directly handle the associated
multiplier by working in frequency space rather than dealing with the oscillatory integral operator given by the kernel of the Bochner Riesz operator.

In this paper,  we obtain the  sharp square function estimates which are new  when $d\ge 9 $.

\begin{thm}\label{mainsquare} Let us set $p_s=p_s(d)$
by \begin{equation} \label{ps} {p_s}=2+\frac{12}{4d-6-k}, \,\,\,
d\equiv k   \,\, (\hspace{-3.5mm} \mod 3), \,\, k=0,1,2.\end{equation} Then, if  $p\ge  \min(p_s, \frac{2(d+2)}{d}) $ and
$\alpha>  d/2-d/p$, the estimate \eqref{square} holds.
\end{thm}

The range here does not match with that of Theorem
\ref{mainsquare}. This  results  from additional time average
which increases the number of decomposed frequency pieces. (See Section  \ref{pf-sq}.)

\subsubsection*{Maximal estimate and pointwise convergence} A straightforward consequence of the estimate \eqref{square} is the maximal estimate
\Be\label{mlp}\|\sup_{t>0} |\cR^\alpha_t f| \|_p\le
C\|f\|_p\,\Ee
for $\alpha>\alpha(p)$, which follows from Sobolev imbedding and \eqref{square}. Hence, Theorem \ref{mainsquare} yields
the sharp maximal bounds for $p\ge p_s(d)$.  When $p\ge2$, it has been conjectured that \eqref{mlp} holds as long
as \eqref{exponent} is satisfied. The sharp $L^2$ bound goes back to Stein
\cite{stein58}. The conjecture in $\mathbb R^2$ and the sharp bounds for $p>{2(d+1)}/({d-1})$, $d\ge 3$ were
verified by the square function estimates \cite{ch, se0}. The bounds were later improved to the range
$p>2(d+2)/d$ by the author \cite{lee1} using $L^p\to L^p(L^4_t)$ estimate.
The inequality \eqref{mlp} has been studied in connection with almost everywhere convergence
of Bochner-Riesz means. However, the problem of showing $\cR^\alpha_t f\to f$ a.e. for $f\in L^p$, $p> 2$,
$\alpha>\alpha(p)$ was settled by Carbery, Rubio de Francia and Vega \cite{caruve}.
Their result relies on weighted $L^2$ estimates.  There are also results on  pointwise convergence  at the critical  $\alpha=\alpha(p)$. See \cite{lee-seeger, amarco}. 
When $1
<p<2$, by Stein's maximal theorem  almost everywhere convergence of $\cR^\alpha_t f\to f$ for $f\in L^p$ is equivalent
to $L^p\to L^{p,\infty}$ estimate for the maximal operator   and  it was shown by Tao \cite{tao-weak} that the stronger
condition $\alpha\ge (2d-1)/(2p)-d/2$
is necessary for \eqref{mlp}. Except for $d=2$ (\cite{tao-maximal})  little is known beyond the
classical result which follows from interpolation between $L^2$ ($\alpha>0$) and $L^1$ ($\alpha>(d-1)/2$) estimates.

\subsubsection*{Radial multiplier} Let $m$ be a function defined on $\mathbb R_+$.  Combining the inequality due to Carbery, Gasper and Trebels \cite{cgt} and Theorem \ref{mainsquare},   we obtain 
 the following  $L^p$ boundedness
result of H\"ormander-Mikhlin type, which is sharp in that the regularity assumption can not be improved.  A similar result for the maximal function $ f\to \sup_{t>0}
|\cF^{-1}(m(t|\cdot|) \widehat f\,)|$ is also possible thanks to the
inequality due to Carbery (see \cite{caesc}).

\begin{cor}\label{radialmultiplier}
Let $d\ge 2$, and $\varphi$ be a nontrivial smooth function with
compact support contained in $(0,\infty)$. If $ \min(p_s, \frac{2(d+2)}{d}) \le
\max(p,p')<\infty$ and $\alpha> d|1/p-1/2|$, then
$$
 \big\|\cF^{-1}[ m(|\cdot|) \widehat f\,\,]\big\|_p
  \lc\, \sup_{t>0} \|\varphi
 m(t\cdot)\|_{L^2_\alpha(\bbR)}\, \|f\|_p.
$$
\end{cor}

\subsubsection*{About the paper.} In section 2, by working in frequency 
side we provide an alternative proof of  Theorem \ref{main}. Although, this doesn't  give improvement over the current range, we  include this
because it has some new consequences, clarifies several   issues, which were not clearly presented  in \cite{bogu}, and provides preparation for Section 3 in which we work  in vector valued setting. The proof in \cite{bogu} is sketchy and doesn't look readily accessible. Also the heuristic that  a function with Fourier support in a ball of radius $\sigma$ behaves  as if  it is constant on balls of radius $1/\sigma$ is now widely accepted and has important role in  the induction argument  but  it doesn't seem justified at high level of rigor.  We provide rigorous argument by making use of Fourier series  (see Lemma \ref{scattered} and Lemma \ref{smod}).  Another problem of the induction argument is  that the primary object (the associated surfaces or phase functions) changes in the course of induction. However, these issues are not properly addressed before in  literature.  We handle this matter  by introducing  a stronger induction assumption (see Remark \ref{stability}) and  carefully handling stability of various estimates.   We also use a different type multilinear decomposition which is more systematic, easier and efficient  for dealing with multiplier operators (see Section \ref{multi-scale},  especially   the discussion at the beginning of Section \ref{multi-scale}).

Section 3 is very much built on the  frequency side analysis in Section 2 as it may be regarded a vector valued extension of Section 2. Consequently, the structure of Section 3 is similar to that of Section 2 and   some of the arguments commonly work  in both sections. In such cases we try to minimize repetition while keeping readability as much as possible.  
We first obtain vector valued extensions of multilinear estimates (Proposition \ref{multisq}, Proposition \ref{confinedsqr}) which serve as  basic estimates for the sharp square function estimate.  Then, to derive linear estimate  (Theorem \ref{mainsquare})  we adapt the frequency  side approach in Section 2  to the vector valued setting and 
prove our main theorem.

Finally, oscillatory integral  approach has its own limit to prove Bochner-Riesz conjecture.   As is now well kown  (\cite{b4, wise, lee2, bogu}), the sharp $L^p$--$L^q$ estimates for the oscillatory operators of Carleson-Sj\"olin type  fail for  $q<q_\circ$, $q_\circ>  \frac{2d}{d-1}$  even under the elliptic condition on the phase \cite{wise, lee2, bogu}. Fourier transform side approach may help further development in a different direction and thanks to its flexibility may have applications to related problems. 

\subsubsection*{Notations.} The following is a list of notation we
frequently use for the rest of the paper.  
\vspace{-3mm}
\begin{enumerate} 

 \item[$\bullet$]  $C$, $c$ are constants which depend only on $d$
and may differ at each occurrence.

 \item[$\bullet$]  For $A, B\ge 0$, $A\lesssim B$ if there is a
constant $C$ such that $A\le CB$.

 \item[$\bullet$] $I=[-1,1]$ and $I^d=[-1,1]^d\subset  \mathbb R^d$.

 \item[$\bullet$]   $\tau_h f(x)=f(x-h)$ and $\tau_i f$ denotes 
 $\tau_{h_i} f$ for some $h_i\in \mathbb R^d$, $i=1,\dots, m$.

 \item[$\bullet$] We denote by ${\mathfrak q}(a, \ell)\subset
\mathbb R^{d}$ the closed cube centered at $a$ with sidelength
$2\ell$, namely, $ a+ \ell I^d$. If $\fq={\mathfrak q}(a, \ell)$,
denote   $a$, the center of $\fq$, by $\mbc(\fq)$.

 \item[$\bullet$] For $r>0$ and  a given cube or rectangle $Q$, we denote by 
 $rQ$ the cube or rectangle which is $r$-times  dilation of $Q$ from the center of $Q$.

 \item[$\bullet$]  Let $\rho\in \mathcal S(\mathbb R^d)$ be a function of which Fourier support is supported in
$\fq(0,1)$ and $\rho\ge 1$ on $\fq(0,1)$. And we also set
$\rho_{B(z,r)}(x):=\rho((\cdot-z)/r).$

 \item[$\bullet$]  For a given set $A\subset \mathbb R^d$, we define
the set $A+O(\delta)$ by
\[A+O(\delta):=\{x\in \mathbb R^d: \dist(x,A )< C\delta\}.\]

 \item[$\bullet$]  For a given dyadic cube $\fq$ and function $f$,
we define $f_\fq$ by $ \widehat {f_\fq}=\chi_\fq \widehat f.$

 \item[$\bullet$] Besides\, $\widehat{}\,$ and ${}^\vee$, \, $\mathcal
F(\cdot)$, $\mathcal F^{-1}(\cdot)$ also denote the Fourier
transform, the inverse Fourier transform, respectively.

 \item[$\bullet$] For a smooth function $G$ on $I^k$\,   $
\|G\|_{C^N(I^k)}:= max_{|\alpha|\le N} \max_{x\in I^k}
|\partial^\alpha G(x)|$
\end{enumerate}

\subsection*{Acknowledgement}  The research of the author was partially supported by NRF (Republic of Korea) grant  No. 2015R1A2A2A05000956. The author would like thank Andreas Seeger for discussions on related problems.

\section{Estimates for multiplier operators}\label{multiplier}
In this section we consider the  multiplier operators of
Bochner-Riesz type which are associated with  elliptic type surfaces.
They are natural generalizations of the Bochner-Riesz operator
$\cR^\alpha_1$. We prove the sharp $L^p$ boundedness of these of
operators and this provides an alternative proof of Theorem
\ref{main}. Basically  we adapt  the induction argument in \cite{bogu}. However, compared to (adjoint) restriction counterpart the induction argument becomes less obvious when we consider it for Fourier multiplier operator. However, exploiting sharpness of bounds for  frequency localized operator $T_\delta$ (see \eqref{tdel}, \eqref{sharp})  we manage to carry out a similar argument.  See Section \ref{closinginduction}.

 From now on we write \[ \xi=(\zeta, \tau)\in \mathbb
R^{d-1}\times \mathbb R.\] Let $\psi$ be a smooth function defined
on $I^{d}$ and $\chi_\circ$ be a smooth function supported in a
small neighborhood of the origin. We consider the multiplier
operator $T^\alpha=T^\alpha(\psi)$ which is defined by
\[\cF(T^\alpha\! f)(\xi)=\big(\tau-\psi(\zeta)\big)_+^\alpha\chi_\circ(\xi)\widehat f(\xi).\]
By a finite decomposition, rotation and translation and by
discarding harmless smooth multiplier, it is easy to see that the $L^p$
boundedness of $\cR_1^\alpha$  is equivalent to that of $T^\alpha$
which is given by $\psi(\zeta)=1-(1-|\zeta|^2)^{1/2}$. A natural generalization of the Bochner-Riesz problem is as follows: 
If $\det H\psi\neq 0$
on the support of $\chi_\circ$
(here, $H\psi$ is  the Hessian
matrix of $\psi$), we may conjecture 
 that, for $1\le p\le \infty$, $p\neq 2$,  
\Be\label{talpha} \|T^\alpha\! f\|_p\le C\|f\|_p \Ee if and only if
$\alpha>\alpha(p)$. From explicit  computation of the kernel of
$T^\alpha$  it is easy to see that the condition $\alpha>\alpha(p)$ is
necessary for \eqref{talpha}.  However, in this paper we only work with specific choices of 
$\psi.$

\subsection{Elliptic function}
Let us set \[\psi_\circ(\zeta)=|\zeta|^2/2.\]  For
$0<\epsilon_\circ\ll 1/2$ and an integer  $N\ge 100d$ we denote by
$\fG(\epsilon_\circ, N)$ the collection of smooth function which is
given by
\[ \fG(\epsilon_\circ, N)=\{ \psi: \|\psi-\psi_\circ\|_{C^N(I^{d-1})}\le \epsilon_\circ \}. \]
If $\psi\in \fG(\epsilon_\circ, N)$ and $a\in \frac12I^{d-1}$,
$H\psi(a)$ has eigenvalues $\lambda_1,\dots,\lambda_{d-1}$ close to
$1$ and we may write $H\psi(a)=P^{-1}DP$ for an orthogonal matrix
$P$ while $D$ is a diagonal matrix with its diagonal entries
$\lambda_1,\dots,\lambda_{d-1}$.  We denote by $\sqrt{H\psi(a)}$
the matrix $P^{-1}D'P$ where $D'$ is the diagonal matrix with its
diagonal entries $\sqrt{\lambda_1},\dots,\sqrt{\lambda_{d-1}}$. So,
$(\sqrt{H\psi(a)}\,)^2=H\psi(a)$.

For $\psi\in  \fG(\epsilon_\circ, N)$, $a\in \frac12I^{d-1}$, and
$0<\eps\le 1/2$, we define
\begin{equation} \label{normalization}\psi_a^\varepsilon (\zeta)=\frac1{\eps^2} \Big(\psi\big(\eps
\big[\sqrt{H\psi(a)}\,\big]^{-1} \zeta+a\big)-\psi(a)-\eps
\nabla\psi(a)\cdot \big[\sqrt{H\psi(a)}\,\big]^{-1}
\zeta\Big).\end{equation} Since $\psi\in \fge $, by Taylor's theorem
it is easy to see that $\|\psi_a^\eps-\psi_\circ\|_{C^N(I^{d-1})}\le
C\eps$ for $\psi \in \fge$. \footnote{Indeed, since
$|\partial^\alpha(\psi_a^\epsilon-\psi_\circ)|\lesssim
\epsilon^{|\alpha|-2}$ for any multiindex $\alpha$, we need only to
show $|\partial^\alpha(\psi_a^\eps-\psi_\circ)|\lesssim  \epsilon$ 
for $|\alpha|=0,1, 2$. This follows by Taylor's theorem since $N\ge
100d$.}  Hence we get  the following.

\begin{lem}\label{normal} Let $\psi \in \fge$ and $a\in \frac12 I^{d-1}$. Then there is a constant
$\kappa=\kappa(\epsilon_\circ, N)$, independent of $a,\,\psi$,  such that
$\psiaep\in \fge$ provided that $0<\eps\le \kappa$.
\end{lem}

\begin{rem}
If $\psi$ is smooth and $H\psi(a)$ has $d-1$ positive
eigenvalues, after finite decomposition and affine
transformations we may assume $\psi\in \fge$ for arbitrarily small
$\epsilon_\circ$ and large $N$. Indeed, for given $\eps>0$,
decomposing the multiplier
$\big(\tau-\psi(\zeta)\big)_+^\alpha\chi_\circ(\xi)$ to multipliers
supported in balls of small radius $\eps/C$ with some large $C$, one may assume
that $\mathcal F f$ is supported in $B((a,\psi(a)), \varepsilon/C)$.
Then, the change of variables \eqref{linearchange} transforms
$\psi\to \psi_a^\eps$ and give rise to  a new multiplier operator  $T^\alpha(\psi_a^\eps)$ and, as can be easily seen by  a simple change of variables, the  operator norm   $\|T^\alpha(\psi_a^\eps)\|_{p\to p}$  remains same. (See the proof of Proposition \ref{rescale}.)  By Lemma \ref{normal} we see   $\psi_a^\eps\in \fge$  if $\varepsilon$ is small enough. 
\end{rem}

\subsection{multiplier operator with localized frequency} Let $\phi$ be a smooth function
supported in $2I$. For $\delta>0$, $\psi\in \fge$,  and $f$ of which
Fourier transform is supported in $\frac12 I^d$ we define the
(frequency localized) multiplier operator $T_\delta =
T_\delta(\psi)$ by \Be \label{tdel}\widehat {T_\delta
f}(\xi)=\phi\Big(\frac{ \tau-\psi(\zeta)}\delta\Big) \widehat
f(\xi).
 \Ee
As is wellknown, the $L^p$ bound for $T_\delta$ largely depends on curvature  of  the surface $\tau=\psi(\zeta)$. By decomposing the
multiplier dyadically  away from the singularity $\tau=\psi(\zeta)$,
in order to prove \eqref{talpha}  for $p>
2d/(d-1)$ and $\alpha>\alpha(p)$,   it is enough to  show that,   for any $\epsilon>0$,
 \Be \label{sharp} 
\|T_\delta f\|_p\le C\delta^{\frac dp-\frac{d-1}2-\epsilon} \|f\|_p
\Ee  
whenever $\widehat f$ is supported in
$\frac12 I^d$.
The following recovers the sharp $L^p$ bound up to the currently best known range in \cite{bogu}.

\begin{prop} \label{localfrequency} Let $\epsilon>0$.
If $p\ge p_\circ(d)$ and $\epsilon_\circ$ is small enough, there is an
$N=N(\epsilon)$ such that \eqref{sharp} holds uniformly  provided
that $\psi\in\fge$ and $\supp\, \widehat f\subset \frac12 I^d$.
\end{prop}

It is possible to remove  loss of $\delta^{-\epsilon}$  in \eqref{sharp} by the $\epsilon$-removal argument in \cite{tao-weak} (in particular, see Section 4).

\subsubsection*{Induction quantity} To control $L^p$ norm of $T_\delta$, for
$0<\delta$, we define $A(\delta)= A_p(\delta)$ by
\begin{align*}
A(\delta)={\sup} \Big\{ \| T_\delta(\psi) f\|_{L^p}: \psi\in \fge,\,
  \|f\|_p\le 1,\, \supp \widehat f\subset \frac12
I^{d}
\Big\}.  \end{align*}

\begin{rem}\label{stability}
Though the induction argument in \cite{bogu} heavily relies
on stability of the multilinear estimates, such issue doesn't
seem properly addressed. In particular, after
(multiscale) decomposition and rescaling  the associated phase
functions (or surfaces) are no longer fixed phase functions (or
surfaces).\footnote{It is only true for the paraboloid.} This
requires the induction quantity defined over a class of phase
functions or surfaces. This leads us to consider $A(\delta)$. 
\end{rem}


From the estimate for  the kernel of $T_\delta$ (see Lemma \ref{kerneldelta1}), it is easy to see that
$A(\delta)\le C$ uniformly in $\psi\in\fge $ if $\delta\ge 1$ and
$A(\delta)\le C\delta^{-\frac{d-1}2}$ if $0<\delta\le 1$, because
$L^1$-norm of the kernel is uniformly $O(\delta^{-\frac{d-1}2})$.
 To prove Proposition  \ref{localfrequency}, we need to show
$A(\delta)\lesssim \delta^{\frac dp-\frac{d-1}2-\epsilon}$ for any $\epsilon>0$. However,
due to lack of monotonicity $A(\delta)$ is not suitable to close
induction.  So, we need to modify $A(\delta)$. For $\beta,$ $\delta>0$, we  define
\[ \mathcal A^\beta(\delta)=\cA_p^\beta(\delta):=\sup_{\delta <  s \le 1} s^{ \frac{d-1}2-\frac
dp +\beta}\,\,
 A_p(s).\]
Hence, Proposition \eqref{localfrequency} follows if we show
$\mathcal A^\beta(\delta)\le C$ for any $\beta>0$.

The following lemma makes precise the heuristic that the bound of
$T_\delta$ improves if it acts on functions of which Fourier
transforms are supported a smaller set. However, this becomes less obvious for multiplier operator  when it is compared to restriction (adjoint) operator ({cf.} \cite{bogu}).  This type of improvement is basically due to parabolic rescaling structure of the operator, and 
generally appears in $L^p$-$L^q$ estimates for $p,q$ satisfying
$(d+1)/q<(d-1)(1-1/p)$, $p\le q$, which are not invariant under the parabolic rescaling.  
The following is important for induction argument to work. 

\begin{prop}\label{rescale} Let $0< \delta\ll 1$, $\psi\in \fge$, and $(a,\mu)\in \mathbb R^{d-1}\times \mathbb R$. Suppose
that $\supp \widehat f\subset
{\mathfrak q}((a, \mu), \eps)\subset  \frac12 I^{d} $, $0<\epsilon<1/2$ and $\delta\le   (10)^{-2}\eps^2$. Then,
 there is a $\kappa=\kappa(\epsilon_\circ, N)$ such that  for $0<\eps\le
\kappa$ 
\Be\label{rescalee}   \|T_\delta  f\|_{p}\le 
CA(\eps^{-2} \delta) \|f\|_p\Ee holds  with $C$, independent of
$\psi$ and $\eps$.
\end{prop}

\begin{proof} Decomposing ${\mathfrak q}(a, \eps)$ into as many as $O(d^d)$, we may assume that
$\widehat f$ is supported in $\mathfrak q((a,\mu),\frac{\eps}{10d})$.   Since $\psi\in \fge$ and $\supp \widehat f\subset \mathfrak q((a,\mu),\eps/(10d))$, by Taylor's theorem we note  that 
$\phi(\frac{ \tau-\psi(\zeta)}\delta) \widehat
f(\xi)$ is supported  in the set 
\[ \Big\{ (\xi, \tau): | \tau -\psi(a)-\nabla\psi(a)\cdot (\zeta-a)| \le \frac{(1+\epsilon_0)   \eps^2}{2 \times 10^2}\Big\}\,. \] 
Hence, we may write 
\[ \phi\Big(\frac{ \tau-\psi(\zeta)}\delta\Big) \widehat
f(\xi)=   \phi\Big(\frac{ \tau-\psi(\zeta)}\delta\Big) \widetilde \chi  \Big(\frac{\tau -\psi(a)-\nabla\psi(a)\cdot (\zeta-a)}{\eps^2}\Big) \widehat
f(\xi),  \] 
where $\widetilde \chi$ is a smooth function supported in $\frac12 I$ such that $\widetilde \chi=1$ on $\frac14 I$. 
Let us set $ M=\big(\sqrt{H\psi(a)}\,\big)^{-1}$ and
make the change of variables in the frequency domain
\Be\label{linearchange}(\zeta,\tau)\to L(\zeta,\tau)=\big(\eps M\zeta+a,\,
\eps^2 \tau +\psi(a)+\eps \nabla\psi(a)\cdot M\zeta\big).\Ee Then it
follows that
\[\mathcal F\big( {T_\delta(\psi) f}\big)( L\xi)=\phi\Big(
\frac{\tau-\psi_a^\eps(\zeta)}{\eps^{-2}\delta}\Big)
\widetilde \chi(\tau) \widehat f(L\xi)\,.\]
 Since $L$ is an invertible affine transformation it is easy to see  $\| \mathcal F^{-1} (\widehat g(L\cdot))\|_p= \eps^{(d+1)(\frac1p-1)} \|g\|_p$ for any $g$.  We  also note that $\supp   ( \widetilde \chi(\tau)\widehat  
f(L\cdot)) \subset \frac{1}{2} I^d$ and   by Lemma \ref{normal} there exists a $\kappa>0$ such that 
$\psiaep\in \fge$ if $0<\eps\le \kappa$ whenever $\psi\in \fge$. So, by the definition of $A(\delta)$ it follows that, for $0<\eps\le \kappa$, 
 \begin{align*}
 &\|T_\delta(\psi) f\|_{p}= \eps^{(d+1)(1-\frac1p)}\Big\|  \mathcal F^{-1}\Big( \phi\Big(
\frac{\tau-\psi_a^\eps(\zeta)}{\eps^{-2}\delta}\Big)
\widetilde \chi(\tau) \widehat f(L\xi) \Big)\Big\|_p 
\\ 
&\le \eps^{(d+1)(1-\frac1p)}  A(\eps^{-2}\delta) \|  \mathcal F^{-1}(\widetilde \chi(\tau) \widehat f(L\xi) )\|_p
\le CA(\eps^{-2}\delta) \|f\|_p. 
\end{align*}
For the last inequality we also use the trivial bound $\|\mathcal F^{-1}(\widetilde \chi(\tau) \widehat g)\|_p\le C\|g\|_p$  for any $1\le p\le \infty$.   The inequality is valid for any  $\psi\in \fge$. This gives the
desired bound.
\end{proof}

We will need the following estimate which  is easy to show by making use of Rubio de Francia's one
dimensional  inequality \cite{ru0}.

\begin{lem}\label{vector} Let $\{\mathfrak q\}$ be a collection of (distinct)  dyadic cubes of the same side length $\sigma$.
Let $2\le p< \infty$. Then, there is a constant $C$, independent of
the collection $\{\mathfrak q\}$, such that
\[ \Big( \sum_{\mathfrak q} \|\mathcal F^{-1} (\widehat f\chi_{\mathfrak q})\|_p^p\Big)^\frac1p\le C\|f\|_p.\]
\end{lem}

\subsection{Multilinear estimates} In this subsection we consider various multilinear estimates which are basically consequences of multilinear restriction and Kakeya estimates  in \cite{becata}.

For $\psi\in \fge$ let us  set \[\Gamma=\Gamma(\psi)=\big\{(\zeta,
\psi(\zeta)): \zeta\in \frac12\,I^{d-1}\big\}.\] Let $2\le k\le d$,
and let $U_1, U_2, \dots, U_k$ be compact subsets of $I^{d-1}$. For
$i=1,\dots, k,$ and $\lambda>0$,  set
\[\Gamma_i=\big\{(\zeta, \psi(\zeta)):
\zeta\in U_i\big\}, \,  \Gamma_i(\lambda)=\Gamma_i+O(\lambda)\,.
\]
 For $\xi=(\zeta,
\psi(\zeta))\in \Gamma(\psi)$, let ${\mathrm N}(\xi)$ be the upward
unit normal vector at $(\zeta, \psi(\zeta))$.

For $v_1, \dots, v_k\in \mathbb R^d$,  denote by  $ V\!ol (v_1, \dots, v_k)$ the $k$-dimensional volume of the parallelepiped  given by  $ \{ s_1v_1+\dots+ s_k v_k :  s_i\in [0,1],\, 1\le i\le k\}.$   Transversality among the surfaces $\Gamma_1, \dots, \Gamma_k$ is important for  the multilinear estimates.  Degree of transversality is quantitatively stated as follows: 
\Be
\label{transverse} V\!ol({\mathrm N}(\xi_1), {\mathrm
N}(\xi_2),\dots, {\mathrm N}(\xi_k))\ge \sigma\Ee
for some $\sigma >0$ whenever $\xi_i\in \Gamma_i$, $i=1,\dots, k$. Since $ \psi\in \fge$,
$\nabla\psi$ is a diffeomorphism which is close to the identity map.
The condition \eqref{transverse} may be replaced by a simpler one that $V\!ol(\zeta_1,
\zeta_2, \dots,  \zeta_k|)\gtrsim \sigma$ whenever
$\zeta_i\in U_i$, $i=1,\dots, k$. The following is due to Bennett,
Carbery and Tao \cite{becata}.

\begin{thm}\label{multi-l2} Let $0<\delta\ll \sigma\ll 1$ and $\psi\in \fge$. Suppose that  $\Gamma_1, \dots,
\Gamma_k$ are given as in the above and \eqref{transverse} is
satisfied whenever $\xi_i\in \Gamma_i$, $i=1,\dots, k$, and suppose
that $\widehat F_i\subset \Gamma_i(\delta)$,  $i=1,\dots, k$. Then,
if $p\ge 2k/(k-1)$ and $\epsilon_\circ$ is sufficiently small, for
$\epsilon>0$  there are constants $N=N(\epsilon)$ such that, for $x\in
\mathbb R^d$,
\[\Big\| \prod_{i=1}^k F_i\Big\|_{L^\frac pk(B(x,\delta^{-1}))}\le
C\sigma^{-C_\epsilon}\delta^{-\epsilon} \prod_{i=1}^k
\delta^\frac12 \|F_i\|_{2}\] holds with $C, C_\epsilon$, independent
of $\psi$.
\end{thm}

Besides stability issue this estimate is essentially the same as  the multilinear restriction
estimate in \cite{becata}. (See \cite[Theorem 1.16]{becata} for the
case $k=d$ (also see Lemma 2.2) and  see \cite[Section 5]{becata}
for the case of lower linearity $2\le k<d$).  Though we are considering only  the
surfaces which are the graphs of $\psi\in \fge$,  but theorem remains
true for surfaces even with vanishing curvature as long as the
transversality condition is satisfied. Uniformity of the estimate
follows from the fact that the multilinear Kakeya and restriction
estimates are stable under perturbation of the associated surfaces.
The estimate is conjectured to be true without $\delta^{-\epsilon}$
loss (this is equivalent with the endpoint $k-$linear restriction
estimate) but it remains open when $k\ge 3$ even though the
corresponding endpoint case for the multilinear Kakeya estimate is
obtained by Guth \cite{gu}.

\begin{rem}\label{bound} The proof of  Theorem \ref{multi-l2} is based on the multilinear
Kakeya estimate and induction on scale argument which involves
iteration of  induction assumption to reduce the exponent of
$\delta^{-1}$. Such improvement of exponent is possible at the
expense of extra loss of bounds in terms of  $\sigma^{-c}$. By following the
argument in \cite{becata} one can easily see that one may take
$C_\epsilon\lesssim C\log\frac1\epsilon.$ (See the paragraph below \eqref{l2222}). Hence, the bound becomes less efficient
when $\sigma$ gets as small as $\delta^c$ for some
$c>0$.  In $\mathbb R^3$ the sharp bound depending on $\sigma$ was recently  obtained by Ramos \cite{jramos}.  However, the argument of Bourgain-Guth avoids such problem by
keeping Fourier supports of functions largely separated while   being decomposed. In contrast with the conventional approach in
which functions are usually decomposed into finer frequency pieces
this was achieved by decomposing the input functions into those of
relatively large frequency supports.
\end{rem}

\begin{lem} \label{kerneldelta1} Let $\varphi\in C_c^\infty(2I)$ and
$\eta\in C_c^\infty(I^d)$ which satisfies $1/2\le \eta\le 2$. Let
$0<\delta\ll \sigma\le 1$.  Set
\[K_\delta=\cF^{-1}\Big(\varphi\Big(\frac{\eta(\xi)(
\tau-\psi(\zeta))}{C\delta}\Big)\widetilde \chi(\xi)\Big),\] and
$\mathfrak K_M(x)=(1+\delta|x|)^{-M}$.
Suppose $\widetilde \chi$
is supported in a cube of sidelength $C\sigma$ and
$|\partial_\xi^\alpha\widetilde \chi|\lesssim \sigma^{-|\alpha|}$ for any $\alpha$.
Then, for any $M$, there is an $N=N(M)$  
such that \Be\label{kerneldelta} |K_\delta (x)|\le
C\delta\sigma^{d-1}\fK_M(x)
\Ee
with $C$ depending only on
$\|\psi\|_{C^N(I^{d-1})}$. 
\end{lem}
\begin{proof} Changing variables $\tau\to \delta\tau+\psi(\zeta)$,
we write
\[
K_{\delta}(x)=(2\pi)^{-d}\delta\int e^{i\delta\tau x_d}\int
e^{i(x'\cdot\zeta+ x_d\psi(\zeta))}\widetilde \varphi(\xi)
 d\zeta d\tau,\]
 where
 \[\widetilde \varphi(\xi)=\varphi\Big(\frac{\eta(\zeta,\delta\tau+\psi(\zeta)){\tau}}{C}\Big)
 \widetilde \chi(\zeta, \delta\tau+\psi(\zeta)).\]
We note that $|\partial_\zeta^\alpha \widetilde \varphi|\lesssim
\sigma^{-|\alpha|}(\|\psi\|_{C^{|\alpha|}}+\|\eta\|_{C^{|\alpha|}})$.
Then, if $|x'|/100\ge |x_d|$,  by integration by parts it follows that
\[\Big|\int
e^{i(x'\cdot\zeta+ x_d\psi(\zeta))} \widetilde \phi(\xi) d\zeta\Big|
\le C\sigma^{d-1}(\|\psi\|_{C^{M}(I^{d-1})}
+\|\eta\|_{C^{M}(I^{d})}) (1+\sigma|x'|)^{-M}\,.\]  Note that $
\widetilde \phi(\xi)=0 $ if $|\tau|\ge 5C$ since $1/2\le \eta\le 1$.
This gives the desired inequality \eqref{kerneldelta} by taking
integration in $\tau$ since $\delta\ll\sigma$. On the other hand, if
$|x'|/100< |x_d|$, we integrate in $\tau$ first. Since
$|\partial_\tau^l \widetilde \varphi|\lesssim
(\|\psi\|_{C^{l}}+\|\eta\|_{C^{l}})$, by integration by parts again
we have $|\int e^{i\delta\tau x_d} \widetilde \phi(\xi) d\tau|\le
C(\|\psi\|_{C^M(I^{d-1})}+ \|\eta\|_{C^{M}(I^{d})}) (1+|\delta
x_d|)^{-M}$. This and taking integration in $\zeta$ yield
\eqref{kerneldelta}.
\end{proof}

From Theorem \ref{multi-l2} and Lemma \ref{kerneldelta1} we can
obtain the sharp multilinear $L^p$ estimate for $T_\delta$ under transversality
condition without localizing the multilinear operator on a ball of radius $1/\delta$. In fact, since  $T_\delta f=K_\delta\ast   f$ and  the kernel $K_\delta$ (from Lemma \ref{kerneldelta1}) is rapidly decaying outside of $B(0,C/\delta)$,  one may handle $f$ as if $f$ were
supported in a ball $B$ of radius $\delta^{-1-\eps}$. This type of
localization and H\"older's inequality make it possible to lift $L^2$ estimate to
that of $L^p$, $p\ge 2$, with sharp bound. Such idea of deducing $L^p$ estimates from $L^2$
ones goes back to Stein \cite[p. 442-443]{st2} (\cite{fe1, fef2}),
and in \cite{lee1, lrs} the similar idea was used to make use of $L^2$
bilinear restriction estimate. The same argument also works with the
multilinear estimates with a little modification. We make it precise
in what follows.

\begin{prop}\label{multilp}
Let $0<\delta\ll \sigma\ll \widetilde \sigma\ll  1$ and $\psi\in
\fge$, and\, let $Q_1,\dots, Q_k\in \frac12 I^d$ be dyadic cubes of
sidelength $\widetilde \sigma$. Suppose that  \eqref{transverse} is
satisfied whenever $\xi_i\in \Gamma\cap Q_i$, $i=1,\dots, k$, and
$\supp\, \widehat f_i\subset Q_i$,  $i=1,\dots, k$. Then, if $p\ge
2k/(k-1)$ and $\epsilon_\circ$ is small enough,  for $\epsilon>0$
there is an $N=N(\epsilon)$ such that
\Be\label{klinearlp}\Big\| \prod_{i=1}^k T_\delta f_i\Big\|_{L^\frac
pk(\mathbb R^d)}\le
C\sigma^{-C_\epsilon}\delta^{-\epsilon} \prod_{i=1}^k
\delta^{\frac dp-\frac{d-1}2} \|f_i\|_{p}\Ee holds with $C,
C_\epsilon$, independent of $\psi$.
\end{prop}

\begin{proof} Set $\widetilde Q_i=\{\xi:
\dist(\xi, Q_i)\le \widetilde c\sigma\}$, and let $\widetilde
\chi_i$ be a smooth function supported in $\widetilde Q_i$  which
satisfies $\widetilde \chi_i=1$ on $Q_i$ and
$|\partial_\xi^\alpha\widetilde \chi_i|\lesssim \sigma^{-|\alpha|}$.
Let us define $K_i$ by
\begin{align*} \mathcal F(K_i)(\xi)=\phi\big(\frac{
\tau-\psi(\zeta)}{\delta}\big)\widetilde \chi_i(\xi). \end{align*}
Since $\widehat f_i$ is supported in $Q_i$, we have $\tdel f_i= K_i
\ast f_i.$

Let $\{\cB\}$ be the collection of boundedly overlapping balls of
radius $\delta^{-1}$ which cover $\mathbb R^d$. For $\eps>0$ we
denote by $\widetilde \cB$ the balls $B(a,\delta^{-1-\eps})$ if
$\cB=B(a,\delta^{-1})$. By decomposing $f_i=\chi_{\widetilde\cB} f_i
+ \chi_{\widetilde\cB^c} f_i$, we bound the $p/k$-th power of the
left hand side of \eqref{klinearlp} by
\begin{align*}
\sum_{\cB}\int_{\cB} \prod_{i=1}^k &|\tdel f_i(x)|^\frac pk
dx=\sum_{\cB}\int_{\cB} \prod_{i=1}^k |K_i\ast f_i(x)|^\frac pk
dx\lesssim I + I\!I,
\end{align*}
where
\begin{align*}
 I= \sum_{\cB}\int_{\cB} \prod_{i=1}^k
|K_i\ast(\chi_{\widetilde \cB} f_i)(x)|^\frac pk\,,
\,\,
I\!I= \sum_{\cB}\,\,\,\Big( \sum_{g_i= \chi_{\widetilde \cB^c} f_i \text{ for some } i  } \int_{\cB}
\prod_{i=1}^k
 |K_\delta\ast g_i(x)|^\frac pk
dx\Big).
\end{align*}
The second sum in $I\!I$  is summation  over  all possible choices  of   $g_i$ with  $g_i= \chi_{\widetilde
\cB} f_i \text{ or } \chi_{\widetilde \cB^c} f_i$, and $g_i= \chi_{\widetilde \cB^c} f_i \text{ for some } i$. So, in the product $\prod_{i=1}^k
 K_\delta\ast g_i(x)$ there is at least one $g_i$  which satisfies  $g_i= \chi_{\widetilde \cB^c} f_i$.

Since $\mathcal F(K_i \ast(\chi_{\widetilde \cB} f_i))\subset
\Gamma(\delta)\cap \widetilde Q_i$, taking a sufficiently small
$\widetilde c>0$, from continuity it is easy to see that $F_1=K_1
\ast(\chi_{\widetilde \cB} f_1), \dots, F_k=K_k
\ast(\chi_{\widetilde \cB} f_k)$ satisfy the assumption of
Theorem \ref{multi-l2}. So, by Theorem \ref{multi-l2}
and Plancherel's theorem we see
\begin{align*}
I&\lesssim \sigma^{-C_\eps}\big(\frac1\delta\big)^{\eps} \sum_{\cB}
\prod_{i=1}^k \delta^\frac{p}{2k} \big\|K_i \ast(\chi_{\widetilde
\cB} f_i)\big\|_2^\frac pk \le
\sigma^{-C_\eps}\big(\frac1\delta\big)^{\eps} \sum_{\cB}
\prod_{i=1}^k \delta^\frac{p}{2k} \big\|\chi_{\widetilde \cB}
f_i\big\|_2^\frac pk
\end{align*} for $\psi\in \fge$ and $\epsilon_\circ$ small enough.  Since $p> 2$,
by applying H\"older's inequality twice  we have
\begin{align*}
I &\lesssim
\sigma^{-C_\eps}\big(\frac1\delta\big)^{\eps}\prod_{i=1}^k
\delta^{\frac pk(\frac12+d(1+\eps)(\frac1p-\frac12))}
\big(\sum_{\cB} \big\|
\chi_{\widetilde \cB} f_i\big\|_p^p\big)^\frac 1{k} \lesssim \sigma^{-C_\eps}\big(\frac1\delta\big)^{c\eps}
\Big(\prod_{i=1}^k \delta^{\frac dp-\frac{d-1}2 }\big\|
f_i\big\|_p\Big)^\frac pk.
\end{align*}

  For $I\!I$, we use
Lemma \ref{kerneldelta1}. There is a constant
$C=C(\|\psi\|_{C^N(I^{d-1})})$ such that $|K_i\ast(\chi_{\widetilde
\cB^c} f_i)(x) |\le C\delta\delta^{\eps (M-d-1)} \fK_{d+1}\ast
|f_i|(x)$ if $x\in B$, and  $|K_i\ast g_i(x)|\le C\delta
\fK_{d+1}\ast |f_i|(x)$. Thus, we get
\begin{align*}
I\!I &\lesssim \delta^{\frac{(k-1)p}{k}}\delta^{\,\eps (N-d-1)\frac
pk}
 \int \prod_{i=1}^k\big(\fK_{d+1}\ast|f_i|(x)\big)^\frac pk dx \lesssim
 \delta^{c_2 N\eps-c_1} \prod_{i=1}^k \|
f_i\|_{p}^\frac pk.
\end{align*}
for some $c_1,$ $c_2>0$ because $\|\fK_{d+1}\ast f\|_p\le
C\delta^{-d}\|f\|_p$ for $1\le p\le\infty$ by Young's convolution
inequality. Combining two estimates for $I$ and $II$ with  $N$ large
enough,  we see that for $\eps>0$ there is an $N$ such that
\[ \Big\|
\prod_{i=1}^k T_\delta f_i\Big\|_{L^\frac pk(\mathbb R^d)}\le
C\sigma^{-C_\eps}\big(\frac1\delta\big)^{c\eps} \prod_{i=1}^k
\delta^{\frac dp-\frac{d-1}2} \|f_i\|_{p}\] for $\psi\in \fge$ and $\epsilon_\circ$ is small enough. Therefore,
choosing $\eps=\epsilon/c$, we get  the desired bound \eqref{klinearlp}.
\end{proof}

In what follows we show that if the normal vectors of the surfaces are confined in
$C\delta$-neighborhood of a $k$-plane in Proposition \eqref{confined}, then the associated multilinear restriction estimate has improved bound. In particular, if one takes $p=\frac{2k}{k-1}$, the bound in \eqref{l2} is $\sim\delta^{-\epsilon}\delta^{\frac d2}$, which is better than the corresponding bound $\sim \delta^{-\epsilon}\delta^{\frac k2}$ in
Proposition \ref{multilp}.  However, it seems difficult
to cooperate on such improvement  to get a better linear bound without using the square sum function (see Proposition \ref{cor-confined} below).

\begin{prop}\label{confined} Let $0<\delta\ll \sigma\ll 1$,
$\psi\in \fge$, and $\Pi$ be a $k$-plane containing the origin.
Suppose that $\Gamma(\psi)$, $\Gamma_1, \dots, \Gamma_k$ are given
as in the above and \eqref{transverse} is satisfied whenever
$\xi_i\in \Gamma_i$, $i=1,\dots, k$. Suppose that
\Be\label{frequency}\supp\,\widehat F_i\subset \Gamma_i(\delta)\cap
{\mathrm N}^{-1}(\Pi+O(\delta)), \, i=1,\dots, k. \Ee Then, if $2\le
p\le 2k/(k-1)$ and $\epsilon_\circ$ is sufficiently small, for
$\epsilon>0$ there is an $N=N(\epsilon)$ such that
\begin{equation}\label{l2}
\Big\| \prod_{i=1}^k F_i\Big\|_{L^\frac{p}{k}(B(x,\delta^{-1}))}\le
C\sigma^{-C_\epsilon}
\delta^{-\epsilon}\delta^{dk(\frac12-\frac{1}{p})} \prod_{i=1}^k
\|F_i\|_{2}
\end{equation}
holds with $C,$ $C_\epsilon$, independent of $\psi$.
\end{prop}

If  $p/k$ were bigger than equal to $\ge 1$, the inequality could be shown by using H\"older's inequality and $k$ linear multilinear restriction estimate in \cite{becata}.  However, this is not true in general and we  prove Proposition \ref{confined} by making use of the induction on scale argument and multilinear Kakeya estimate.   
The following is a consequence of Proposition \ref{confined}.
 
\begin{cor} \label{cor-confined} Suppose that the same assumptions in Proposition \ref{confined} hold.
Let $\{\fq\}$, $\fq\subset \frac12 I^d$,  be the collection of  dyadic cubes of side length
$\ell$,  $2^{-2}\delta< \ell \le 2^{-1}\delta$. Then, if $2\le p\le
2k/(k-1)$,  for $\epsilon>0$ there is an $N=N(\epsilon)$ such that,
for $x\in \mathbb R^d$,
\begin{equation}\label{squarefunt}
\Big\| \prod_{i=1}^k F_i\Big\|_{L^{\frac pk}(B(x,\delta^{-1}))}\le
C\sigma^{-C_\epsilon} \delta^{-\epsilon} \prod_{i=1}^k
\Big\|\Big(\sum_{\fq} | F_{i\,\fq}|^2\Big)^\frac12
\rho_{B(x,\delta^{-1})}\Big\|_{p}
\end{equation}
holds with $C$, $C_\epsilon$, independent of $\psi\in \fge$.
\end{cor}

This may be compared with a discrete  formulation of multilinear
inequality in \cite{bogu} (see (1.1), p. 1250).  The inequality \eqref{squarefunt}   can be easily deduced from Proposition \ref{confined} by the standard argument using 
Plancherel's theorem and orthogonality ({\it cf.} Proof of Corollary
\ref{sq}). So, we omit the proof.

\begin{proof}[Proof of Proposition \ref{confined}] For $p=2$ the estimate \eqref{l2} follows from
H\"older's inequality and Plancherel's theorem. Hence, in view of
interpolation, it is enough to show \eqref{l2} for $p=2k/(k-1)$.

We prove \eqref{l2} by adapting the proof of multilinear restriction
estimate in \cite{becata}. By translation we may assume $x=0$. We
make the following assumption that, for $0<\delta\ll \sigma$ and
some $\alpha>0$,
 \begin{equation}\label{l22}
\Big\| \prod_{i=1}^k
F_i\Big\|_{L^\frac{2}{k-1}(B(0,\delta^{-1}))}\lesssim
\delta^{-\alpha} \delta^{\frac d2}\prod_{i=1}^k \|F_i\|_{2}
\end{equation}
holds uniformly for $\psi\in \fge$ whenever \eqref{frequency} holds
and \eqref{transverse} is satisfied for $\xi_i\in \Gamma_i$
$i=1,\dots,k$. It is clearly  true with a large $\alpha>0$ as can be
seen by making use of Lemma \ref{kerneldelta1}. We show \eqref{l22}
implies that, for $\eps>0$, there is an $N$ such that
 \Be
\label{l2222}\Big\| \prod_{i=1}^k
F_i\Big\|_{L^\frac{2}{k-1}(B(0,\delta^{-1}))} \lesssim C_\epsilon
\sigma^{-\kappa}\delta^{-\frac\alpha 2-c\eps} \delta^{\frac{d}2}
\prod_{i=1}^k\|F_i\|_2 \Ee holds uniformly for $\psi\in \fge$. In what follows we set $R=\delta^{-1}$.

Iteration of implication from \eqref{l22} to \eqref{l2222} allows us
to suppress $\alpha$ as small as $\sim \eps$. In fact, since the
implication remains valid as long as $\psi\in \fge$,
by fixing an $\eps$ and iterating the implication \eqref{l22} $\to$
\eqref{l2222} $l$ times we have the bound 
\[C_\eps^l\sigma^{-\kappa l} R^{2^{-l}\alpha+
c\eps(1+2^{-1}\eps+\dots+2^{-l+1})}\le C_\eps^l \sigma^{-\kappa l}
R^{2^{-l}\alpha+ 2c\eps}.\] 
Choosing $l$ such that $2^{-l}\alpha\sim \eps$ gives the bound
$\widetilde C_\eps\sigma^{Ck\log\frac{\alpha}\eps} R^{C\eps}.$
Hence, taking $\eps=\epsilon/C$, we get the desired bound.

Let  $\{\mbq\,\}$ be the collection of dyadic cubes (hence
essentially disjoint) of sidelength $\ell$,
 $\ell< R^{-1/2}\le 2\ell$,
 so  that $\mathbb R^d=\bigcup \mbq $. Since the Fourier transform of $\rho_{B(z,\sqrt R)} F_i$ is
 supported in $\Gamma(\delta^\frac 12)\cap {\mathrm N}^{-1}(\Pi+O(\delta^\frac 12))$, by the assumption it follows  that
 \begin{align*}
&\Big\| \prod_{i=1}^k F_i\Big\|_{L^{\frac 2{k-1}}(B(z,R^{\frac12}))}
\lesssim \Big\| \prod_{i=1}^k \rho_{B(z,\sqrt R)}
F_i\Big\|_{L^{\frac 2{k-1}}(B(z,R^{\frac12}))}\\ &\lesssim
\delta^{-\frac\alpha 2} \delta^{\frac d4}\prod_{i=1}^k  \|
\rho_{B(z,\sqrt R)} F_i\|_{2} \lesssim \delta^{-\frac\alpha 2}
\delta^{\frac d4}\prod_{i=1}^k \Big\| \rho_{B(z,\sqrt R)}
\Big(\sum_{\mbq}|F_{i\,\mbq}|^2\Big)^\frac12 \Big\|_{2}.
\end{align*}
Here $F_{i\,\mbq}$ is given by $\mathcal F(F_{i\,\mbq})=\widehat{F_i}\chi_\mbq$. 
Since the supports of $\mathcal F(\rho_{B(z,\sqrt R)}F_{i\,\mbq})$
are boundedly overlapping,  the last inequality follows from
Plancherel's theorem. By rapid decay of $\rho$ we have, for a large 
$M>0$, \Be\label{lsq}
\begin{aligned}
 \Big\| \prod_{i=1}^k
F_i\Big\|_{L^{\frac 2{k-1}}(B(z, \sqrt R))}  \lesssim\delta^{-\frac\alpha 2} \delta^{\frac d4}\prod_{i=1}^k
\Big\| \chi_{B(z,R^{\frac12+\eps})}
\Big(\sum_{\mbq}|F_{i\,\mbq}|^2\Big)^\frac12 \Big\|_{2} + \delta^M
\prod_{i=1}^k \|F_i\|_2.
\end{aligned}
\Ee

For a given $\xi\in {\mathrm N}^{-1}(\Pi)$, let $\{v_1,\dots, v_{k-1}\}$ be an orthonormal basis for the
tangent space $T_{\xi}({\mathrm N}^{-1}(\Pi))$ at $\xi$,
$v_k={\mathrm N}(\xi)$, and let $v_{k+1}$, $\dots,$ $v_d$ form an
orthonormal basis for $(\text{span}\{v_1,\dots, v_{k-1}, 
v_k\})^{\perp}$. (So, the vectors  $v_1,\dots, v_{k-1},$  $v_{k+1}$, $\dots,$ $v_d$ depend on $\xi\in {\mathrm N}^{-1}(\Pi)$.)  Then,  we   define $\mathbf p(\xi)$ and $\mathbf P(\xi) $ by
\begin{align*}
&\mathbf p(\xi)=\xi+ \big\{x: |x\cdot v_j|\le C_1\sqrt \delta,\,
j=1, \dots, k-1,\,
  |x\cdot v_j|\le C_1\delta,\, j=k+1, \dots, d\,\big\},\\
 &{\mathbf P}(\xi)=\big\{x: |x\cdot v_j|\le  C\sqrt\delta 
,\, j=1, \dots, k-1,\,
  |x\cdot v_j|\le C,\, j=k+1, \dots, d\,\big\}.
  \end{align*}
 Since ${\mathrm N}^{-1}(\Pi)$ is smooth, ${\mathrm
N}^{-1}(\Pi)+O(\delta)$ can be covered by a collection of  boundedly
overlapping $\{\mathbf p(\xi_\alpha)\}$, $\xi_\alpha\in {\mathrm
N}^{-1}(\Pi)$ (here, we are seeing   ${\mathrm N}^{-1}(\Pi)$ as a subset of $\mathbb R^d$), such that for any $\mbq$ there exists
$\xi_\alpha$ satisfying \Be \label{include}\supp\, \widehat F_i\cap
\mbq\subset  \frac12\, \mathbf p(\xi_\alpha)\Ee with a sufficiently
large $C_1>0$. 

For $(i, \mbq)$ satisfying  $\supp\,\widehat F_i\cap \mbq\neq
\emptyset$   let us denote by $\xi_{i, \mbq}$ the $\xi_\alpha$ which
satisfies \eqref{include} (if there are more than one, we simply 
choose one of them).
We also denote by  $L(i,\mbq)$  the bijective affine map from
$\frac12\,\mathbf p(\xi_{i,\mbq})$ to $\fq(0,1)$. Then we define  
$\widetilde{ F_{i\,\mbq}}$ by
\[\mathcal F(\widetilde{ F_{i\,\mbq}})(\xi)=\frac1{\rho(L(i,\mbq)\xi)} \widehat {F_{i\,\mbq}}(\xi)\,.\]
We also set
$\mathbf P_{i,\mbq}=\mathbf P(\xi_{i, \mbq})$ and $\mbK_{i,\mbq}=\mathcal F^{-1}(\rho(L(i,\mbq) \,\cdot\,)).$  By $R {\mathbf P}_{i,\mbq}$ we denote the rectangle which is $R$ times dilation of ${\mathbf P}_{i,\mbq}$ from the center of ${\mathbf P}_{i,\mbq}$.  Also denote by $\widetilde {\mathbf P}_{i,\mbq}$ 
the  set $R^{1+\eps} {\mathbf P}_{i,\mbq}$ which is  the $R^{1+\varepsilon}$ times dilation of ${\mathbf P}_{i,\mbq}$ from its center.
Since $\mbK_{i,\mbq}\ast\widetilde{ F_{i\,\mbq}}=F_{i\,\mbq}$ and
$|\mbK_{i,\mbq}| \lesssim \frac{\chi_{R{\mathbf P}_{i,\mbq}}}{|R{\mathbf P}_{i,\mbq}|}$,
 we have, for $y\in B(x,2R^{\frac12+\eps})$ and some $c>0$, 
\[
|F_{i\,\mbq}(y)|^2=|\mbK_{i,\mbq}|\ast |\widetilde{F_{i\,\mbq}}|^2(y)
\lesssim \frac{\chi_{R{\mathbf P}_{i,\mbq} }}{|R{\mathbf P}_{i,\mbq} |}\ast|\widetilde{F_{i\,\mbq}}|^2(y)
\lesssim 
R^{c\eps}\frac{\chi_{\widetilde {\mathbf P}_{i,\mbq} }}{|\widetilde {\mathbf P}_{i,\mbq} |}
\ast|\widetilde{F_{i\,\mbq}}|^2(x).
\]
The last inequality is trivial  since $|\widetilde {\mathbf P}_{i,\mbq}|\sim R^{c\epsilon} |R{\mathbf P}_{i,\mbq}|$ for some $c>0$.  Hence,    for $x,y\in B(z,
R^{\frac12+\eps})$ we have
\Be \label{easy1} 
 \sum_{\mbq}|F_{i\,\mbq}|^2(y) \lesssim
 R^{c\eps}
\sum_{\mbq}\frac{\chi_{\widetilde {\mathbf P}_{i,\mbq} }}
{|\widetilde {\mathbf P}_{i,\mbq}
|}\ast|\widetilde{F_{i\,\mbq}}|^2(x).
\Ee
Taking integration in $y$ over $B(z, R^{\frac12+\eps})$ for each $1\le i\le k$, we see that, for $x\in B(z, R^{\frac12+\eps})$,
\Be 
 \prod_{i=1}^k \Big\| \chi_{B(z,R^{\frac12+\eps})}
\Big(\sum_{\mbq}|F_{i\,\mbq}|^2\Big)^\frac12 \Big\|_{2}\lesssim
R^{c\eps} R^\frac{dk}4 \prod_{i=1}^k
\Big(\sum_{\mbq}\frac{\chi_{\widetilde {\mathbf P}_{i,\mbq}
}}{|\widetilde {\mathbf P}_{i,\mbq} |}
\ast|\widetilde{F_{i\,\mbq}}|^2\Big)^\frac12(x).\Ee
 Now, integration
in $x$ over $B(z, R^{\frac12+\eps})$ yields
\Be 
\prod_{i=1}^k \Big\| \chi_{B(z,R^{\frac12+\eps})}
\Big(\sum_{\mbq}|F_{i\,\mbq}|^2\Big)^\frac12 \Big\|_{2}\lesssim
R^{c\eps} R^{\frac d4} \Big\|\prod_{i=1}^k
\Big(\sum_{\mbq}\frac{\chi_{\widetilde {\mathbf P}_{i,\mbq}
}}{|\widetilde {\mathbf P}_{i,\mbq} |}
\ast|\widetilde{F_{i\,\mbq}}|^2\Big)^\frac12 \Big\|_{{L^\frac
2{k-1}(B(z,R^{\frac12+\eps}))}}.
\Ee 
Combining this with \eqref{lsq} 
we have, for any large $M>0$,
\Be
\label{easy2}
\Big\| \prod_{i=1}^k F_i\Big\|_{L^{\frac
2{k-1}}({B(z,\sqrt R)})} \lesssim \delta^{-\frac\alpha
2-c\eps} \Big\| \prod_{i=1}^k
\Big(\sum_{\mbq}\frac{\chi_{\widetilde {\mathbf P}_{i,\mbq} }}{|\widetilde {\mathbf P}_{i,\mbq} |}
\ast|\widetilde{F_{i\,\mbq}}|^2\Big)^\frac12 \Big\|_{L^\frac 2{k-1}(B(z,R^{\frac12+\eps}))}
 + \delta^M \prod_{i=1}^k \|F_i\|_2\,.
\Ee

We now cover $B(0, R)$ with boundedly overlapping balls $B(z,\sqrt
R)$ and use the above inequality for each of them. Then we get
\[\begin{aligned}
 \Big\| \prod_{i=1}^k
F_i\Big\|_{L^{\frac2{k-1}}(B(0,R))} \lesssim \delta^{-\frac\alpha
2-c\eps} \Big\| \prod_{i=1}^k
\Big(\sum_{\mbq}\frac{\chi_{\widetilde {\mathbf P}_{i,\mbq}
}}{|\widetilde {\mathbf P}_{i,\mbq} |}
\ast|\widetilde{F_{i\,\mbq}}|^2\Big)^\frac12 \Big\|_{L^\frac
2{k-1}(B(0,\,2R))} + \delta^{M-C}\prod_{i=1}^k \|F_i\|_2\,.
\end{aligned}\]
Here we have an increased $c$ because of overlapping of the balls  $B(z,R^{\frac12+\eps})$ in the right hand side. 
Since $\sum_\mbq \|\widetilde{F_{i\,\mbq}}\|_2^2\sim \|F_i\|_2^2$,
for \eqref{l2222}
 it is sufficient to show
\[\Big\| \prod_{i=1}^k
\Big(\sum_{\mbq}\frac{\chi_{\widetilde {\mathbf P}_{i,\mbq}
}}{|\widetilde {\mathbf P}_{i,\mbq} |}
\ast|\widetilde{F_{i\,\mbq}}|^2 \Big)\Big\|_{L^\frac
1{k-1}(B(0,\,2R))} \lesssim \sigma^{-\kappa}\delta^{\frac{d}2-c\eps}
\prod_{i=1}^k\Big(\sum_\mbq \|\widetilde{F_{i\,\mbq}}\|_2^2\Big).\]
By rescaling this is equivalent to \Be\label{multi-kakeya}\Big\|
\prod_{i=1}^k \Big(\sum_{\mbq}\frac{\chi_{{\mathbf P}_{i,\mbq}
}}{|{\mathbf P}_{i,\mbq} |} \ast f_{i,\mbq} \Big)\Big\|_{L^\frac
1{k-1}(B(0,2))} \lesssim \sigma^{-\kappa} R^{c\eps}
\prod_{i=1}^k\Big(\sum_{\mbq} \|f_{i,\mbq}\|_1\Big).\Ee

Let $\mathcal I_{i}=\{\mbq: \supp\,\widehat F_i\cap \mbq\neq \emptyset\}$,  $I_i\subset \mathcal I_i$ and $\mathcal T_{i,\mbq}$ be a finite subset of
$\mathbb R^d$.  Allowing the loss of $(\log R)^C$ in bound, by a standard reduction with pigeonholing it suffices to show
\Be\label{multi-kakeya0}
\Big\| \prod_{i=1}^k
\Big(\sum_{\mbq\in I_i}\sum_{\tau\in \mathcal T_{i,\mbq}}{\chi_{{\mathbf P}_{i,\mbq} +\tau}}
 \Big)\Big\|_{L^\frac 1{k-1}(B(0,2))}
\lesssim \sigma^{-\kappa/2} R^{c\eps} \prod_{i=1}^k\Big(\sum_{\mbq\in I_i}\sum_{\tau\in
\mathcal T_{i,\mbq}}{|{\mathbf P}_{i,\mbq}+\tau|}\Big).\Ee We write
$x=(u,v)\in \Pi\times \Pi^\perp\,(=\mathbb R^d)$. Then the left hand
side is clearly bounded by
\[\sup_{v\in \Pi^\perp} \Big\| \prod_{i=1}^k
\Big(\sum_{\mbq\in I_i}\sum_{\tau\in \mathcal T_{i,\mbq}}{\chi_{{\mathbf P}_{i,\mbq} +\tau}} (\cdot, v)
 \Big)\Big\|_{L^\frac 1{k-1}(\widetilde B(0,2))},\]
  where $\widetilde B(0,\rho)\subset \mathbb R^k$ is the ball of radius $\rho$ which is centered at the origin.

 For $v\in \Pi^\perp $ let us set
\[({\mathbf P}_{i,\mbq} + \tau)^v=\{u: (u,v)\in {\mathbf P}_{i,\mbq} + \tau \}.\]
Then $({\mathbf P}_{i,\mbq} + \tau)^v$ is contained in a tube of
length $\sim 1$ and width $CR^{-1/2}$ of which axes are  parallel
with $\mathrm N(\xi_{i,\mbq})$. This is because the longer sides of
${\mathbf P}_{i,\mbq}$ except the one parallel to $\mathrm
N(\xi_{i,\mbq})$
 are transversal to $\Pi$. More precisely, we can show that if $\epsilon_\circ$ is sufficiently
 small and $N$ is large enough,  there a constant $c>0$, independent of $\psi\in \fge$, such
 that,
  for $w\in \big(T_{\xi_{i,\mbq}}(\mathrm N^{-1} (\Pi))\oplus \text{span}\{\mathrm
  N(\xi_{i,\mbq})\}\big)^\perp$,
\Be \label{angle} \measuredangle(w, \Pi)\ge c>0.\Ee Since
\eqref{transverse} is satisfied whenever $\xi_i\in \Gamma_i$,
$i=1,\dots, k$,  $\mathrm N(\xi_{1,\mbq}),\dots,\mathrm
N(\xi_{k,\mbq})$ which are, respectively,  parallel to the axes of tubes $({\mathbf
P}_{1,\mbq} + \tau)^v,\dots, ({\mathbf P}_{k,\mbq} + \tau)^v$
satisfy
 $|V\!ol(\mathrm N(\xi_{1,\mbq}),\dots,\mathrm N(\xi_{k,\mbq}))\gtrsim \sigma$.
 Also note that $|{\mathbf P}_{i,\mbq}^v +\widetilde \tau|\sim |{\mathbf P}_{i,\mbq}|$.
 Hence, by the multilinear Kakeya estimate in $\mathbb R^k$ (Theorem \ref{k-kakeya}) it follows that
\[ \Big\| \prod_{i=1}^k
\Big(\sum_{\mbq,\tau}{\chi_{{\mathbf P}_{i,\mbq} +\tau}}(\cdot, v)
 \Big)\Big\|_{L^\frac 1{k-1}(\widetilde B(0,2))}
 \lesssim \sigma^{-1}\prod_{i=1}^k\Big(\sum_{\mbq,\tau}{|{\mathbf P}_{i,\mbq}+\tau|}\Big).\]
This gives the desired inequality \eqref{multi-kakeya0}.

Now it remains to show \eqref{angle}.  By continuity, taking sufficiently small $\epsilon_\circ$,  we only need to show \eqref{angle} when
$\psi=\psi_\circ$ since $\|\psi-\psi_\circ\|_{C^N(I^{d-1})}\le
\epsilon_\circ$.  Though it is easy to show and intuitively obvious, we include a proof for clarity.    By rotation we may assume
$\Pi\cap\{x_d=-1\}=\{(\mathbf y, \mathbf a,-1):
 \my\in \mathbb R^{k-1}\}$ for some $\mba\in \mathbb R^{d-k}$.
 Since $\Pi$ contains the origin, $\Pi$ can  parametrized (except $\Pi\cap \{x_d=0\}$) as follows:
\Be \label{parametrization} s(\my,\mba,-1), \, s\in \mathbb R, \,
\my\in \mathbb
 R^{k-1}.\Ee  We may assume  $\Gamma_i(\delta)\cap ({\mathrm
N}^{-1}(\Pi)+O(\delta))\neq \emptyset$ because otherwise $F_i=0$ and
there is nothing to prove. Since $\mathrm N(\Gamma)\cap
\Pi=\emptyset$ if $|\mba|$ is large, so we may assume that
$|\mba|\le C$ for some $C>0$ and note that $\xi_{i,\mbq}\in
\Gamma(\psi)$. 
Furthermore, it suffices to show that
 \Be\label{ttt}\Pi\cap \Big(T_{\xi_{i,\mbq}}(\mathrm N^{-1} (\Pi))\oplus \text{span}\{\mathrm N(\xi_{i,\mbq})\}\Big)^\perp=
 \{0\},\Ee
which implies  $\measuredangle(w, \Pi)>0$ if $w\in
\big(T_{\xi_{i,\mbq}}(\mathrm N^{-1} (\Pi))\oplus
\text{span}\{\mathrm
  N(\xi_{i,\mbq})\}\big)^\perp$.
  Then, by continuity  and compactness
  \eqref{angle} follows.
We now verify  \eqref{angle} with $\psi=\psi_\circ$. By rotation we may assume $\mba=(0,\dots,0, a)=:(\mathbf 0, a)\in \mathbb R^{d-k-1}\times \mathbb R$. Using the above parametrization of $\Pi$, we see that \[\Pi=\text{span}\{e_1, \dots, e_{k-1}, (0,\dots,0, \mathbf 0, a,-1) \}.\]  The normal vector at $(x',|x'|^2/2)\in \mathbb R^{d-1}\times \mathbb R$ is parallel to $(x',-1)$. Hence, if $(x',|x'|^2/2)\in \mathrm N^{-1}(\Pi)$, that is, $(x',-1)\in \Pi$,  then $x'$ takes the form $x'=(\my, \mba)$ because of \eqref{parametrization}. Hence, it follows that  $\mathrm N^{-1}(\Pi)=\{(\my,\mathbf 0, a,\frac12(|\my|^2+|a|^2))\}$.  Then, if $\xi_{i,\mbq}=(\my,\mathbf 0,a,\frac12(|\my|^2+a^2))$,  $T_{\xi_{i,\mbq}}(\mathrm N^{-1}(\Pi))$ is spanned by  $\my_1=(1,0,\dots, \mathbf 0,0, y_1),$ $\my_2=(0,1,\dots,0, \mathbf 0,0, y_2),  \dots, \my_{k-1}=(0,0,\dots,1, \mathbf 0,0,$ $ y_{k-1})$. For \eqref{ttt} it is sufficient to show that $\fP:=\Pi\cap (\text{span}\{(\my,\mathbf 0,a,-1), \my_1, \dots, \my_{k-1}\})^\perp=\{0\}$. Let $w\in \fP$.  Then, since  $w\in \text{span}\{e_1, \dots, e_{k-1}, $ $ (0,\dots,0, \mathbf 0,  a,-1) \}$, we may write $w  =(c_1,\dots, c_{k-1}, \mathbf 0, c_k a, -c_k). $ Also, $w\cdot \my_1=\cdots=w\cdot \my_{k-1}=w\cdot (\my, \mathbf 0, a,-1)=0$ gives  $c_1=\dots=c_k=0$. So, $v=0$ and, hence, we get \eqref{ttt}.  This completes the proof.
\end{proof}

\subsection{Scattered modulation sum of scale $\sigma$}
When the Fourier transform of a given function $f$ is supported in a
ball of radius $\sigma$, then $f$ behaves as though it were constant
on balls of radius $\sigma^{-1}$. This observation has important role in
Bourgain-Guth's argument \cite{bogu} and is widely taken for granted without being made rigorous. There seems to be several ways which
make this heuristic rigorous (see \cite{tv2, t2}).
For this purpose we make use of Fourier series expansion.

Fix $\sigma>0$ and  large positive constants $M=M(d)\ge 100d$ and $C_M$ which are to be chosen to be large. 
For $l\in \sigma^{-1} \mathbb Z^d$ we  set 
\Be\label{A-tau} A_l= A_l(\sigma)=C_M(1+|\sigma l|)^{-M}, \quad \tau_{l} f(x)=f(x-l).\Ee
For $\sigma>0$, we define $[F]_{\sigma}$,
$|\![F ]\!|_{\sigma}$ (scattered modulation sum of $\sigma$-scale)
by
\Be \label{scatteredd} [F]_{\sigma}(x)=\sum_{l\in \sigma^{-1}\mathbb Z^d}
A_{l}|\tau_{ l} F(x)|,\,\, |\![F ]\!|_{\sigma}(x)= \sum_{l_1,l_2\in   \sigma^{-1} \mathbb Z^d}
A_{l_1}A_{l_2}
 |\tau_{l_1+l_2} F(x)|.\Ee
We have the following lemma.

\begin{lem}\label{scattered}  Let $\xi_0, x_0\in \mathbb R^d$. Suppose that $F$ is a function with
$\widehat F$ supported in $\mathfrak q(\xi_0, \sigma)$. Then, if
$x\in \fq(x_0, 1/\sigma) $,
\[|F(x)|\le
[F]_{\sigma}(x_0)\le |\![F]\!|_{\sigma}(x).\]
\end{lem}

 It should be noted  that the inequality holds regardless of
$\xi_0, x_0,$ and $\sigma$. 

\begin{proof}
Let $a$ be a smooth function
supported in $[-\pi, \pi]^d$ and $a(x)=1$ if $|x_i|\le 1,$ $i=1,\dots, d$.
Let us set
\[A(x,\xi)=a(x)a(\xi) e^{ix\cdot \xi}.\]
Since $|\partial^\alpha_\xi A|\le C_\alpha$ for any multi-indices $\alpha$,
by expanding $A$ into  Fourier series  in $\xi$ we have
\begin{equation}\label{fseries}
a(x)a(\xi) e^{ix\cdot \xi}=\sum_{l\in \mathbb Z^d} a_l(x) e^{-i \xi \cdot l}, \quad x,\, \xi\in [-\pi, \pi]^d
\end{equation}
while $a_l$ satisfies  $|a_l(x)|\le C_M(1+|l|)^{-M}$ for any large $M>0$.
On the other hand, from the inversion formula we have
\[F(x)= (2\pi)^{-d}\int e^{i(x-x_0)\cdot \xi_0} e^{i
(x-x_0)\cdot (\xi-\xi_0)}  e^{ix_0\cdot \xi} \widehat F(\xi) d\xi.\]
Hence, since $x\in \mathfrak q(x_0, \frac1\sigma)$, inserting the harmless bump function
$a$, we may write \begin{align*}
  &F(x)
= (2\pi)^{-d}\, e^{i(x-x_0)\cdot \xi_0}\int A\Big(\sigma({x-x_0}) ,
\frac{\xi-\xi_0}\sigma\Big) e^{ix_0\cdot \xi}\, \widehat F(\xi)\, d\xi.
\end{align*}
Using  \eqref{fseries} we have
\[\begin{aligned}
  &F(x)= (2\pi)^{-d}\, e^{i(x-x_0) \cdot \xi_0} \sum_{l\in\mathbb Z^d}  a_l\big(
\sigma({x-x_0})\big) \int e^{-i \frac{(\xi-\xi_0)}\sigma\cdot l} e^{ix_0\cdot
\xi}\, \widehat F(\xi)\, d\xi.
\end{aligned}
\]
 Then it follows that
\begin{equation}
\label{fexpansion}|F(x)|\le \sum_{l\in\smathbbz^d} A_l |\tau_{l} F(x_0)| \le \sum_{l_1,l_2\in \smathbbz^d}
A_{l_1}A_{l_2}
 |\tau_{(l_1+l_2)} F(x)|\,.
 \end{equation}
The second inequality follows by applying the first one to
each $\tau_{l} F$ with the roles of $x,$ $x_0$ interchanged.
\end{proof}

 \subsection{Multi-scale  decomposition}\label{multi-scale}
We now attempt to bound part of $T_\delta f$ with a sum of products
which satisfy the transversality assumption while the remaining parts
are given by  a sum of functions which have relatively small Fourier
supports. The first is rather directly estimated by making use of the
multilinear estimates and the latter is to be handled by Proposition
\ref{rescale}, the induction assumption and Lemma \ref{vector}. 

 In what follows, we basically adapt  the idea in \cite{bogu}.  However, concerning  the decomposition  in \cite{bogu}, reappearance of many small scale functions  in large scale decomposition  becomes  problematic  when one attempts to sum up  resulting estimates.  For the adjoint restriction estimates this can  be overcome  by using $L^\infty$-function as was done in \cite{bogu}.  But such argument doesn't work  for the multiplier operators and  leads loss in its bound.  To get over this,  unlike the decomposition in \cite{bogu} where  one starts to decompose with $d$-linear products and proceeds by reducing the
degree multi-linearity based on dichotomy, we decompose the multiplier operator by
increasing the degree of multi-linearity in order to avoid small scale functions appearing inside of  large scale ones. 
 This has a couple of advantages. First, this allows us to keep the function relatively intact  in the course of decomposition so that  we can easily add up decomposed pieces to obtain the sharp $L^p$ bound.        Secondly,   the decomposition makes it possible to obtain directly obtain  $L^p$-$L^p$  estimate. Hence we don't need to rely on the factorization theorem to deduce $L^p$-$L^p$  from $L^\infty$-$L^p$. (The same is also true for the adjoint restriction operators.) Hence, we can obtain the sharp $L^p$ bounds for multiplier operators of Bochner-Riesz type which lacks  symmetry.

\subsubsection{Spatial and frequency dyadic cubes} Let $0<\epsilon_\circ\ll 1$, $1\ll N$, $\psi\in \fge$, and
$T_\delta$ be given by \eqref{tdel}. Let
$\kappa=\kappa(\epsilon_\circ, N)$ be the number given in
Proposition \ref{rescale} so that \eqref{rescalee} holds whenever
$0<\eps\le \kappa$ and $\psi\in \fge$.   Let $m$ be an integer such that $2\le m\le d-1$, and $ \sigma_1, \dots,
\sigma_m$ be dyadic numbers such that \Be\label{dyadics}\delta\ll
\sigma_{m}\ll \dots \ll \sigma_1\ll \min(\kappa,1).\Ee These
numbers will be specified to terminate induction. We call $\sigma_i$ $i$-\emph{th  scale}.

Let us denote by $\{\fq^i\}$  the collection of the dyadic cubes $\fq^i$ of sidelength $2\sigma_i$ which are contained in $I^d$ (so, $\fq^i$ denotes the member of $\{\fq^i\}$ and the cubes $\fq^i$ are
essentially disjoint).  Rather than introducing new notation to denote  each collection of $\fq^i$, we take the convention that  $\{\fq^i\}$ denotes the collection of all dyadic cubes of  sidelength $2\sigma_i$ contained in $I^{d}$ if it is not specified otherwise. For each $i$-th scale there is a unique collection so that  there will be no  ambiguity, and we also use $\fq^i$ as indices which run over the set  $\{\fq^i\}$.   Thus,  we may write 
\Be\label{icube}
\bigcup_{\fq^i} \fq^i= I^d.
\Ee 
For the rest of this section,  we assume that 
\[\supp \widehat f\,\subset \frac12 I^d.\] Since $f= \sum_{\fq^i} \ffqi$, for
$i=1,\dots, m$, we write
 \Be \label{i-scale}
 \tdel f= \sum_{\fq^i} \tdel \ffqi.
 \Ee Clearly, we may
assume that $\fq^i$ is contained in $C\sigma_i$- neighborhood of the
surface $\Gamma(\psi)$ because $\tdel \ffqi=0$ otherwise.  So, in what follows, 
$\fq^i$, $\fq^i_1, \dots, \fq^i_{i+1}$ and $\fq^i_\ast$
denote the elements of $\{\fq^i\}$.

For convenience we extend in a trivial way the map $\mathrm N$ defined on
$\Gamma(\psi)$ to the cube $I^d$ by setting, for
$\xi=(\zeta,\tau)\in I^d$,
\[\mbn(\zeta,\tau)=\mathrm N(\zeta, \psi(\zeta)).\]
This extension is not  necessarily needed in what follows  because
we only consider a small neighborhood of $\Gamma(\psi)$. However, this allows us to 
define normal vector for any point  in $I^d$  and  makes  exposition simpler. This definition 
of $\mbn$ becomes coherent with the one given in the next section. 

\begin{defn}\label{transcubes} Let $k$ be an integer such that $1\le k\le m$ and  fix a constant $c>0$. 
 Let $\qqq {k} 1,\dots, \qqq {k}{k+1}\in \{\qqq k{}\}$ ($k$-th scale cubes). 
We say {\it $\qqq {k} 1,$ $\qqq {k}2,$ $\dots,$ $\qqq {k}{k+1}$ are $(\sigma_1,\sigma_2,\dots,\sigma_{k})$
transversal} if 
\Be\label{trans-k+1} V\!ol(\mbn(\xi_1), \mbn(\xi_2), \dots, \mbn(\xi_{k+1}))\ge c \sigma_1\sigma_2\dots\sigma_{k}, \Ee
whenever $\xi_i\in {\qqq {k}i}$, $i=1,\dots, k+1$. And we simply  denote
this by $\qqq {k} 1,$ $\qqq {k}2,$ $\dots,$ $\qqq {k}{k+1}: trans$ omitting dependence on  $\sigma_1,\sigma_2,\dots,\sigma_{k}$.
\end{defn}

Let us set  \[M_i=\frac1{\sigma_i},   \,\,  i=1,\dots, m.\] We denote by $ \{ \fQ^i \}$ the
collection of the dyadic cubes of sidelength $2M_i$, which
covers $\mathbb R^d$ (so, $\fQ^i$ again denotes a member of the sets
$\{ \fQ^i \}$). We write 
 \Be\label{spatialcube}
 \bigcup_{\fQ^i} \fQ^i=\mathbb R^d.\, \footnote{ Here we take the same convention for $\{ \fQ^i\}$ as we do for $\{\fq^i\}$.}
\Ee
 Since the Fourier support
of $\tdel \ffqi$ is contained $\fq^i$, it may be thought of as a
constant on $\fQ^i$ by invoking Lemma \ref{scattered} with
$\sigma=\sigma_i$. Since the scale $\sigma_i$ is clear from the side
length of the cube $\fq^i$, we simply set
\[
[\tdel f_{\!\fq^{i}}]:=[\tdel f_{\!\fq^{i}}]_{\sigma_i}, \,\, \,
|\![\tdel f_{\!\fq^{i}}]\!|:=|\![\tdel f_{\!\fq^{i}}]\!|_{\sigma_i}.
\]

\subsubsection{$\sigma_1$-scale decomposition} 
Bilinear decomposition is rather elementary. Fix $x\in \mathbb R^d$.  From
\eqref{i-scale} note that 
\[ |\tdel f(x)|\le \sum_{\qq1}|\tdel \ffqq 1(x)|\,.\] 
We
denote by $\qqq{1}{\ast}=\qqq{1}{\ast}(x)$  a cube $\qqq{1}{}\in \{\qq1\}$  such
that
$ |T_\delta f_{\qqq{1}{\ast}}(x)| =\max_{\qq1}|T_\delta f_{\qqq{1}{\ast}}(x)|.$
(There may be many such cubes
but $\qqq{1}{\ast}$ denotes just one of them.) Then we consider the following two
cases separately:
\[ \sum_{\qq1}|\tdel \ffqq 1(x)|\le 100^d  |T_\delta f_{\qqq{1}{\ast}}(x)|, \
\sum_{\qq1}|\tdel \ffqq 1(x)|>  100^d  |T_\delta f_{\qqq{1}{\ast}}(x)|.\]
For the second case 
$\sum_{\dist(\qq1,\qqq{1}{\ast})< 10\sigma_1} |\tdel \ffqq 1(x)|  < 50^d
 |T_\delta f_{\qqq{1}{\ast}}(x)| \le 2^{-d}\sum_{\qq1}|\tdel \ffqq 1(x)|.$ 
 Hence there is
 $\qqq11\in \{\qqq1{}\}$ such that $\dist(\qqq11,
\qqq{1}{\ast})\ge 10\sigma_1 $ 
and  
\[\sum_{\qq1} |\tdel \ffqq 1(x)| \lesssim  \sigma_1^{-{(d-1)}} |\tdfff 11{}(x) |\le \sigma_1^{-{(d-1)}}
|\tdfff 11{} (x)T_\delta f_{\qqq{1}{\ast}}(x)|^\frac12.\] 
From these two cases
we get  
\begin{align}
\label{scale1}
\sum_{\qq1}&|\tdfff 1{}(x)|
 \lesssim \max_{\qq1}| \tdfff 1{}(x)|+C\sigma_1^{-(d-1)/2}\max_{\substack{\dist(\qqq11,\qqq12)\gtrsim
\sigma_1}}|\tdfff 11{}(x) \tdfff 12{}(x)|^\frac12.
\end{align}
Using imbedding $\ell^p\subset\ell^\infty$, Proposition \ref{rescale} and Lemma \ref{vector} give
\Be\label{maxp}
\|\max_{\qq1}| \tdfff 1{}\|_p\le
\Big(\sum_{\qq 1} \| \tdfff 1{}\|_p^p\Big)^\frac1p 
 \le \Big(\sum_{\qq 1} A(\sigma_1^{-2}\delta)^p\|f_{\qq
1}\|_p^p\Big)^\frac1p \lesssim  A(\sigma_1^{-2}\delta)\|f\|_p.
\Ee
 Hence,  combining this with \eqref{scale1},  we have
\begin{align}\label{decomp1}
\|\tdel f\|_p \lesssim   A(\sigma_1^{-2}\delta)\|f\|_p+
\sigma_1^{-C}
\max_{\dist(\qqq11,\qqq12)\gtrsim \sigma_1}\| \tdfff 11 \tdfff
12\|_{\frac p2}^\frac12.
\end{align}
We now proceed to decompose the bilinear expression appearing  in the left hand side.

In the following section we explain how one can achieve trilinear decomposition out of \eqref{decomp1} before 
we  inductively obtain the full  $k$ linear decomposition which we need for the proof of Theorem \ref{main}.   Once one gets familiar with it,  extension to higher degree of multi-linearity becomes more or less obvious.

\subsubsection{$\sigma_2$-scale decomposition}\label{s2scale}    Suppose that we are
given two cubes $\qqq11$ and $\qqq12$ of $1$st scale such that
$\dist(\qqq11,\qqq12)$ $\gtrsim \sigma_1$. For $i=1,2$, we denote by
$\{\qqq2i\}$ the collection of dyadic cubes $\qqq2i$ of sidelength
$\sigma_2$ contained in $\qqq1i$ such that \Be\label{dyadics2}
 \qqq1i=\bigcup_{\qqq2i} \qqq2i, \,\,\,i=1,2.\Ee
 We also denote by $\{\qq2\}$ the set
$\{\qqq21\}\cup \{\qqq22\}$. 
Then it follows that
\Be\label{decomp-2} \tdelf 1i=\sum_{\qqq 2i} \tdelf 2i, \ \
i=1,2.\Ee We may also assume that $\qqq21$, $\qqq22$ are contained in the
$C\sigma_2$-neighborhood of $\Gamma(\psi)$ because $\tdelf 21$,
$\tdelf 22$ are zero otherwise.

Decomposition from this stage is no longer simple as in the  $\sigma_1$-scale  case.  We need to use spatial localization in order  to compare the values of the decomposed pieces. This makes it possible to bounds large part of the operator with transversal  products.

 Let us fix a
cube $\fQ^2$ and $x_0$ be the center of $\fQ^2$. Let
$\qqq2{1\ast}\in \{\qqq21\}$, $\qqq2{2\ast}\in \{\qqq22\}$ be the
cubes such that
\[ \tdff{2}{1\ast}{}\!(x_0)=\max_{\qqq21} \tdff 21{}(x_0), \quad 
\tdff{2}{2\ast}{}\!(x_0)=\max_{\qqq22} \tdff 22{}(x_0).\]
 Let us define 
 $\Lambda_i^2\subset  \{ {\fq^2_i}\}$, $i=1,2,$  by
 \[
\Lambda_i^2= \big\{\qqq2i: \tdff2i{}(x_0)\ge \sigma_2^{2d}
\max\big(\tdff{2}{1\ast}{}\!(x_0),\,\tdff{2}{2\ast}{}
\!(x_0)\big)\big\} .\]
Using \eqref{decomp-2}, we split the summation to get
\Be\label{decomp-3}\tdelg 11 \tdelg 12=\sum_{(\qqq21, \qqq 22)\in
\Lambda_1\times \Lambda_2} \tdelg 21\tdelg 22+ \sum_{(\qqq21, \qqq
22)\not\in \Lambda_1\times \Lambda_2} \tdelg 21\tdelg 22.\Ee Since
there are at most $O(\sigma_2^{-2(d-1)})$ $(\qqq21, \qqq 22)$,  the
second sum in the right hand side is bounded by
\Be\label{err1}\sum_{(\qqq21, \qqq 22)\not\in \Lambda_1\times
\Lambda_2} |\tdelg 21(x)||\tdelg 22(x)| \le \sigma_2^{d}\max_{\qqq
2{}}(\tdff{2}{}{}(x_0))^2. \Ee

For a cube $\mathfrak q$ we denote by $\mathbf c(\mathfrak q)$ the center of $\mathfrak q$.
Let $\Pi=\Pi(\qqq2{1\ast}, \qqq2{2\ast})$ be the $2$-plane which is
spanned by  $\mbn_1=\mbn(\mathbf c({\qqq2{1\ast}}))$,
$\mbn_2=\mbn(\mathbf c({\qqq2{2\ast}}))$,  and define
\Be\label{neighbor} \mathfrak N=\fN(\fQ^2, \qqq11, \qqq12)
=\big\{\fq^2\in \Lambda^2_1\cup \Lambda^2_2  : \dist(\mbn(\fq^2),
\Pi)\le C\sigma_2 \big \}.\Ee
Clearly, $V\!ol(\mbn_1, \mbn_2)\gtrsim \sigma_1$ and
$\dist(\mbn(\qqq2{}), \Pi)\gtrsim \sigma_2$ if
 $\qqq2{}\not \in \mathfrak N$. Since $\sigma_1\gg \sigma_2$,  if $\qqq2{}\not \in \mathfrak N$, then
$ V\!ol(\mbn_1,\mbn_2, \mbn(\xi))\gtrsim \sigma_1\sigma_2 $ for
$\xi\in {\qqq2{}}$.  Also,
$\mbn(\qqq2{i\ast})\subset\mbn_i+O(\sigma_2)$, $i=1,2$. So, it
follows that 
\Be\label{trans-2}
V\!ol(\mbn(\xi_1), \mbn(\xi_2), 
\mbn(\xi_3))\gtrsim \sigma_1\sigma_2\Ee if $\xi_1\in \qqq2{1\ast}$,
$\xi_2\in \qqq2{2\ast}$, and $\xi_3\in \qqq2{}\not \not\in \mathfrak
N.$ That is,  $\qqq2{1\ast},$ $\qqq2{2\ast}$, $\qqq2{}$ are
transversal.
Hence, we split $\sum_{(\qqq21, \qqq 22)\in \Lambda_1\times
\Lambda_2}\tdelg 21\tdelg 22$ into
\Be\label{decomp-4}\sum_{\substack{(\qqq21, \qqq 22)\in \Lambda_1\times \Lambda_2\,:\,
\qqq21,\qqq22\in \fN}}
\tdelg 21(x)\tdelg 22(x)+\sum_{\substack{(\qqq21, \qqq 22)\in \Lambda_1\times \Lambda_2\, :\,
\qqq21 \text{ or } \qqq22\not\in \fN}}\tdelg 21(x)\tdelg 22(x).\Ee
Each term appearing in the second sum can be bounded by a product of
three operators which satisfy transversality condition. Indeed,
suppose that $(\qqq21, \qqq 22)\in \Lambda_1\times \Lambda_2$ and
$\qqq22\not\subset \fN$. The case that $\qqq21\not\subset \fN$ can
be handled similarly by symmetry.  Since $\tdff22{}(x_0)\ge
\sigma_2^{2d} \tdff{2}{1\ast}{}\!(x_0)$, we have
\begin{align*}\tdff
21{}(x_0)\tdff 22{}(x_0) &\le \big(\tdff 2{1\ast}{}(x_0)\tdff
2{2\ast}{}(x_0))^\frac23 (\tdff 21{}(x_0)\tdff22{}(x_0)\big)^\frac13 \\
&\le \sigma_2^{-2d/3} \big(\tdff 2{1\ast}{}(x_0)\tdff
2{2\ast}{}(x_0)\tdff 22{}(x_0)\big)^\frac23.
\end{align*}
Hence, from this and \eqref{trans-2} it follows that
\Be\label{transss}
\Big|\sum_{\substack{(\qqq21, \qqq 22)\in \Lambda_1\times \Lambda_2 \, : \, 
\qqq21 \text{ or } \qqq22\not\in \fN}}\tdelg 21(x)\tdelg
22(x)\Big| \le \sigma_2^{-C}\sum_{\qqq21, \qqq 22, \qqq23:trans}
\big(\prod_{i=1}^3\tdff 2{i}{}(x_0)\big)^\frac23. \Ee
We combine \eqref{decomp-3}, \eqref{err1}, \eqref{decomp-4} and
\eqref{transss} to get,   for $x\in \fQ^2$,
\begin{align*}
|\tdelg 11(x)\tdelg 12(x)|
&\le \sigma_2^{d} (\max_{\qq2}\tdff{2}{}{}(x_0))^2+
|\sum_{\substack{(\qqq21, \qqq 22)\in \Lambda_1\times \Lambda_2\,:\,
\qqq21,\qqq22\in \fN}}\tdelg 21(x)\tdelg 22(x)|
\\
&\qquad+
\sigma_2^{-C}\sum_{\qqq21, \qqq 22, \qqq23:trans} \big(\prod_{i=1}^3\tdff 2{i}{}(x_0)\big)^\frac23.
\end{align*}
 Using Lemma
\ref{scattered} again,  we have,  for $x\in \fQ^2$,  
\Be
\label{scale2}
\begin{aligned}
|\tdelg 11(x)\tdelg 12(x)|
&\le \sigma_2^{d} (\max_{\qq2}|\!\tdff{2}{}{}\!|(x))^2+
|\sum_{\substack{(\qqq21, \qqq 22)\in \Lambda_1\times \Lambda_2\,:\,
\qqq21,\qqq22\in \fN}}\tdelg 21(x)\tdelg 22(x)|
\\
&\qquad+
\sigma_2^{-C}\sum_{\qqq21, \qqq 22, \qqq23:trans} \big(\prod_{i=1}^3|\!\tdff 2{i}{}\!|(x)\big)^\frac23.
\end{aligned}
\Ee

Taking $L^{p/2}$ on the  both sides of inequality (integrating on each of $\fQ^2$), summing along $\fQ^2$, 
and using Proposition \ref{rescale} and Lemma \ref{vector},  we get
\Be\label{scale22}
\begin{aligned}
\|\tdelg 11\tdelg 12\|_{\frac p2}&\lesssim (A(\sigma_2^{-2}
\delta))^2\|f\|_p^2 + \Big(\sum_{\fQ^2}\Big\|
           \sum_{
                 \qqq21,\qqq22\subset 
[\fN](\fQ^2, \qqq11, \qqq12)
}\tdelg 21\tdelg
22\Big\|_{L^\frac p2 (\fQ^2)}^{\frac p2}\Big)^\frac 2p \\
& + \sigma_2^{-C}\sup_{\tau_1, \tau_2, \tau_3}
\max_{\qqq21,\qqq22,\qqq23: trans}\|\tdtfg 21 \tdtfg 22 \tdtfg
23\|_{\frac p3}^\frac23\,,
\end{aligned}
\Ee
where $[\fN](\fQ^2, \qqq11, \qqq12)$ is a subset of
$\fN(\fQ^2, \qqq11, \qqq12)$.
Here, for simplicity we now denote $\tau_{ l_i} f$ by $\tau_i f$ just 
to indicate translation by a vector. Precise value of $l_i$ is not significant in the
overall argument.  To show \eqref{scale22},  
for the first term in the right hand side of  \eqref{scale2} we may repeat the same argument 
as in \eqref{maxp}. In fact, by  \eqref{scatteredd} and  rapid decay of $A_l$\footnote{Note that the sequence is independent of $\fQ^2$.}  combined with H\"older's inequality to summation along $l, l'$, and using Proposition \ref{rescale} and Lemma \ref{vector} we have
\[ \| \max_{\qq2}|\!\tdff{2}{}{}\!|\|_p\lesssim  \sup_{\tau_2}  
\| \max_{\qq2}|\tdel  (\tau_2 f_{\!\qqq{2}{}})| \|_p \lesssim  A(\sigma_1^{-2}\delta) \|f\|_p.  \] 
For the third term of the right hand side of  \eqref{scale2},   thanks to  \eqref{scatteredd}  and rapid decay of $A_l$, it is enough to note that there are  as many as $O(\sigma_2^{-C})$ $\qqq21, \qqq 22, \qqq23: trans$. 

 We combine \eqref{scale22} with \eqref{decomp1} to get \Be\label{2scale-2}
\begin{aligned}
& \|\tdel f\|_p  \lesssim A(\sigma_1^{-2}\delta)\|f\|_p+ \sigma_1^{-C}
A(\sigma_2^{-2} \delta)\|f\|_p
\\
&+ \sigma_1^{-C}\sup_{\tau_1, \tau_2} \max_{\qqq11,\qqq12:
trans}\Big(\sum_{\fQ^2}\Big\| \sum_{\substack{\qqq21\subset
\qqq11,\,\qqq22\subset \qqq12\,:\, \qqq21,\qqq22\subset
[\fN](\fQ^2, \qqq11, \qqq12)}}\tdtaf 21\tdtaf
22\Big\|_{L^\frac p2 (\fQ^2)}^{\frac p2}\Big)^\frac 1p \\
&+
\sigma_1^{-C}\sigma_2^{-C}\sup_{\tau_1, \tau_2, \tau_3} \max_{\qqq21,\qqq22,\qqq23: trans}\|\tdtaf 21 \tdtaf
22 \tdtaf 23\|_{\frac p3}^\frac13.
\end{aligned}
\Ee Here $[\fN](\fQ^2, \qqq11, \qqq12)$ also depends on
$\tau_1,\tau_2$.     We keep decomposing the trilinear transversal part
in order to achieve higher level of multilinearity.    

\subsubsection{From $k$-transversal to $(k+1)$-transversal} Now we proceed inductively. Suppose that we are  given dyadic
cubes $\qqq {k-1} 1,\qqq {k-1}2,\dots, \qqq {k-1}{k}$ of $(k-1)$-th scale  which are
transversal:  
\Be \label{trans-k}
V\!ol(\mbn(\xi_1),\mbn(\xi_2), \dots, 
\mbn(\xi_{k}))\ge c\,\sigma_1\sigma_2\dots\sigma_{k-1} 
\Ee whenever
$\xi_i\in {\qqq {k-1}i}$, $i=1,\dots, k$. As before, we denote by
$\{\qqq k i\}$ the collection of dyadic cubes of sidelength
$2\sigma_k$ contained in $\qqq {k-1} i$ such that \Be\label{k-cube}
\bigcup_{\qqq ki} \qqq ki=\qqq{k-1}i,   \,\,\,i=1,\dots, k, \Ee and
we also denote by $ \{\fq^k\}$ the set $\bigcup_{i=1}^k
\{\qqq{k}i\}$.
 Hence,
\begin{equation}\label{summation} \prod_{i=1}^kT_\delta f_\qqq
{k-1}i = \prod_{i=1}^k \Big(\sum_{\qqq k{i}} T_\delta f_\qqq {k}i\Big)
=\sum_{\qqq k{1}, \dots,\qqq kk }\prod_{i=1}^k \Big( T_\delta f_\qqq {k}i\Big).\end{equation}

Fix  a dyadic cube $\fQ^k$ of sidelength $2M_i$  and let $x_0$ be the center of $\fQ^k$. For  $i=1,
\dots, k$, let us denote by $\qqq {k}{i\ast}\in \{\qqq k i\}$ such that \[\tdfg k {i\ast}{}(x_0)=\max_{\qqq k i}
\tdfg k i {}(x_0)\] and we set, for $i=1, \dots, k,$
\[\Lambda_i^k=\big\{\qqq ki:
\tdfg ki{}(x_0)\ge (\sigma_k)^{kd} \max_{i=1,\dots, k}\tdfg k
{i\ast}{}(x_0)\big\}.\]
 Then, it follows that
\Be
\label{min}
\sum_{(\qqq k1,\dots,\qqq kk)\not\in \prod_{i=1}^k \Lambda^k_i} \prod_{i=1}^k \tdfg ki{}(x_0)\le
\max\tdfg{k}{}{}(x_0).
\Ee

Let $\mbn_1,\dots, \mbn_{k}$ denote the normal vectors $\mbn(\mathbf
c({\qqq k{1\ast}})), \dots, \mbn(\mathbf c({\qqq k{k\ast}})) $,
respectively, and let $\Pi^k=\Pi^k(\fQ^k, \qqq
{k-1} 1,\qqq {k-1}2,\dots, \qqq {k-1}{k})$ be
the $k$-plane spanned by $\mbn_1,\dots, \mbn_{k}$.  Now, for a
sufficiently large constant $C>0$, we
 define
\Be\label{normalset}\fN=\fN(\fQ^k, \qqq
{k-1} 1,\qqq {k-1}2,\dots, \qqq {k-1}{k})= \{\qqq k{}:
\dist(\mbn(\qqq k{}), \Pi^k)\le C\sigma_k\}.
\Ee By
\eqref{trans-k} it follows that  if $\qqq k{i}\not\in \fN$, 
\eqref{trans-k+1} holds
whenever $\xi_1\in \qqq k{1\ast}, \dots, \xi_k\in \qqq k{k\ast} $
and $\xi_{k+1} \in \qqq ki$. Hence, $\qqq k{1\ast}, \dots, \qqq
k{k\ast}, \qqq ki$ are transversal.

We write \Be\label{summation2}\sum_{(\qqq k1,\dots,\qqq kk)\in
\prod_{i=1}^k \Lambda^k_i} \prod_{i=1}^k \tdelg {k}{i}{}=
\sum_{\substack{(\qqq k1,\dots,\qqq kk)\in \prod_{i=1}^k
\Lambda^k_i:\\ \qqq k1,\dots,\qqq k k\in \fN}} \prod_{i=1}^k \tdelg
{k}{i}{}+\sum_{\substack{(\qqq k1,\dots,\qqq kk)\in \prod_{i=1}^k
\Lambda^k_i:\\ \qqq ki\not\in \fN \text{ for some }i}}
 \prod_{i=1}^k \tdelg {k}{i}{}
.\Ee Consider a $k$-tuple $(\qqq k1,\dots,\qqq k k)$ which appears
in the second sum.  There is a $\qqq ki\not\in\fN$. By the same
manipulation as before, we get
 \[\prod_{i=1}^k \tdfg {k}{i}{}(x_0)\le
\sigma_k^{- \frac{d k^2}{k+1}} \prod_{i=1}^k \big(\tdfg
{k}{i\ast}{}(x_0)\big)^\frac k{k+1}\big(\tdfg {k}{i}{}(x_0)\big)^\frac k{k+1}.
\]
Since  $\qqq k{1\ast}, \dots, \qqq k{k\ast}, \qqq ki$ are
transversal, by  Lemma \ref{scattered}  we have,  for $x\in \fQ^k$,
\Be \label{obs}\begin{aligned}
 \Big|\sum_{\substack{(\qqq k1,\dots,\qqq kk)\in \prod_{i=1}^k \Lambda^k_i:\\ \qqq ki\not\in \fN \text{ for some }i}}
 \prod_{i=1}^k \tdelg {k}{i}{}(x)\Big|\lesssim \sigma_k^{-C}
 \sum_{\qqq k1,\dots,\qqq k {k+1}: trans} \prod_{i=1}^{k+1} \big(\tdfg {k}{i}{} (x_0)\big)^{\frac
k{k+1}}.
\end{aligned}
\Ee Combining \eqref{min} and \eqref{obs} with \eqref{summation} and
\eqref{summation2}, and applying Lemma \ref{scattered} yield, for
$x\in \fQ^k$,
\begin{align*}
\Big|\prod_{i=1}^k \tdelg {k-1}{i}(x)\Big|
\lesssim \big(\max_{\qqq k{}}|\!\tdfg{k}{}{}&\!|(x)\big)^k+
\sigma_k^{-C}\sum_{\qqq k1,\dots,\qqq k
{k+1}:\,trans} \prod_{i=1}^{k+1} \big(|\!\tdfg {k}{i}{}\!| (x)\big)^{\frac
k{k+1}}\\
&
+
\Big|\sum_{{
             \substack{
             \qqq k1,\dots,\qqq k k\in \\ [\fN](\fQ^k,\qqq
{k-1} 1,\qqq {k-1}2,\dots, \qqq {k-1}{k})}}}
\prod_{i=1}^k \tdelg {k}{i}(x)\Big|, 
\end{align*}
 where $[\fN](\fQ^k,\qqq {k-1} 1,\qqq {k-1}2,\dots, \qqq
{k-1}{k})$ is a subset of $\fN(\fQ^k, \qqq {k-1} 1,\qqq
{k-1}2,\dots, \qqq {k-1}{k})$.
After taking $p/k$-th power of both sides of inequality, we
integrate on $\mathbb R^d$, and use Lemma
\ref{rescale} and Lemma \ref{vector} along with \eqref{scatteredd}  to
get \Be\label{scale-k1}
\begin{aligned}
  \Big\|\prod_{i=1}^k \tdelg {k-1}{i}(x)&\Big\|_{L^\frac pk}^\frac1k \lesssim
  A(\sigma_k^{-2} \delta)\|f\|_p  + \sigma_k^{-C}\sup_{\tau_1,\dots,\tau_{k+1}}\,\,\max_{\substack{\qqq
k1,\dots,\qqq k {k+1}:trans}}\Big \| \prod_{i=1}^{k+1} \tdeltg
{k}{i}\Big\|_{L^\frac p{k+1}}^{\frac 1{k+1}}
 \\
 &+
\Big(\sum_{\fQ^k}\Big\|\sum_{
             \substack{
             \qqq k1,\dots,\qqq k k\in \\ [\fN](\fQ^k,\qqq
{k-1} 1,\qqq {k-1}2,\dots, \qqq {k-1}{k})}} \prod_{i=1}^k \tdel
f_{\qqq {k}{i}}\Big\|^\frac pk_{L^\frac pk(\fQ^k)}\Big)^\frac 1p\,. 
\\
&
\end{aligned}
\Ee

\subsubsection{Multi-scale decomposition} For $k=2,\dots, d-1$, let us set
\[ \fM^{k}\!f= \sup_{\tau_1,\dots,\tau_k}\, \max_{\qqq {k-1}1, \dots, \qqq {k-1}k: trans}
\Big(\sum_{\fQ^k}\Big\|\sum_{\substack{\qqq{k}i\subset
\qqq{k-1}i:\\\qqq k1,\dots,\qqq k k\in [\fN](\fQ^k)}}
\prod_{i=1}^k \tdel (\tau_i f_{\qqq {k}{i}})\Big\|^\frac pk_{L^\frac
pk(\fQ^k)}\Big)^\frac 1p.\] Here  $[\fN](\fQ^k)$ depends on 
$\tau_1,\dots,\tau_k$, and  $\qqq {k-1}1, \dots, \qqq {k-1}k$, but  $\mathbf n(\fq^k)$, $\fq^k \in  [\fN](\fQ^k)$ is contained in a $k$-plan.
Starting from \eqref{2scale-2} we iteratively apply \eqref{scale-k1}
to the transversal products to get \Be\label{scale-22}
\begin{aligned}
&\|\tdel f\|_p \lesssim
\sum_{k=1}^m \sigma_{k-1}^{-C} A(\sigma_k^{-2}\delta)\|f\|_p 
+ \sum_{k=2}^m\sigma_{k-1}^{-C} \fM^k f\\
&\qquad\qquad\qquad\qquad\qquad + \sigma_l^{-C}\sup_{\tau_1, \dots, \tau_{m+1}}
\max_{\qqq{m}1,\dots\qqq m{m+1}: trans} \Big \| \prod_{i=1}^{m+1}
\tdel \tau_i f_{\qqq {m}{i}}\Big\|_{L^\frac p{m+1}}^{\frac 1{m+1}}.
\end{aligned}
\Ee

\subsection{Proof of Proposition \ref{localfrequency}}
\label{closinginduction}  For given  $\beta>0$, we need to show that $\cA^\beta(s)\le C$ for
$0<s\le 1$ if $p\ge p_\circ(d)$.  Let $\epsilon>0$ be small enough such that $(100d)^{-1}\beta\ge \epsilon$, 
\label{proof1}  and choose small $\epsilon_\circ>0$
and $N=N(\epsilon)$ large enough such that Proposition \ref{multilp}
and Corollary \ref{cor-confined} hold uniformly for $\psi \in
\fge$.

Let $0< s< \delta\le 1$, and let  $\sigma_1, \dots,
\sigma_m$ be dyadic numbers satisfying \eqref{dyadics}. Since
$A(\delta) \le C$ for $\delta\gtrsim 1$ and $s\le
\sigma_k^{-2}\delta$, we see
\Be \label{est1a}
\begin{aligned}
 A(\sigma_k^{-2}\delta) &\le
A(\sigma_k^{-2}\delta)\chi_{(0, 10^{-2}]}(\sigma_k^{-2}\delta)+C \le
(\sigma_k^{-2} \delta)^{-\frac{d-1}2+\frac dp-\beta}
\cA^\beta(s)+C .
\end{aligned}
\Ee
By Proposition \ref{multilp} and  Lemma \ref{vector}  we have, for
$p\ge{2(m+1)}/m$,  \Be\label{l+1} \sup_{\tau_1, \dots,
\tau_{m+1}} \max_{\qqq{m}1,\dots\qqq m{m+1}: trans} \Big \|
\prod_{i=1}^{m+1} \tdel \,\tau_i f_{\qqq {m}{i}}\Big\|_{L^\frac
p{m+1}}^\frac1{m+1} \lesssim
(\sigma_1\cdots\sigma_{m})^{-C_\epsilon}\delta^{-\epsilon}
\delta^{\frac dp-\frac{d-1}2}\|f\|_p\,,\Ee which uniformly holds for
$\psi \in \fge$.

We have two types of estimate for $\fM^k f$. Since $\qqq {k-1}1,
\dots, \qqq {k-1}k$ are already transversal,
\[\Big|\sum_{\substack{\qqq{k}i\subset \qqq{k-1}i:\\\qqq
k1,\dots,\qqq k k\subset [\fN](\fQ^k)}} \prod_{i=1}^k \tdel
(\tau_i f_{\qqq {k}{i}})\Big| \le \sum_{\substack{\qqq k1,\dots,\qqq
k k: trans}} \Big|\prod_{i=1}^k \tdel (\tau_i f_{\qqq
{k}{i}})\Big|.\] Here we  slightly   abuse the definition `$trans$' and $\qqq
k1,\dots,\qqq k k: trans$ means that \eqref{trans-k} holds if
$\xi_i\in \qqq ki$, $i=1,\dots, k$. Since there are as many as
$O(\sigma_k^{-C})$  $(\qqq {k}1, \dots, \qqq {k}k)$ and the above inequality holds regardless of $\fQ^k$, we get
\[\fM^{k}\!f\lesssim  \sigma_k^{-C}
\sup_{\tau_1,\dots,\tau_k} \max_{\qqq {k}1, \dots, \qqq {k}k: trans} \Big\|\prod_{i=1}^k \tdel (\tau_i f_{\qqq
{k}{i}})\Big\|_{\frac pk}^\frac1k.\]
Since $\qqq {k}1, \dots, \qqq {k}k$ are transversal, by Proposition
\ref{multilp} (also see Remark \ref{bound}) and Lemma \ref{vector},  we get, for $p\ge
\frac{2k}{k-1}$,
\[ \Big\|\prod_{i=1}^k \tdel (\tau_i f_{\qqq
{k}{i}})\Big\|_{\frac pk}^\frac1k \lesssim
(\sigma_1\cdots\sigma_{k-1})^{-C_\epsilon} \delta^{\frac
dp-\frac{d-1}2-\epsilon} \prod_{i=1}^k\|\tau_i f_{\qqq
ki}\|_{p}^\frac1k\lesssim \sigma_k^{-C_\epsilon} \delta^{\frac
dp-\frac{d-1}2-\epsilon} \|f\|_{p}\,.\] 
Hence, for $p\ge
\frac{2k}{k-1}$, we have
the  uniform estimate  for  $\psi\in\fge$
\Be\label{klinear}\fM^k\!f\lesssim  \sigma_k^{-C} \delta^{\frac
dp-\frac{d-1}2-\epsilon}\|f\|_{p}. \Ee

On the other hand, fixing $\tau_1, \dots, \tau_k$, $\qqq {k-1}1, \dots, \qqq {k-1}k: trans$, and $\fQ^k$,  we consider the integrals appearing in the definition of    $\fM^k\!f$. 
Let
us write $\fQ^k=\fq(z,1/\sigma_k)$. Using  Corollary \ref{cor-confined},  for $2\le p\le 2k/(k-1),$
we have \Be\label{est4}
\begin{aligned}
\Big\|\!\!\!\sum_{\substack{\qqq{k}i\subset \qqq{k-1}i:\\\qqq
k1,\dots,\qqq k k\in [\fN](\fQ^k)}} \prod_{i=1}^k \tdel
(\tau_i f_{\qqq {k}{i}})\Big\|_{L^\frac pk(\fQ^k)}
\!\!\lesssim
\sigma_{k-1}^{-C_\epsilon}\sigma_{k}^{-\epsilon}
\prod_{i=1}^k \Big\|\Big(\!\!\!\sum_{\qqq {k}{i}\in
[\fN](\fQ^{k})} |\tdel (\tau_i f_{\qqq k{i}})|^2
\Big)^\frac12\rho_{B(z,\frac C{\sigma_k})}\Big\|_p.
\end{aligned}
\Ee
Since
$[\fN](\fQ^k)\subset \fN(\fQ^k,\qqq {k-1}1,$ $\dots,$ $\qqq
{k-1}k)$, it is clear that if $\qqq ki\in [\fN](\fQ^k)$, 
$\qqq ki\subset \mathrm N^{-1}(\Pi)+O(\sigma_k) $ for a $k$-plane
$\Pi$.
Since $\qqq {k-1}1, \dots, \qqq {k-1}k: trans$ and $\qqq ki\subset \qqq{k-1}i$, $i=1,\dots, k$,
 $\sum_{\qqq k1\in \fN(\fQ^k)}\tdel (\tau_1 f_{\qqq
{k}{1}}), \dots,$ $\sum_{\qqq kk\in \fN(\fQ^k)}\tdel (\tau_k f_{\qqq
{k}{k}})$ satisfy the assumptions of Corollary \ref{cor-confined}
(Proposition \ref{confined})  with $\delta=\sigma_k$ and
$\sigma=\sigma_1\cdots\sigma_{k-1}$. Hence, Corollary
\ref{cor-confined} gives \eqref{est4}.

Recalling that $\fq^k_i$ are contained in $C\sigma_k$-neighborhood
of $\Gamma(\psi)$, we see that  $\# \fN(\fQ^{k})$ is
$\lesssim \sigma_{k}^{1-k}$. So, by H\"older's inequality we have 
\[
\Big\|\!\!\!\sum_{\substack{\qqq{k}i\subset \qqq{k-1}i:\\\qqq
k1,\dots,\qqq k k\subset \fN(\fQ^k)}} 
          \prod_{i=1}^k \tdel (\tau_i f_{\qqq {k}{i}})\Big\|^\frac 1k_{L^\frac pk(\fQ^k)}
    \!\lesssim
                      \sigma_{k-1}^{-C}\sigma_{k}^{-\epsilon-p(k-1)(\frac12-\frac1p)}
                       \max_{1\le i\le k}\Big\|\Big(\sum_{\qqq {k}{}} |\tdel (\tau_i f_{\qqq
                          k{}})|^p\Big)^\frac1p\rho_{B(z,\frac C{\sigma_k})} \Big\|_p.
\]
Here we bound $\sigma_1, \dots, \sigma_{k-1}$ with $\sigma_{k-1}$
using \eqref{dyadics} and replace $C_\epsilon$ with a larger
constant $C$, since $\epsilon$ is fixed. By using rapid
decay of $\rho$ we sum the estimates along $\fQ^k$ to get
\Be\label{est6} \fM^k f\lesssim
\sigma_{k-1}^{-C}\sigma_{k}^{-\epsilon-(k-1)(\frac12-\frac1p)}
\sup_{h} \Big\|\Big(\sum_{\qqq {k}{}} |\tdel
(\tau_h f_{\qqq k{}})|^p\Big)^\frac1p \Big\|_p. \Ee
By Proposition \ref{rescale}, Lemma \ref{vector}, and \eqref{est1a} we get, for $2\le p\le 2k/(k-1)$,
\[
\begin{aligned}
\fM^k f 
\lesssim
(\sigma_{k-1}^{-C}\sigma_k^{\beta+\frac{2d-k-1}{2}-\frac{2d-k+1}p}
\delta^{-\frac{d-1}2+\frac dp-\beta}
\cA^\beta(s)+ \sigma_{k}^{-C})\|f\|_{p}.
\end{aligned}
\]
Here we also use $(100d)^{-1}\beta\ge \epsilon$. 
So, if $p\ge \frac{2(2d-k+1)}{2d-k-1}$, 
$\fM^k f \lesssim  ( \sigma_{k-1}^{-C}\sigma_k^{\alpha}
\delta^{-\frac{d-1}2+\frac dp-\beta}
\cA^\beta(s)+ \sigma_{k}^{-C})\|f\|_{p}$
for some $\alpha>0$. Combining this with \eqref{klinear}, we have for
some $\alpha>0$
\[
\begin{aligned}
\fM^k f &\lesssim \Big(\sigma_k^{-C}\delta^{-\frac{d-1}2+\frac
dp-\epsilon}
+\sigma_{k-1}^{-C}\sigma_{k}^{\alpha}\delta^{-\frac{d-1}2+\frac
dp-\beta} \cA^\beta(s)
+\sigma_{k}^{-C}\Big) \|f\|_p
\end{aligned}
\]
provided that $p\ge\min (\frac{2(2d-k+1)}{2d-k-1},
\frac{2k}{k-1})$.

  Since $(100d)^{-1}\beta\ge \epsilon$ and $p_\circ>\frac {2d}{d-1}$, from \eqref{est1a}  we note that   $A(\sigma_k^{-2}\delta)\lesssim  \sigma_{k}^{\alpha} \delta^{-\frac{d-1}2+\frac dp-\beta}\cA^\beta(s)$. Thus, by  \eqref{scale-22},  the above inequality,
\eqref{est1a}, and \eqref{l+1} we obtain
\begin{align}
\label{est7}
\|\tdel f\|_p&\lesssim \sum_{k=1}^{m}\Big( 
\sigma_{k-1}^{-C} \sigma_{k}^{\alpha}\cA^\beta(s)
+\sigma_k^{-C}  \Big)\delta^{-\frac{d-1}2+\frac dp-\beta}
\|f\|_p+ \sigma_{m}^{-C} \delta^{-\frac{d-1}2+\frac dp-\beta}
\|f\|_p
\end{align}
for some $\alpha>0$ provided that
\Be \label{prange} p\ge  \min \Big(\frac{2(2d-k+1)}{2d-k-1},
\frac{2k}{k-1}\Big), \ k= 2, \dots, m,\, \& \  \  p  \ge
\frac{2(m+1)}{m}.\Ee

Since the estimates \eqref{l+1}--\eqref{est6} hold uniformly for $\psi\in
\fge$, so does \eqref{est7}.  Taking sup along $\psi $
and $f$,  we have
\begin{align*}
A(\delta)\le & \Big(\sum_{k=1}^{m} C\sigma_{k-1}^{-C}
\sigma_{k}^{\alpha}\cA^\beta(s) + C\sigma_{m}^{-C}\Big)
\delta^{-\frac{d-1}2+\frac dp-\beta}.
\end{align*}
By multiplying $\delta^{\frac{d-1}2-\frac dp-\beta}$ to both
sides, $\delta^{\frac{d-1}2-\frac dp+\beta}A(\delta)
 \le \sum_{k=1}^{m} C\sigma_{k-1}^{-C}
\sigma_{k}^{\alpha}\cA^\beta(s) + C\sigma_{m}^{-C}.$ This
is valid as long as $s<\delta\le 1$. Hence, taking sup for
$s<\delta\le 1$ yields
\[\cA^\beta(s)
 \le \sum_{k=1}^{m} C\sigma_{k-1}^{-C}
\sigma_{k}^{\alpha}\cA^\beta(s) + C\sigma_{m}^{-C}\] if
\eqref{prange} is satisfied. Therefore, choosing $\sigma_1\ll
\dots\ll \sigma_{m}$, successively,  we can make $\sum_{k=1}^{m}
C\sigma_{k-1}^{-C} \sigma_{k}^{\alpha}\le \frac12.$ This gives the desired 
$\cA^\beta(s)\le C\sigma_m^{-C}$ provided  that
\eqref{prange} holds.

Finally,  we only need to check that the minimum of
\[\mathcal P(m)= \max\Big(\frac{2(m+1)}{m}, \max_{k=2,\dots,m}
\min \Big(\frac{2(2d-k+1)}{2d-k-1}, \frac{2k}{k-1}\Big)\Big),\,\,\,
2\le m\le d-1\] is $p_\circ(d)$ as can  be done by routine
computation. This completes proof. \qed

\begin{rem} The minimum of  $\mathcal P$ is achieved when  $m$ is near ${2d}/3$.   
So, it  doesn't seem that the argument makes use of the full strength of the
multilinear restriction estimates.
\end{rem}

\section{Square function estimates}
In this section we prove Theorem \ref{mainsquare}. We firstly obtain
multi-(sub)linear square function estimates which are vector valued extensions of  multilinear restriction
estimates. Then, we  modify  the argument in Section \ref{proof1} to
obtain the sharp square function estimate from these multilinear
estimates. Although basic strategy here is similar to the one in the
previous section, due to the additional integration in $t$ we need to
handle a family of surfaces.
This argument in this section is very much in parallel with that of the previous section.

\subsection{One parameter family of elliptic functions} As before, 
for
$0<\epsilon_\circ\ll 1/2$ and an integer  $N\ge 100d$, we  denote by $\fgee$ the class of smooth functions
defined on $I^{d-1}\times I$ which satisfy
 the following:
\begin{align}
\label{ell1}&\| \psi-\psi_\circ-t\|_{C^N(I^{d-1}\times I)} \le
\epsilon_\circ.
\end{align}
This clearly implies that, for all $(x,t)\in I^{d-1}\times I$, 
\begin{align}
\label{ell2} \partial_t\psi(x,t) \in [1-\epsilon_\circ, 1+\epsilon_\circ].
\end{align}
For $\psi\in \fgee$ and $z_0=(\zeta_0,t_0)\in \frac12 I^d$,  define
\begin{align*}
\psi_{z_0}^\epsilon(\zeta,t)=&\epsilon^{-2}
\Big(\psi\big(\zeta_0+\epsilon\,\mathcal H_{z_0}^\psi\zeta,\,
 t_0+\frac{\epsilon^2t}{\partial_t\psi(z_0)}\big)
- \psi(z_0)- \epsilon\,\nabla_\zeta\psi(z_0)\mathcal
H_{z_0}^\psi\zeta\Big),
\end{align*}
where $\mathcal H_{z_0}^\psi
=(\sqrt{H(\psi(\cdot,t_0))(\zet_0)}\,\,)^{-1}.$ Then we have the
following.

\begin{lem}\label{classsquare} Let $z_0\in \frac12 I^d$
and $\psi\in \fgee$. There  is a $\kappa=\kappa(\epsilon_\circ,
N)>0$, independent of $\psi, \zeta_0,\, t_0$,
 such that $\psi_{z_0}^\epsilon$
is contained in $\fgee$ if $0<\epsilon \le \kappa$.
\end{lem}

\begin{proof}
It is sufficient to show that $|\partial_\zeta^\alpha
\partial_t^\beta (\psi_{z_0}^\epsilon(\zeta,t)-\psi_\circ(\zeta)-t)|\le
C\epsilon$, with $C$ independent of $\psi\in \fgee$, if
$|\alpha|+\beta\le N$ and $(\zeta,t)\in I^d$.

Let $0<\epsilon\le 1/4$. If $(\zeta,t)\in I^d$ and
$|\alpha|+2\beta>2$, trivially $|\partial_\zeta^\alpha
\partial_t^\beta (\psi_{z_0}^\epsilon(\zeta,t)-\psi_\circ(\zeta,t)-t)|\le
C\epsilon$ because $z_0=(\zeta_0,t_0)\in \frac12 I^d$. Thus, it is
sufficient to consider the cases $\beta=1, |\alpha|=0$; $\beta=0$,
$0\le |\alpha|\le 2$.
The first case is easy to handle. Indeed, from Taylor's theorem and  \eqref{ell2} 
 $\partial_t (\psi_{z_0}^\epsilon(\zeta,t)-\psi_\circ-t)=
(\partial_t\psi(z_0))^{-1}\big(\partial_t
\psi(\zeta_0+\epsilon\,\mathcal H_{z_0}^\psi\zeta,\,
 t_0+\frac{\epsilon^2t}{\partial_t\psi(z_0)})-\partial_t\psi(z_0)\big)=O(\epsilon).$
.

To handle the second case,  we consider  Taylor's expansion of
$\psi$ in $t$ with integral remainder:
\[\psi(\zeta,t)=\psi(\zeta,t_0)+\partial_t \psi(\zeta,t_0)(t-t_0) +  R_1(\zeta, t), 
\]
where  $R_1(\zeta, t)=(t-t_0)^2\int_{0}^{1}  (1-s) \partial_t^2\psi(\zeta, (t-t_0)s+t_0) ds.$  
The change of variables $t\to t_0+\epsilon^2
(\partial_t\psi(z_0))^{-1} t$, $\zeta\to \zeta_0+\epsilon\,\mathcal
H_{z_0}^\psi\zeta$ gives
\begin{align*}
\psi\Big(\zeta_0+\epsilon\,\mathcal H_{z_0}^\psi\zeta,\,
 t_0+\frac{\epsilon^2t}{\partial_t\psi(z_0)}\Big)
&= \epsilon^2\psi(\cdot,t_0)_{\zeta_0}^\epsilon (\zeta)+ \psi(z_0)+
 \epsilon\nabla_\zeta\psi(z_0)\,\mathcal H_{z_0}^\psi\zeta
 \\
&\qquad+  \frac{\epsilon^2\partial_t \psi\big(\zeta_0+\epsilon\,\mathcal H_{z_0}^\psi\zeta,t_0\big)}{\partial_t\psi(z_0)} t
+  \widetilde R(\zeta, t) 
\end{align*}
where $\psi(\cdot,t_0)_{\zeta_0}^\epsilon$ is defined by
\eqref{normalization} and  $\widetilde R(\zeta, t)=R_1(\zeta_0+\epsilon\,\mathcal
H_{z_0}^\psi\zeta, t_0+\epsilon^2
(\partial_t\psi(z_0))^{-1} t)$.
Hence, it follows that
\begin{align*}
\psi_{z_0}^\epsilon -\psi_\circ -t
=\psi(\cdot,t_0)_{\zeta_0}^\epsilon (\zeta)-\psi_\circ
+\frac{\partial_t
\psi(\zeta_0+\epsilon\,\mathcal H_{z_0}^\psi\zeta,t_0)-
\partial_t\psi(z_0)}{\partial_t\psi(z_0)} t+ \epsilon^{-2} \widetilde R(\zeta, t). 
\end{align*}
 Since $\psi(\cdot,t_0)-t_0\in \fge$ and
$(\psi(\cdot,t_0)-t_0)_{\zeta_0}^\epsilon=\psi(\cdot,t_0)_{\zeta_0}^\epsilon$,
$|\partial_\zeta^\alpha(\psi(\cdot,t_0)_{\zeta_0}^\epsilon-\psi_\circ)|\le
C\epsilon$ on $I^d$ for $|\alpha|=0,1,2$ (similarly to the proof of  Lemma
\ref{normal}). By \eqref{ell2} and mean value theorem   we also have
$({\partial_t\psi(z_0)})^{-1}$ $\partial_\zeta^\alpha({\partial_t
\psi(\zeta_0+\epsilon\,\mathcal H_{z_0}^\psi\zeta,t_0)-
\partial_t\psi(z_0)})t=O(\epsilon)$ in $C^N(I^{d-1})$  for $|\alpha|=0,1,2$.  Note that 
\[  \epsilon^{-2}\widetilde R(\zeta, t)=  \frac{\epsilon^2 t^2}{(\partial_t\psi(z_0))^2} \int_0^{1} (1-s) 
\partial_t^2\psi\big(\zeta_0+\epsilon\,\mathcal
H_{z_0}^\psi\zeta, \epsilon^2
(\partial_t\psi(z_0))^{-1} t s+t_0\big)   ds\,.\] 
Thus, again by  \eqref{ell2}   it is easy to see that 
$\partial_\zeta^\alpha ( \epsilon^{-2}\widetilde R)=O(\epsilon^{2+|\alpha|})$ for any $\alpha$. Therefore, combining the all together  we have  $|\partial_\zeta^\alpha(\psi_{z_0}^\epsilon(\cdot,
t)-\psi_\circ-t)|\le C\epsilon$ on $I^{d-1}$ for $|\alpha|=0,1,2$.
\end{proof}

\subsection{Square function with localized frequency}
Abusing the conventional notation we denote by $m(D)f$ the
multiplier operator given by $\widehat{m(D)f}(\xi)$ $=m(\xi)
\widehat f(\xi)$, and we also write $D=(D',D_d)$ where $D'$, $D_d$
correspond to the frequency variables $\zeta,$ $\tau$, respectively.

In order to show \eqref{square}, by Littlewood-Paley decomposition,
scaling, and further finite decompositions, it is sufficient to show
\[\Big\|\Big(\int_{1-\eps^2}^{1+\eps^2} \Big|\frac{\partial}{\partial t} \cR^\alpha_t f(x)
\Big|^2 dt \Big)^{1/2}\Big\|_p\le C \|f\|_p\] for some small $\eps>0$. And
by decomposing $\widehat f$ which may now be assumed to be supported in $S^{d-1}+O(\varepsilon^2)$ and rotation  we may assume $\widehat f$ is
supported in $B(-e_d,c\eps^2)$ with some $c>0$. Hence, by discarding harmless smooth
multiplier the matter reduces to showing
\[
  \big\|\|(D_d+\sqrt{t^2-|D'|^2}\,\,)_+^{\alpha-1} f\|_{L^2_t(1-\eps^2, 1+\eps^2)}\big\|_p\le C\|f\|_p.
  \]
  By changing variables in frequency domain, $D_d\to D_d+1$,
$(D',D_d)\to (\eps D', \eps^2 D_d)$ and
 $t\to \eps^2t+1$, this is equivalent to
\Be
\label{modified-square} 
\big\|\|(D_d-\psi_{br}(D',t)))_+^{\alpha-1}
\chi_\circ(D) f\|_{L^2_t(I)}\big\|_p\le C\|f\|_p
\Ee where
$\psi_{br}(\zeta,
t)=\eps^{-2}(1-\sqrt{1+2\eps^2t+\eps^4t^2-\eps^2|\zeta|^2})$ and $
\chi_\circ$ is a smooth function supported in a small neighborhood
of the origin. Clearly, $\psi_{br}$ satisfies \eqref{ell1} with
$\epsilon_\circ=C\eps^2$ for some $C>0$.  Consequently, we are led to consider general
$\psi\in\fgee$ rather than  the specific  $\psi_{br}$.

Let us define  the class  $\cE(N)$ of smooth functions 
by setting
\[\cE(N)=\big\{\eta\in C^\infty(I^d\times I): \|\eta\|_{C^N(I^d\times I)}\le 1,
 1/2\le \eta\le 1 \big\}.\]
Let $\psi\in \fgee$ and $\eta\in \cE(N)$. For $0<\delta$ and $f$
with $\widehat f$ supported in $\frac12 I^d$, we define
$S_\delta=S_\delta ({\psi},\eta)$  by \Be\label{srdef} S_\delta
f(x)=\Big\|\phi\Big(\frac{\eta(D,t)(D_d-\psi(D',t))}{\delta}\Big)f\Lpi
2. \Ee
Compared to $\psi$, the role of $\eta$ is less significant but this
enables us to handle more general square functions (in
particular, see {\it Remark \ref{spherical})}. By dyadic
decomposition away from the singularity \eqref{modified-square}  is
reduced to obtaining the sharp bound
 \Be \label{square1}
\|S_\delta f\|_p\le C \delta^{\frac dp-\frac{d-2}{2}-\epsilon}
\|f\|_p, \,\,\epsilon>0, \Ee when $\widehat f\,$ is supported in a
small neighborhood of the origin. This is currently verified for
$p\ge \frac{2(d+2)}{d}$  (\cite{lrs}) by making use of bilinear
restriction estimate for the elliptic surfaces. The following is our
main result concerning the  estimate \eqref{square1}.

\begin{prop}\label{localf2} Let $p_s=p_s(d)$ be given by
\eqref{ps} and $\supp\,
\widehat f\subset \frac12 I^d$. If $p\ge \min ( p_s(d), \frac{2(d+2)}d)$ and $\epsilon_\circ$ is sufficiently small,
for $\epsilon>0$ there is an
$N=N(\epsilon)$ such that \eqref{square1} holds uniformly for
 $\psi\in\fgee$, $\eta\in \cE(N)$.
\end{prop}

\subsubsection*{Proof of Theorem \ref{mainsquare}} 
By choosing small $\eps>0$ in the above, we can make $\psi_{br}$ be in
$\fgee$ for any $\epsilon_0$ and $N$. Hence, Proposition
\ref{localf2} gives \eqref{square1} for any $\epsilon>0$ if  $p\ge \min ( p_s(d), \frac{2(d+2)}d)$. Hence, dyadic decomposition of the multiplier operator  in \eqref{modified-square} and using \eqref{square1}  followed by summation along dyadic
pieces gives  \eqref{modified-square} for $\alpha>d/2-d/p$. This proves  Theorem \ref{mainsquare}.\qed

\begin{rem}\label{spherical} As has been shown before, for the proof of Theorem \ref{mainsquare} it suffices to consider an operator  which is defined without $\eta$  but by allowing $\eta$ in \eqref{srdef} we can
handle the square function estimates for the operator
$f\to\phi\big(\frac{1-|D|/t}{\delta}\big)f$ which is closely related to smoothing estimates for the solutions to the 
Schr\"odinger and wave equations (for example, see \cite{lrs}).  In fact,  Proposition
\ref{localf2} implies, for $\epsilon>0$, \Be \label{sphericalsq}
\Big\|\Big(\int_{1/2}^{2} \Big|\phi\Big(\frac{1-|D|/t}{\delta}\Big)
f\Big|^2 dt\Big)^\frac12 \Big\|_{p} \le \delta^{\frac d2-\frac
dp-\epsilon} C\|f\|_p \Ee if $p\ge p_s(d)$. Indeed,   by finite
decompositions, rotation and scaling, as before, it is sufficient to
consider time average over the interval $I_\eps=(1-\eps^2,
1+\eps^2)$ and we may assume that $\widehat f$ is supported in
$B(-e_d,c\eps^2)$. Writing $1-|\xi|/t=t^{-2}(t+|\xi|)^{-1}(\tau-{\sqrt{t^2-|\zeta|^2}}{})(\tau+\sqrt{t^2-|\zeta|^2})$
for $\xi\in B(-e_d,c\eps^2)$,  the same change of variables $D_d\to
D_d+1$, $(D',D_d)\to (\eps D', \eps^2 D_d)$ and
 $t\to \eps^2t+1$ transforms $\phi(\frac{1-|\xi|/t}\delta)$ to
$\phi(\frac{\eta(\xi,t)(\tau-\psi_{br})}{\eps^{-2}\delta/2})$ with a smooth
$\eta$ which satisfies $ \eta\in (1-c\eps/2, 1+c\eps/2)$. Hence, we
now apply Proposition \ref{localf2} with sufficiently small $\eps$
to get \eqref{sphericalsq}.
\end{rem}

Similarly as before, in order to control $L^p$ norm of $S_\delta$ we define $B(\delta)=B_p(\delta)$ by
\begin{align*}
 B(\delta)\equiv {\sup} \Big\{ \| {S_\delta}
({\psi},\eta) f\|_{L^p}: \psi\in\fgee,\, \eta\in \cE(N),\,
 \|f\|_p\le 1,\, \supp \widehat f\subset \frac12 I^d\Big\}. \end{align*}
 %
As before, using Lemma \ref{kerneldelta1} it is easy to see that
$B(\delta)\le C$ if $\delta\ge 1$, and  $B(\delta)\le
C\delta^{-c}$ for some $c>0$, otherwise (for example, see the paragraph below Proposition \ref{multisq}). We also define for $\beta>0$ and $\delta\in (0,1)$, 
\[ \mathcal B^\beta(\delta)= \cB_{p}^\beta(\delta)\equiv \sup_{\delta <s \le 1}
s^{ \frac{d-2}2-\frac dp +\beta}\,\,
 B_{p}(s).\]
 Thus, Theorem \ref{mainsquare} follows if we show $\mathcal
B^\beta(\delta)\le C$ for  any $\beta>0$. As observed in the previous section  the bound for $\sdel f$ improves if the Fourier transform of $f$ is contained in a set of smaller diameter. 
The following plays a  crucial role in  the induction argument  (see Section \ref{pf-sq}). 

\begin{prop}\label{rescalesquare}
Let $0< \delta\ll 1$,
$\psi\in \overline{\fC}(\epsilon_\circ,N)$, and $\eta\in \cE(N)$.
Suppose that $\widehat f$ is supported in $\mathfrak q(a, \eps)$, $
10\sqrt\delta\le \eps\le 1/2$, and $a\in \frac12 I^d$. Then, if
$\epsilon_\circ>0$ is small enough, there is  a $ \kappa=
\kappa(\epsilon_\circ, N)$ such that \Be\label{rescaless}
\|{S_\delta} ({\psi},\eta) f\|_{p}\le
C\eps^{\frac1p+\frac12}B_{p}(\eps^{-2} \delta) \|f\|_p\Ee holds with
$C$, independent of $\psi$, and $\eps$, whenever $\eps\le \kappa$.
\end{prop}

\begin{proof} By breaking the support of $\widehat f$ into a finite number of dyadic cubes,  we may assume
that $\widehat f$ is supported in $\fq(a,\nu\eps)$ for a small
constant $\nu>0$ satisfying  $\nu^2 d^2\in [2^{-5},2^{-4}) $. This only increases the bound by a constant multiple.  Since $\widehat f$\, is supported in $\mathfrak q(a,
\nu\eps)$ and $a=(a',a_d)\in \frac12I^d$,   from \eqref{ell2} and the
fact that $1/2\le \eta\le 1 $ it is clear that
$\phi\big(\frac{\eta(D,t)(D_d-\psi(D',t))}{\delta}\big)f$
$\not\equiv 0$ for $t$ contained in an interval $[\alpha, \beta]$ of
length $\lesssim \nu\eps$ because
$\phi\big(\frac{\eta(\xi,t)(\tau-\psi(\zeta,t))}{\delta}\big)$ is
supported in $O(\delta)$-neighborhood of $\tau=\psi(\zeta,t)$.

 Let $\alpha=t_0<t_1<\ldots< t_l=\beta$, $l\le O(\eps^{-1})$, such
that $t_{k+1}-t_{k}\le \nu^2\eps^2$. Since $\delta\le
10^{-2}\eps^2$, by \eqref{ell1} and \eqref{ell2} it follows that if $t\in
[{t_k}, t_{k+1}]$, then 
$\phi\big(\frac{\eta(\xi,t)(\tau-\psi(\zeta,t))}{\delta}\big)\widehat
f(\xi)$ is supported in the parallelepiped
\[\cP_k=\Big\{(\zeta,\tau): \max_{i=1,\dots, d-1}|\zeta_i-a'_i|<\nu\eps, 
|\tau-\psi(a',{t_k})-\nabla_\zeta\psi(a',{t_k})(\zeta-a')|
\le   2d^2\nu^2\eps^2 \Big\}.\] 
This follows from  Taylor's theorem since  $\psi\in \fgee$.  
By  \eqref{ell2} it is easy to see that
$\{\cP_k\}_{k=1}^l$ are overlapping boundedly. In fact, 
$\phi\big(\frac{\eta(\xi,t)(\tau-\psi(\zeta,t))}{\delta}\big)\widehat
f(\xi)$, $t\in [{t_k},
t_{k+1}]$,  is supported in
\[\widetilde \cP_k=\{\xi\in \mathfrak q(a, c\eps): |\tau-\psi(\zeta,t_k))|\le C\eps^2 \}, \,k=0,\dots, l-1, \]
with $C\ge 3d^2\nu^2\eps^2$  and  $\{\widetilde \cP_k\}$  are boundedly overlapping because of \eqref{ell2}, and by 
Taylor's expansion  it is easy to see that $\cP_k\subset
\widetilde\cP_k$ because the 2nd remainder is uniformly $O(\eps^2)$
for $\psi\in \fgee$.

Let $\varphi$ be a smooth function supported in $2I^d$ and
$\varphi=1$ on $I^d$. Let $L_{\cP_k}$ be the affine map which
bijectively maps $\cP_k$ to $I^d$, and set
$\varphi_{\cP_k}=\varphi(L_{\cP_k}\cdot)$ so that $\varphi_{\cP_k}$
vanishes outside of  $2\cP_k$ and equals $1$ on $\cP_k$. Here
$2\cP_k$ denotes the parallelepiped which is given by dilating
$\cP_k$ twice from the center of $\cP_k$. Then we have
\begin{align*}
({S_\delta}f(x))^2= \sum_k \int_{I_k} \Big|
\phi\Big(\frac{\eta(D,t)(D_d-\psi(D',t))}{\delta}\Big)
\varphi_{\cP_k}(D)f(x)\Big|^2 dt.
\end{align*}
Since $p\ge 2$, by H\"older's inequality it follows that
\begin{align*}
{S_\delta}f(x)\le C\eps^{\frac1p-\frac12}\Big(\sum_k
\Big\|\phi\Big(\frac{\eta(D,t)(D_d-\psi(D',t))}{\delta}\Big)
\varphi_{\cP_k}(D)f(x)\Big\|_{L^2_t(I_k)}^p\Big)^\frac1p.
\end{align*}
Hence it is sufficient to show that
\Be \label{proj}
 \Big\|\Big\|\phi\Big(\frac{\eta(D,t)(D_d-\psi(D',t))}{\delta}\Big)
\varphi_{\cP_k}(D)f\Big\|_{L^2_t(I_k)} \Big\|_p\le C \eps
B_{p}(\eps^{-2} \delta) \|\varphi_{\cP_k}(D)f\|_p. \Ee Because
$(\sum_k \|\varphi_{\cP_k}(D)f\|_p^p)^\frac1p\le C \|f\|_p$ for
$2\le p\le \infty$. This follows by interpolation between the
estimates for $p=2$ and $p=\infty$. The first is an easy consequence
of Plancherel's theorem because $\{2\cP_k\}$ are boundedly
overlapping and the latter is clear since $\mathcal
F^{-1}(\phi_{\cP_k})\in L^1$ uniformly.

Now we make the change of variables
\[
t\to \eps^2 (\partial_t\psi(a',{t_k}))^{-1}t+{t_k}, \quad
\xi \to L(\xi)=(L'(\xi), L_d(\xi)),\]
where
\begin{align*}
&L'(\xi)= \eps\mathcal H^\psi_{(a',{t_k})}\zeta+a',\quad L_d(\xi) =
\eps^2 \tau+\psi(a',{t_k})+\eps\nabla_\zeta\psi(a',{t_k})\mathcal
H^\psi_{(a',{t_k})}\zeta,
\end{align*}
and
\[\eps^{2}x_d\to  x_d, \quad  \eps\mathcal H^\psi_{(a',{t_k})}(x'+ x_d\nabla_\zeta \psi(a',{t_k}))\to x'.\]
Then, \eqref{proj} follows if we show
\begin{align*}
 \Big\|\Big\|\phi\Big(\frac{\eta(L(D),t)(D_d-\psi_{a',{t_k}}^\eps(D',t))}{\eps^{-2}\delta}\Big)
f\Big\|_{L^r_t(0, 2\nu^2)} \Big\|_p\le C B_{p}(\eps^{-2} \delta) \|f\|_p
\end{align*}
when the support  $\widehat f$ is contained  in $L^{-1}(2\mathcal
P_k)$. Clearly, $\eta(L(\xi),t) \in \cE(N)$ and  $L^{-1}(2\mathcal P_k)$ is contained in the set  
$\{(\zeta, \tau): |\zeta|\le 4\nu , |\tau|\le 8d^2\nu^2\} \subset \frac12 I^d$. From Lemma \ref{classsquare} there 
exists $\kappa>0$ such that $\psi_{a',{t_k}}^\eps\in
\fgee$ if $0<\eps\le \kappa$. Hence, 
using the definition of $B_{p}(\delta)$ we get the desired
inequality for $\eps\le \kappa$.
\end{proof}

\subsection{Multi-(sub)linear square function estimates}
Let $\psi\in \fgee$ and   set  
\Be
\label{gamma}\Gamma^t=\Gamma^t(\psi):=\big\{(\zeta,\psi(\zeta,t)):
\zeta\in \frac12 I^d\, \big\}. 
\Ee  As before we denote by
$\Gamma^{t}(\delta)$ the $\delta$-neighborhood
$\Gamma^{t}+O(\delta)$. Clearly, from \eqref{ell2} it follows that,
for $\delta>0$, 
\Be\label{deltanbd}
\Gamma^{t}(\delta)\cap\Gamma^{\,s}(\delta)=\emptyset, \,\,\text{ if
} |t-s|\ge C\delta \Ee for some $C>0$. We also denote by $\mathrm
N^t$ the (upward) normal map from the surface $\Gamma^t$ to $\mathbb
S^{d-1}$.

\begin{defn}[Normal vector field $\mbn=\mbn(\psi)$] \label{nvector}
The map $(\zeta, t)\to (\zeta, \psi(\zeta,t))$ is clearly one to one
and we may assume that the image of this map contains $I^d$ by extending  $\psi(\zeta,t)$ to a larger set $I^{d-1}\times CI$, while \eqref{ell1} is satisfied. Hence,
 for each $\xi=(\zeta,\tau)\in I^d$ there is a unique $t$
such that $\xi=(\zeta, \psi(\zeta,t))$.  Then we define $\mathbf
n(\xi)$ to be the normal vector to $\Gamma^{t}$ at $\xi$, which forms
a vector field on $I^d$. \end{defn}

A natural attempt for multilinear generalization of  $S_\delta$  is  to consider $\prod_{i=1}^kS_\delta f_i$ under transversality condition between $\supp f_i$.  But,  induction on scale argument  does not work  well with this naive generalization and  it  doesn't seem easy to obtain  the sharp multilinear square function estimates directly.   We get around the difficulty  by considering a vector valued extension in which we discard the exact structure of the operator $S_\delta$.  As is clearly seen in its proof, the estimate in Proposition
\ref{multisq} is not limited to the surfaces given by $\psi\in
\fgee$ but it holds for more general class of surfaces as long as
the transversality  is satisfied.

\begin{prop}\label{multisq}  Let  $2\le k\le d$ be an integer and $0<\sigma\ll 1$, and let $\Gamma^{t}$ be given by $\psi\in \fgee$, and the functions  $G_i$, $1\le i\le k$, be defined on $\mathbb R^d\times  I$.
Suppose that, for each $t\in I$,  $G_1(\cdot,t), \dots, G_k(\cdot,t)$ satisfy that, for $0<\delta\ll \sigma$,
\Be\label{deltanbdg} \supp\,\widehat G_i(\cdot, t)\subset
\Gamma^t(\delta), \quad t\in I,\Ee and suppose that \Be
\label{transverse1} V\!ol(\mbn(\xi_1), \mbn(\xi_2), \dots,
\mbn(\xi_k))\gtrsim  \sigma, \Ee
whenever $\xi_i\in \supp\, \widehat G_i(\cdot, t)+O(\delta)$ for some $t\in I$. %
 Then, if $p\ge
2k/(k-1)$ and $\epsilon_\circ>0$ is small enough, for $\epsilon>0$  there  is an  $N=N(\epsilon)$
such that
 \Be \label{multil2}\Big\| \prod_{i=1}^k
\|G_i\|_{L^2_t(I)}\Big\|_{L^\frac pk(B(x,\delta^{-1}))}\le
C\sigma^{-C_\epsilon}\delta^{-\epsilon} \prod_{i=1}^k
\big(\delta^\frac12 \|G_i\|_{L^2_{x,t}}\big)\Ee holds with $C,
C_\epsilon$, independent of $\psi$.
\end{prop}

Without being concerned about the optimal $\alpha$ for a while, we
first observe that,  for $p\ge 2$, there is an $\alpha$ such that
\Be \label{trivial} \Big\|\|G_i\|_{L^2_t(I)}\Big\|_{L^p(\mathbb
R^d)}\le C\delta^{-\alpha} \|G_i\|_{L^2_{x,t}} \Ee holds uniformly
if $\psi\in \fgee$ and  $N$ is large enough ($N\ge 100d$). (It is
enough to keep $\|\psi\|_{C^N(I^d)}$ uniformly bounded.)
To see this, let $\varphi$ be a smooth function supported in $2I$
and $\varphi=1$ on $I$,
and we set $K_\delta^t=\mathcal F^{-1}(\varphi\big(\frac{\tau-\psi(\zeta,t)}{C\delta}\big)\widetilde\chi(\xi))$.
Then, by Lemma \ref{kerneldelta1} $|K_\delta^t(x)|\le C\delta
\fK_M(x)$ for a large $M$ with $C$, depending only on
$\|\psi\|_{C^N(I^d)}$. Since $\supp\,\mathcal F(G_i(\cdot,
t))\subset \Gamma^t(\delta)$, 
$G_i(\cdot, t) =K^{t}_\delta\ast G_i(\cdot, t)$. So, $|G_i(x, t)|
\le C\delta\fK_{M}\ast |G_i(\cdot, t)|,$ $t\in I$ and 
by Minkowski's inequality  we get \Be \label{kernel}
\|G_i(x, t)\|_{L^2_t(I)} \le C\delta\fK_M\ast (\|G_i(\cdot,
t)\|_{L^2_t(I)})(x). \Ee Young's convolution inequality gives the
inequality \eqref{trivial}, namely with $\alpha={d-1}$, if taking sufficiently large $M$.

\begin{proof}[Proof of Proposition \ref{multisq}]
Since $\mathcal F(G_i(\cdot,
t))=\varphi\big(\frac{\tau-\psi(\zeta,t)}{C\delta}\big)\widetilde\chi(\xi)
\mathcal F(G_i(\cdot, t))$, by Schwarz's inequality and Plancherel's
theorem, $|G_i(x,t)|\lesssim \delta^\frac12 \|G_i(\cdot,t)\|_{2}$.
So, this gives  \eqref{multil2}  for $p=\infty$. Thus, by
interpolation it is sufficient to show \eqref{multil2} with
$p=\frac{2k}{k-1}$.

Let us set $R=\delta^{-1}$ and we may set $x=0$. Following the same argument as in  
the proof of Proposition \ref{confined} we start with the assumption
  that, for $0<\delta\ll \sigma$,
\Be\label{assume1} \Big\| \prod_{i=1}^k
\|G_i\|_{L^2_t(I)}\Big\|_{L^\frac 2{k-1}(B(0,R))}\lesssim  R^\alpha
R^{-\frac k2} \prod_{i=1}^k \|G_i\|_{L^2_{x,t}}\Ee holds uniformly
$\psi\in \fgee$ whenever  \eqref{deltanbdg} and \eqref{transverse1}
are satisfied. By \eqref{trivial} and H\"older's inequality, this is
true for a large $\alpha>0$.
Hence, it is sufficient to show  \eqref{assume1} implies that for
$\eps>0$ there is an $N=N(\eps)$ such that, for some $\kappa>0$,
 \Be \label{assume2} \Big \| \prod_{i=1}^k
\|G_i\|_{L^2_t(I)}\Big\|_{L^\frac 2{k-1}(B(0,R))}\lesssim
C_\epsilon\sigma^{-\kappa} R^{\frac\alpha 2+c\eps}
R^{-\frac{k}2}\prod_{i=1}^k \|G_{i}\|_{L^2_{x,t}} \Ee holds
uniformly for $\psi\in \fgee$.  Then, iterating this implication from \eqref{assume1} to \eqref{assume2} gives the desired inequality. (See the paragraph below \eqref{l2222}.)

Since $\widehat \rho_{B(z,\sqrt R)}$ is supported in a ball of
radius $\sim R^{-\frac12}$,  the Fourier transform of
$\rho_{B(z,\sqrt R)} G_i(\cdot, t)$ is contained in
$\Gamma^{t}+O(R^{-1/2})$ for each $t$ and \eqref{transverse1} holds
with $\delta=R^{-\frac12}$ since $\delta\ll \sigma$. Hence, by the assumption \eqref{assume1},
it follows that
\Be\label{half0} \Big \| \prod_{i=1}^k \|\rho_{B(z,\sqrt R)}
G_i\|_{L^2_t(I)}\Big \|_{L^{\frac 2{k-1}}(B(z,\sqrt R))}\le C
R^\frac\alpha 2 R^{-\frac k4} \prod_{i=1}^k \|\rho_{B(z,\sqrt R)}
G_i\|_{L^2_{x,t}}. \Ee

 We now decompose  $G_i(\cdot, t)$ into $\{G_{i,\mbq}(\cdot,
 t)\}$ which is defined by
\Be \label{decomp2}\mathcal F({G_{i,\mbq}(\cdot, t)})=\chi_\mbq
\mathcal F( {G_{i}(\cdot,t)}).\Ee Here $\{\mbq\}$ are the dyadic
cubes of sidelength $l$, $R^{-1/2}<l\le 2 R^{-1/2}$,  which we already used in
the Proof of Proposition \ref{confined}. We write 
\[ G_i(x,t)=\sum_{\mbq} G_{i,\mbq}(x,
 t).\]
In what follows we may assume $G_{i,\mbq}\neq 0$. By
\eqref{deltanbdg} it follows that, for each $t$, the cubes
$\{\mbq\}$ appearing in the sum are contained in
$\Gamma^{t}(R^{-\frac12})$ because $G_{i,\mbq}(\cdot, t)=0$, otherwise.
We also note from \eqref{ell2} that there is an interval
$I_{i,\mbq}$ of length $C R^{-1/2}$ such that $G_{i,\mbq}(\cdot,
t)=0$ if $t\not\in I_{i,\mbq}$.  Hence we may multiply the
characteristic function of $\chi_{I_{i,\mbq}}$ so that \Be
\label{interval} G_{i,\mbq}=G_{i,\mbq}(\cdot, t)
\chi_{I_{i,\mbq}}(t).\Ee Since the Fourier supports of
$\{\rho_{B(z,\sqrt R)} G_{i,\mbq}(\cdot, t)\}$ are boundedly
overlapping,   by Plancherel's theorem it follows that
\Be\label{l2ortho} \prod_{i=1}^k \|\rho_{B(z,\sqrt R)}
G_i\|_{L^2_{x,t}} \le C \prod_{i=1}^k \Big\|\Big(\sum_{\mbq}|\rho_{B(z,\sqrt R)}
G_{i,\mbq}|^2\Big)^\frac12\Big\|_{L^2_{x,t}}.\Ee
Combining this with \eqref{half0} we have 
\[
\Big \| \prod_{i=1}^k \|\rho_{B(z,\sqrt R)} G_i\|_{L^2_t(I)}\Big \|_{L^
{\frac 2{k-1}}}\le C R^\frac\alpha 2 R^{-\frac k4} \prod_{i=1}^k
\Big\|\Big(\sum_{\mbq}|\rho_{B(z,\sqrt R)}
G_{i,\mbq}|^2\Big)^\frac12\Big\|_{L^2_{x,t}}. 
\] 
Since $\rho_{B(z,\sqrt R)}$ is rapidly decaying outside of $\ball
z{\sqrt R}$, we have for any large $M>0$ 
\Be \label{half2}
\begin{aligned}
\Big \| \prod_{i=1}^k \|\rho_{B(z,\sqrt R)} G_i\|_{L^2_t(I)}\Big
\|_{L^ {\frac 2{k-1}}}
\\
\lesssim  R^{\frac\alpha 2-\frac k4} 
\prod_{i=1}^k
\Big\|&\chi_{B(z,R^{\frac12+\epsilon})}\Big(\sum_{\mbq}|
G_{i,\mbq}|^2\Big)^\frac12\Big\|_{L^2_{x,t}}+  R^{-M}
\prod_{i=1}^k \| G_i\|_{L^2_{x,t}}.\end{aligned}
\Ee

We now partition the interval $I_{i,\mbq}$ further into
intervals $I_{i,\mbq}^l=[t_{l}, t_{l+1}]$, $l=1,\dots, \ell_0$, of
length $\sim R^{-1}$. Then the Fourier support of
$G_{i,\mbq}(\cdot,t)$, $t\in I_{i,\mbq}^l=[t_{l}, t_{l+1}]$ is
contained in $O(R^{-1})$ neighborhood of $\Gamma^{t_l}$. 
Let
$(\zeta_\mbq, \tau_\mbq)$ be the center of $\mbq$ and we define a set 
$\mbr_{i,\mbq}^{l}$  by  \Be
\mbr_{i,\mbq}^{l}=\Big\{(\zeta,\tau): |\zeta-\zeta_\mbq|\le
C\delta^\frac12, \ |\tau-\psi(\zeta_\mbq, {t_l})- \nabla_\zeta
\psi(\zeta_\mbq, {t_l}) \cdot(\zeta-\zeta_\mbq) |\le C\delta\Big\}\Ee
with a constant $C>0$ large enough.  It follows that Fourier transform of $G_{i,\mbq}(\cdot,t)$, $t\in I_{i,\mbq}^l$ is supported in  $\mbr_{i,\mbq}^{l}$.  This is easy to see from  2nd
order Taylor approximation because $\psi\in \fgee$.

Also define $\fm_{i,\mbq}^{l}$ by
\Be \label{flocal} \fm_{i,\mbq}^{l}=\rho\Big(\frac{\zeta-\zeta_\mbq}{C\sqrt\delta},
\frac{\tau-\psi(\zeta_q, {t_l})- \nabla_\zeta \psi(\zeta_q, {t_l})\cdot(\zeta-\zeta_\mbq)}{C\delta}
\Big) \Ee
with a suitable $C>0$ such that $\fm_{i,\mbq}^{l}$ is  comparable
to $1$  on $\mbr_{i,\mbq}^{l}$.
Now, we set \Be\label{tildeg}
 \mathcal F( G_{i,\mbq}^l(\cdot, t))=\big(\fm_{i,\mbq}^{l}\big)^{-1}\,\mathcal F(G_{i,\mbq}(\cdot, t))
\chi_{I_{i,\mbq}^l}(t).\Ee Denoting by $\mbn_{i,\mbq}^l$ 
the normal vector $\mbn(\zeta_\mbq, \psi(\zeta_\mbq, t_l))$, we also set
with a large   $C>0$
\[ {\mathbf T} _{i,\mbq}^{l}=\big\{x: |x\cdot \mbn_{i,\mbq}^l|\le C,\, |x-(x\cdot \mbn_{i,\mbq}^l)
\mbn_{i,\mbq}^l|\le CR^{-\frac12}\,\,\big\}.\]

Let us set $K_{i,\mbq}^{l}=\mathcal F^{-1}(\fm_{i,\mbq}^{l})$ so
that $ G_{i,\mbq}(\cdot,t)= G_{i,\mbq}^l(\cdot,t)\ast
K_{i,\mbq}^{l}$ if $t\in I_{i,\mbq}^l$. Since $\widehat \rho$ is
supported in $\fq(0,1)$, $|K_{i,\mbq}^{l}|\lesssim R^{-\frac{d+1}2}
\chi_{R{\mathbf T} _{i,\mbq}^{l}}$.
By \eqref{interval}  it
follows that $\sum_{\mbq} \| G_{i,\mbq}\ltwo^2=\sum_{\mbq} \|
G_{i,\mbq}\|_{L^2_t(I_{i,\mbq})}^2=
\sum_{\mbq,\,l}\|G_{i,\mbq}\|_{L^2_t(I_{i,\mbq}^l)}^2$. 
Thus, by \eqref{tildeg} we have 
\Be\label{squarehalf}
\begin{aligned}
\sum_{\mbq} \| G_{i,\mbq}\ltwo^2&=\sum_{\mbq,\,l}
\| G_{i,\mbq}^l(\cdot,t)\ast K_{i,\mbq}^{l}\|_{L^2(I_{i,\mbq}^l)}^2
\lesssim  \sum_{\mbq,\,l}\|
G_{i,\mbq}^l(\cdot,t)\|_{L^2(I_{i,\mbq}^l)}^2\ast |K_{i,\mbq}^{l}|
\\&\lesssim  \sum_{\mbq,\,l}\|
G_{i,\mbq}^l(\cdot,t)\|_{L^2(I_{i,\mbq}^l)}^2\ast
(R^{-\frac{d+1}2}\chi_{R{\mathbf T} _{i,\mbq}^l}).
\end{aligned}
\Ee

We denote by $\widetilde{{\mathbf T} }_{i,\mbq}^{l}$ the tube
$R^{1+\eps} {\mathbf T} _{i,\mbq}^{l}$ which is an $R^{1+\eps}$ times
dilation of ${\mathbf T} _{i,\mbq}^{l}$ from its center.
 So, from \eqref{squarehalf} we have, for $x,y\in B(z,R^{1/2+\eps})$,
\[  \sum_{\mbq} \| G_{i,\mbq} (y,\cdot)\ltwo^2\lesssim R^{ c\eps}  \sum_{\mbq,l} \| G_{i,\mbq}^l(\cdot,t)\|_{L^2(I_{i,\mbq}^l)}^2\ast
\Big(\frac{\chi_{\widetilde{{\mathbf T}
}_{i,\mbq}^{l}}}{|\widetilde{{\mathbf T} }_{i,\mbq}^{l}|}\Big)(x).
\]
Once we have this equality we can repeat the argument from \eqref{easy1} to \eqref{easy2} which is in Proof of Proposition \ref{confined} and   also using \eqref{half2}, we have
\begin{align*}
&\Big \| \prod_{i=1}^k \|G_i\|_{L^2_t(I)}\Lp {\frac 2{k-1}}{\ball
{0}{R}}
\!\!\! \lesssim \!\! R^{c\eps+\frac\alpha 2+ \frac{d-k}4}
 \Big \| \prod_{i=1}^k \!\Big(\!\sum_{\mbq,\,l}\| G_{i,\mbq}^l(\cdot,t)
 \|_{L^2(I_{i,\mbq}^l)}^2\ast(\frac{\chi_{\widetilde{{\mathbf T} }_{i,\mbq}^{l}}}
 {|\widetilde{{\mathbf T} }_{i,\mbq}^{l}|})\!\Big)^\frac12 \Lp {\frac 2{k-1}}{\ball {0}{\,2
R}} \!\!\!+ \cE,
\end{align*}
 where $\cE=R^{-M}
\prod_{i=1}^k \| G_i\|_{L^2_{x,t}}$ for any large $M>0$. 
Hence, for \eqref{assume2} it suffices to show that \begin{align*}\Big \|
\prod_{i=1}^k \Big(\sum_{\mbq,\,l}\| G_{i,\mbq}^l(\cdot,t)
 &\|_{L^2(I_{i,\mbq}^l)}^2\ast(\frac{\chi_{\widetilde{{\mathbf T} }_{i,\mbq}^{l}}}
 {|\widetilde{{\mathbf T} }_{i,\mbq}^{l}|})\Big)^\frac12 \Lp {\frac 2{k-1}}{\ball {0}{\,2
R}}\lesssim \sigma^{-\kappa} R^{c\epsilon}
R^{-\frac{d+k}4}\prod_{i=1}^k \|G_i\|_{L^2_{x,t}}.\end{align*}
Since $\|\| G_{i,\mbq}^l\|_{L^2_t(I_{i,\mbq}^l)}\|_2\sim
\|\|G_{i,\mbq}\|_{L^2_t(I_{i,\mbq}^l)}\|_2$ by \eqref{tildeg},
making use of disjointness of $I_{i,\mbq}^l$ and the supports of $
\mathcal F(G_{i,\mbq}(\cdot,t))$, and by Plancherel's theorem, $\sum_{\mbq,\,l}\|\|
G_{i,\mbq}^l\|_{L^2_t(I_{i,\mbq}^l)}\|_2^2\sim \sum_{\mbq} \|
G_{i,\mbq}\ltwo^2=\|\| G_{i}\|_{L^2_t(I)}\|_2^2.$   Hence, the above inequality 
follows from 
\begin{align*}
\Big \| \prod_{i=1}^k
\sum_{\mbq,\,l}f_{i,\mbq}^l\ast(\frac{\chi_{\widetilde{{\mathbf T}
}_{i,\mbq}^{l}}}
 {|\widetilde{{\mathbf T} }_{i,\mbq}^{l}|}) \Lp {\frac 1{k-1}}{\ball {0}{\,2
R}}\le C\sigma^{-\kappa} R^{c\epsilon} R^{-\frac{d+k}2}\prod_{i=1}^k
\sum_{\mbq,\,l} \|f_{i,\mbq}^l\|_1.\end{align*}

Let $\mathcal I_{i}=\{(\mbq,l): G_{i,\mbq}^l\neq 0\}$,  $I_i\subset \mathcal I_i$
and $\mathcal T_{i,\mbq}^l$ be a finite subset of
$\mathbb R^d$. By scaling and pigeonholing,  losing $(\log R)^C$ in its bound, this
reduces  to
\Be  \label{multi-kakeya00} \Big \| \prod_{i=1}^k
\sum_{(\mbq,\,l)\in \mathcal I_{i}}\sum_{\tau\in \mathcal T_{i,\mbq}^{l}}{\chi_{{{\mathbf
T}}_{i,\mbq}^{l}+\tau}}
  \Lp {\frac 1{k-1}}{\ball {0}{\,2
}}\le C\sigma^{-\kappa} R^{c\epsilon}
R^{\frac{d-k}2}\prod_{i=1}^k \sum_{(\mbq,\,l)\in \mathcal I_{i}}\sum_{\tau\in \mathcal T_{i,\mbq}^{l}}
{|{{\mathbf T} }_{i,\mbq}^{l}+\tau|}.
\Ee
Here we note that if $G_{i,\mbq}\neq 0$, then $\mbq\in \supp\, \mathcal F(G_i(\cdot,t))+O(\sqrt\delta)$ for some $t$.
So, by \eqref{transverse1} we have
$V\!ol(\mbn_{1}, \dots, \mbn_{k})\gtrsim\sigma$ whenever
$\mbn_i\in \{ \mbn_{i,\mbq}^{l}: G_{i,\mbq}^l\neq 0\}$, $i=1,\dots, k$.
Therefore, the estimate  follows from the  multilinear Kakeya
estimate which is stated below in Theorem \ref{k-kakeya}. This completes the proof.
\end{proof}

\begin{thm}[\cite{becata, gu, cv}]\label{k-kakeya} Let $2\le k\le d$, $1\ll R$ and\, $\fT_i$, $i=1,2,\dots, k$
be collections of  tubes
of  width $R^{-1/2}$ (possibly with infinite length), 
of which   major axes are parallel to the vectors in $\Theta_i\subset
\mathbb S^{d-1}$. Suppose  $ V\!ol(\theta_1,\theta_2, \dots, \theta_k)\ge \sigma $
holds whenever $\theta_i\in \Theta_i$, $i=1,\dots, k$, then there is
a constant $C$ such that,  for any subset $\mathcal T_i\subset \fT_i$, $i=1,\dots, k$,
\[\Big\| \prod_{i=1}^k \Big(\sum_{T_i\in \mathcal T_i} {\chi_{T_i}} \Big)\Big\|_{\frac{1}{k-1}(B(0,1))}\le
CR^{\frac{d-k}2}\sigma^{-1} \prod_{i=1}^k \Big(\sum_{T_i\in \mathcal T_i} | {T_i}|\Big).\]
\end{thm}

This is a rescaled version of the estimate  due to  Guth \cite{gu} (the case $d=k$) and Carbery-Valdimarsson
\cite{cv} (also see \cite{becata}).  However,  we don't need the endpoint
estimate for our purpose and the estimate in
\cite{becata} is actually  enough because we allow $\delta^{-\epsilon}$ loss in our estimate.

\begin{cor}\label{multi-squarekkk}
Let $\psi\in \fgee$, $\eta\in \mathcal E(N)$, and $0<\delta\ll
\sigma$. Suppose  that \eqref{transverse1} holds whenever $\xi_i\in \supp
\widehat f_i+O(\delta)$, $i=1,2,\dots,k$. Then, if $p\ge 2k/(k-1)$
and $\epsilon_\circ$ is small enough, for $\epsilon>0$, there is an
$N=N(\epsilon)$ such that the following estimate holds with $C,$
$C_\epsilon$, independent of $\psi$ and $\eta$:
\begin{align}
\nonumber &\Big\| \prod_{i=1}^k S_\delta(\psi,\eta)
f_i\Big\|_{L^\frac pk(B(x,\delta^{-1}))}\le
C\sigma^{-C_\epsilon}\delta^{-\epsilon} \prod_{i=1}^k \Big(\delta
\|f_i\|_{2}\Big).
\end{align}
\end{cor}

To show this we need only to replace $G_i$ with
$\phi\big(\frac{\eta(D,t)(D_d-\psi(D',t))}{\delta}\big)f_i$ and
apply Proposition \ref{multisq}. The assumptions in Proposition
\ref{multisq} are satisfied  with $G_1, \dots, G_k$. Thus, the
estimate is straightforward because
$\|\phi\big(\frac{\eta(D,t)(D_d-\psi(D',t))}{\delta}\big)f_i\|_{L_{x,t}^2}\lesssim
\delta^\frac12\|f\|_2$, which follows by Plancherel's theorem and taking $t$-integration first.

The following is a consequence  of Corollary \ref{multi-squarekkk}
and localization argument in the proof of Proposition \ref{multilp}.

\begin{prop}\label{multi-square2}
Let $0<\delta\ll \sigma\ll \widetilde \sigma\ll  1$ and $\psi\in
\fgee$, $\eta\in \mathcal E(N)$ and let $Q_1,\dots, Q_k\subset
\frac12I^d$ be dyadic cubes of sidelength $\widetilde \sigma$.
Suppose that \eqref{transverse1} is satisfied whenever $\xi_i\in
Q_i$, $i=1,\dots, k$, and suppose that $\supp \widehat f_i\subset
Q_i$, $i=1,\dots, k$. Then, if $p\ge 2k/(k-1)$ and $\epsilon_\circ$
is small enough, for $\epsilon>0$ there is an $N=N(\epsilon)$ such
that
\begin{align}
\label{lppp}&\Big\| \prod_{i=1}^k S_\delta (\psi,\eta)
f_i\Big\|_{\frac pk}\le C\sigma^{-C_\epsilon}\delta^{-\epsilon}
\prod_{i=1}^k \Big(\delta^{\frac dp-\frac{d-2}{2}} \|f_i\|_{p}\Big).
\end{align} holds with
$C, C_\epsilon$, independent of $\psi$ and $\eta$.
\end{prop}

\begin{proof} The proof is similar to that of Proposition \ref{multilp}. So, we shall be brief.
Let $\varphi$, $\widetilde Q_i$, $\widetilde \chi_i$, $\{\cB\}$, and
$\{\widetilde \cB\}$ be the same as in  the proof of Proposition
\ref{multilp}. We set
\[  K_i^t = \mathcal F^{-1}\Big(\phi\Big(\frac{\eta(\xi,t)(\tau-\psi(\zeta,t))}{\delta}\Big)\widetilde \chi_i(\xi)\Big)\,.\] 
Then  $S_\delta (\psi,\eta) f_i= \| K_i^t \ast f_i\|_{L^2_t(I)}.$ The $p/k$-th
power of the left hand side of \eqref{lppp} is bounded by
\begin{align*}
\sum_{\cB}\int_{\cB} \prod_{i=1}^k \| K_i^t \ast
f_i\|_{L^2_t(I)}^\frac pk dx  & \lesssim    I+I\!I, 
\end{align*}
where 
\[   I=\sum_{\cB}\int_{\cB} \prod_{i=1}^k
\| K_i^t \ast (\chi_{\widetilde \cB} f_i)\|_{L^2_t(I)}^\frac pk\, dx, \ \  I\!I=   
\sum_{\cB}\,\,\,\Big( \sum_{\substack{
g_i= \chi_{\widetilde \cB^c} f_i \text{ for some } i} } \int_{\cB}
\prod_{i=1}^k
 \| K_i^t \ast g_i\|_{L^2_t(I)}^\frac pk
dx\Big).\] 
As before, the second sum is taken over all choices with $ g_i= \chi_{\widetilde
\cB}
f_i \text{ or } \chi_{\widetilde \cB^c} f_i, $ and  $g_i= \chi_{\widetilde \cB^c} f_i$  for some $i$. 
By choosing $c>0$ small enough, we see that $ \widetilde \chi_1(D)
(\chi_{\widetilde \cB} f_1)$, $\dots$, $ \widetilde \chi_k(D)
(\chi_{\widetilde \cB} f_k)$ satisfy the assumption of Corollary
\ref{multi-squarekkk}. Since $K_i^t \ast(\chi_{\widetilde \cB}
f_i))=\phi\big(\frac{\eta(D,t)(D_d-\psi(D',t))}{\delta}\big)\widetilde
\chi_i(D) (\chi_{\widetilde \cB} f_i)$,  by Corollary
\ref{multi-squarekkk} and H\"older's inequality
\begin{align*}
 I\,\lesssim
\sigma^{-C_\eps}\big(\frac1\delta\big)^{\eps} \sum_{\cB}
\prod_{i=1}^k \delta^\frac{p}{k} \big\|\chi_{\widetilde \cB}
f_i\big\|_2^\frac
pk \lesssim \sigma^{-C_\eps}\big(\frac1\delta\big)^{c\eps}
\Big(\prod_{i=1}^k \delta^{\frac dp-\frac{d-2}2 }\big\|
f_i\big\|_p\Big)^\frac pk.
\end{align*} 
To handle $I\!I$ we note from 
Lemma \ref{kerneldelta1} that  $|K_i^t(x)| \le C\delta\fK_M(x)$ with $C$,
depending only on $\|\psi\|_{C^N(I^{d-1})},$ $
\|\eta\|_{C^N(I^{d})}$. Thus, $\|K_i^t\ast(\chi_{\widetilde \cB^c}
f_i)(x) \|_{L^2_t}\le C\delta\delta^{\eps (M-d-1)} \fK_{d+1}\ast
|f_i|(x)$ if $x\in B$, and  $\|K_i\ast f_i(x)\|_{L^2_t(I)}\le C\delta
\fK_{d+1}\ast |f_i|(x)$. The rest of proof is the same as before. We
omit the details.
\end{proof}


\subsection{Multilinear square function estimate with confined direction sets }
From the point view of  Proposition \ref{confined}
we may expect  a better estimate thanks to smallness of  supports of Fourier transforms  of the  input functions 
when they are confined in a small neighborhood of a $k$-dimensional submanifold.
The following is a vector valued generalization of Proposition \ref{confined}.

\begin{prop}\label{confinedsqr} Let  $k$, $2\le k\le d$, be an integer, $0<\sigma \ll 1$ be fixed,
and\, $\Pi\subset \mathbb R^d$ be a $k$-plane  containing the
origin. Let $\psi\in\fgee$  and
$\Gamma^{t}$ be 
defined by \eqref{gamma}. For $0<\delta\ll \sigma$, suppose that the
functions $G_1,\dots, G_k$ defined on $\mathbb R^d\times I$ satisfy \eqref{deltanbdg} for $t\in I$ and
\eqref{transverse1} whenever $\xi_i\in \supp\, \mathcal F(G_i(\cdot, t))+O(\delta), \,\, i=1,2,\dots,k$, for  some $t\in I$. Additionally we assume that, for all $t\in I$,  \Be
\label{confinedsq} \mbn\big(\supp\, \widehat G_1(\cdot, t)\big),\dots,
\mbn\big(\supp\, \widehat G_k(\cdot, t)\big) \subset \mathbb S^{d-1}\cap
(\Pi+O(\delta)).\Ee
 Then, if $2\le p\le 2k/(k-1)$ and $\epsilon_0$ is sufficiently small, for $\epsilon>0$
there is an $N=N(\epsilon)$ such that  \Be
\label{squarel2} \Big\| \prod_{i=1}^k
\|G_i\|_{L^2_t(I)}\Big\|_{L^\frac pk(B(x,\delta^{-1}))}\lesssim
\sigma^{-C_\epsilon} \delta^{dk(\frac12-\frac1p)-\epsilon}
\prod_{i=1}^k \|G_i\|_{L^2_{x,t}} \Ee
holds uniformly for $\psi\in \fgee$.
\end{prop}

The following   is an easy consequence of
\eqref{squarel2}.

\begin{cor}\label{sq}
Let $\{\fq\}$, $\fq\subset \frac12 I^d$,  be the collection of dyadic cubes of side length $l$, $\delta< l\le 2\delta$. Define $ G_{i,\fq}$ by $\mathcal F( {G_{i,\fq}(\cdot, t)})=\chi_\fq \mathcal F( {G_{i}(\cdot, t)})$
and set $R=1/\delta$. Suppose that
the same assumptions as in Proposition \ref{confinedsqr} are satisfied.  Then,
if $2\le p\le 2k/(k-1)$ and $\epsilon_\circ$ is small enough, for $\epsilon>0$
there is an $N=N(\epsilon)$ such that
\begin{equation}
\label{squarefunt21} \Big\| \prod_{i=1}^k
\|G_i\|_{L^2_t(I)}\Big\|_{L^{\frac pk}(B(x,R))}\lesssim
\sigma^{-C_\epsilon} \delta^{-\epsilon} \prod_{i=1}^k
\Big\|\Big(\sum_{\fq } \| G_{i,\fq}\ltwo
^2\Big)^\frac12\rho_{B(x,R)}\Big\|_{p}
\end{equation}
holds uniformly for $\psi\in \fgee$. 
\end{cor}

 \begin{proof} Observe that
$ \big\| \prod_{i=1}^k   \|G_i\|_{L^2_t(I)}\big\|_{L^\frac
pk( B(x,R))}\le \big\| \prod_{i=1}^k
\|\rho\big(\frac{\cdot-x}{R}\big) G_i\|_{L^2_t(I)}\big\|_{L^\frac
pk}.$ Then, the functions $\rho\big(\frac{\cdot-x}{R}\big)
G_i$, $i=1,\dots, k$, satisfy the assumption in Proposition
\ref{confinedsqr} because $\supp\, \mathcal F(\rho\big(\frac{\cdot-x}{R}\big)
G_i(\cdot, t))$ $=\supp\,  \widehat G(\cdot, t)+O(R^{-1})$. So, from  Proposition
\ref{confinedsqr} we get
\[
\Big\| \prod_{i=1}^k   \|G_i\|_{L^2_t(I)}\Big\|_{L^\frac
pk(B(x,R))}\lesssim \sigma^{-C_\epsilon}
    R^{\epsilon} \prod_{i=1}^k
R^{-d(\frac12-\frac1p)}\big\|\| \rho\big(\frac{\cdot-x}{R}\big)
G_{i}\|_{L_{}^2} \big\|_{L^2_t(I)}.\] Since $G_i=\sum_{\fq}
G_{i,\fq}$ and supports of $\{\mathcal
F(\rho\big(\frac{\cdot-x}{R}\big) G_{i, \fq}(\cdot, t))\}_\fq$ are
boundedly overlapping, by Plancherel's theorem it follows that
$\big\|\| \rho\big(\frac{\cdot-x}{R}\big)
G_{i}\|_{L_{x}^2} \big\|_{L^2_t(I)}\lesssim \big\|\big(\sum_{\fq }
\| \rho\big(\frac{\cdot-x}{R}\big)
G_{i,\fq}\|_2^2\big)^\frac12\big\|_{L^2_t(I)}.$ Combining this with
the above inequality, we get
\begin{align*}
\Big\| \prod_{i=1}^k   \|G_i\|_{L^2_t(I)}\Big\|_{L^\frac
pk(B(x,R))}\lesssim \sigma^{-C_\epsilon}
    R^{\epsilon}  \prod_{i=1}^k
R^{-d(\frac12-\frac1p)}\Big\||\rho\big(\frac{\cdot-x}{R}\big)|\Big(\sum_{\fq } \|
G_{i,\fq}\ltwo ^2\Big)^\frac12\Big\|_{2}.
\end{align*}
Now H\"older's inequality gives the desired estimate \eqref{squarefunt21}.
\end{proof}

As an application of Corollary  \ref{sq} we obtain the following.

\begin{cor}\label{squarefunt222}
Let $\psi\in \fgee$, $\eta\in \mathcal E(N)$, $0<\delta\ll
\widetilde \sigma\ll \sigma$, and
${S_{\delta}}={S_{\delta}}(\psi,\eta)$ be defined by \eqref{srdef}.
Let $\Pi$ be a $k$-plane which contains the origin. Suppose
\eqref{transverse1} holds whenever $\xi_i\in \supp\,\widehat
f_i+O(\widetilde \sigma)$, $i=1,2,\dots,k,$ and \Be
\label{angleconf}\mbn\big(\supp\,\widehat f_i\big)\subset
\Pi+O(\widetilde \sigma), \quad i=1,2,\dots,k.\Ee Let $\{\fq \}$, $\fq\in \frac12 I^d$,  be
the  collection of dyadic cubes of side length $l$,
 ${\widetilde \sigma}< l\le 2{\widetilde \sigma}$.
 Define $f_{i,\fq}$ by $\mathcal F( {f_{i,\fq}})=\chi_\fq \mathcal F( {f_{i}})$.
Then, if  $2k/(k-1)\le p\le 2$ and $\epsilon_\circ$ is sufficiently small, for $\epsilon>0$
there is an $N=N(\epsilon)$ such that
\[\Big\| \prod_{i=1}^k S_{\delta}f_i\Big\|_{L^\frac pk(B(x,1/\widetilde \sigma))}
\lesssim \sigma^{-C_\epsilon}{\widetilde \sigma}^{-\epsilon}
\prod_{i=1}^k \Big\| \Big(\sum_{\fq} |S_{\delta}
f_{i,\fq}|^2\Big)^\frac12 \rho_{B(x,1/\widetilde\sigma)}
 \Big\|_{L^p}\]
 holds uniformly for $\psi$ and $\eta$.
 \end{cor}

This follows from Corollary \ref{sq}. Indeed, it suffices to check that $G_i=\rho\big(\widetilde\sigma({\cdot-x})\big)$
$\phi\big(\frac{D_d-\psi(D',t)}{\sigma}\big)f_i$ satisfies the assumption of Corollary \ref{sq} with $\delta=\widetilde \sigma$ as long as $\sigma \ll \widetilde\sigma$. This is clear because 
$\widehat{G_i}(\cdot, t)=
 \widetilde \sigma^{-d}\big( e^{i<\cdot, x>}\rho\big( \cdot/\widetilde \sigma\big)\big)
 \ast \big( \phi\big(\frac{\tau-\psi(\zeta,t)}{\sigma}\big)\widehat f_i\,\big)$.

\begin{proof}[Proof of Proposition \ref{confinedsqr}]
The argument here is similar to the proof of Proposition \ref{multisq}.
The estimate for $p=2$ follows from H\"older's  inequality and Plancherel's theorem.
So, by interpolation it is
sufficient to show \eqref{squarel2} for $p=2k/(k-1)$.

Let us set $R=1/\delta\gg 1$ and we may set $x=0$. As usual we start
with the assumption that, for $0<\delta\ll \sigma$,
\Be\label{assume11} \Big\| \prod_{i=1}^k
\|G_i\|_{L^2_t(I)}\Big\|_{L^\frac pk(B(0,R))}\le C R^\alpha
R^{-\frac d2} \prod_{i=1}^k \|G_i\|_{L^2_{x,t}}\Ee holds uniformly
for $\psi\in \fgee$  whenever $G_1,\dots, G_k$ satisfy
\eqref{deltanbdg}, \eqref{transverse1} and \eqref{confinedsq}. By
\eqref{trivial} and H\"older's inequality \eqref{assume11} is true
with some large $\alpha$.  As before it is sufficient to show that
\eqref{assume11} implies  for any $\eps>0$ there is an $N=N(\eps)$
such that
\[ \Big\|
\prod_{i=1}^k \|G_i\|_{L^2_t(I)}\|_{L^\frac pk(B(0,R))}\le
C\sigma^{-\kappa} R^{\frac\alpha 2+c\eps} R^{-\frac d2}
\prod_{i=1}^k \|G_i\|_{L^2_{x,t}}
\]
holds uniformly for $\psi\in \fgee$. Then
iteration of this implication gives the desired estimate
\eqref{squarel2}. 

Fix $z\in \mathbb R^{d}$ and consider $\rho_{B(z,\sqrt R)}
G_1(\cdot, t), \dots, \rho_{B(z,\sqrt R)} G_k(\cdot, t)$. Then it is
clear from \eqref{deltanbdg} and \eqref{confinedsq} that $\supp\, \cF(\rho_{B(z,\sqrt R)} G_i(\cdot, t))$ is contained in
$\Gamma^{t}+O(R^{-1/2})$ and
 $\mbn(\supp\, \mathcal F(\rho_{B(z,\sqrt R)} G_i(\cdot, t)))\subset  \Pi+O(R^{-1/2})$. Also,
 since $\delta\ll \sigma$, \eqref{transverse1} holds if $\xi_i\in
 \supp\, \mathcal F(\rho_{B(z,\sqrt R)} G_i(\cdot, t))$.  Hence, by the assumption \eqref{assume11} we get
\Be\label{halfsq} 
\Big \| \prod_{i=1}^k \|\rho_{B(z,\sqrt R)}
G_i\|_{L^2_t(I)}\Big\|_{L^\frac 2{k-1}}\lesssim  R^\frac\alpha 2
R^{-\frac d4} \prod_{i=1}^k \|\rho_{B(z,\sqrt R)}
G_i\|_{L^2_{x,t}}.\Ee

Now we proceed in the same way as in the proof of Proposition
\ref{multisq}, and we keep using the same notations.
As before,  let
 $\{\mbq\,\}$ be the collection of dyadic cubes (hence essentially disjoint) of sidelength $\sim R^{-1/2}$
 such  that $I^d=\bigcup \mbq $.
We decompose the function $G_i(\cdot, t)$ into $G_{i,\mbq}(\cdot, t)$
which is defined by \eqref{decomp2}, and get \eqref{l2ortho}, which is clear.
Then, combining $\eqref{l2ortho}$
and  \eqref{halfsq},
 we have
$\big \| \prod_{i=1}^k \|\rho_{B(z,\sqrt R)} G_i\|_{L^2_t(I)}\big
\|_{L^ {\frac 2{k-1}}}\le C R^\frac\alpha 2 R^{-\frac d4}
\prod_{i=1}^k \big\|\big(\sum_{\mbq}|\rho_{B(z,\sqrt R)}
G_{i,\mbq}|^2\big)^\frac12\big\|_{L^2_{x,t}}. $ Then this gives 
\begin{align}
\label{half22} \Big \| \prod_{i=1}^k \|\rho_{B(z,\sqrt R)}
G_i\|_{L^2_t(I)}\Big \|_{L^ {\frac 2{k-1}}}
\lesssim R^{\frac\alpha 2-\frac d4} \prod_{i=1}^k
\Big\|\chi_{B(z,R^{\frac12+\eps})}\Big(\sum_{\mbq}|
G_{i,\mbq}|^2\Big)^\frac12\Big\|_{L^2_{x,t}} + \cE.
\end{align}
where $\cE= R^{-M} \prod_{i=1}^k \|
G_i\|_{L^2_{x,t}}$ for 
any large $M$.

We also denote by $(\mathrm N^t){}^{-1}$ (defined from $\mathrm N^t(I^{d-1})$ to $I^{d-1}$)  the inverse of $\mathrm N^t:\Gamma^{t}\to \mathbb S^{d-1}$
which is well defined because $\psi\in \fgee$.
Since $\partial_t\psi\in (1-\epsilon_\circ, 1+\epsilon_\circ)$, there is an interval $I_{i,\mbq}$ of length $C
R^{-1/2}$ such that $G_{i,\mbq}(\cdot, t)=0$ if $t\not\in I_{i,\mbq}$ (see \eqref{interval}).
As in  the proof of Proposition
\ref{multisq} we partition $I_{i,\mbq}$ into intervals 
$I_{i,\mbq}^l=[t_l, t_{l+1}]$, $l=1,\dots, l_0$, of sidelength $\sim R^{-1}$.
Since the Fourier transform of $G_i(\cdot, t)$ is supported in
$\Gamma^{t_l}+O(\delta)$ if $t\in I_{i,\mbq}^l=[t_l, t_{l+1}]$ and  the normal
vectors are confined in $\Pi+O(\delta)$, it follows that
\[\supp\, \mathcal F( G_{i,\mbq}(\cdot, t)) \subset \Gamma^{t_l}(\delta)\cap \big(({\mathrm N}^{t_l})^{-1}(\Pi)+O(\delta)\big), \quad t\in
[t_l, t_{l+1}].\]

Fix $t_l$, and let us set 
\[\xi_{i,\mbq}^{t_l}=(\zeta_{i,\mbq}^{t_l},\tau_{i,\mbq}^{t_l})\in \big(({\mathrm
N}^{t_l})^{-1}(\Pi)\cap \Gamma^{t_l}\big) \cap  \big(\supp \mathcal F(
G_{i,\mbq}(\cdot, t_l))+O(\delta)\big). \] 
(As before, we may assume that  this set is nonempty, otherwise the associated function $G_{i,\mbq}^l=0$. See below.)  Let  $v_1, \cdots, v_{k-1}$ be an
orthonormal basis for the tangent space $T_{\xi_{i,\mbq}^{t_l}}(
({\mathrm N}^{t_l})^{-1}(\Pi))$ at $\xi_{i,\mbq}^{t_l}$, and 
$u_1, \cdots, u_{d-k}$ be a set of orthonormal vectors such that $ \{
{\mathrm N}^{t_l}(\xi_{i,\mbq}^{t_l}), v_1,$ $\dots, v_{k-1},
u_1,\dots, u_{d-k}\}$ forms an orthonormal basis for $\mathbb
R^{d}$. Let us  set
\[\begin{aligned}
\mbr_{i,\mbq}^{t_l}= \big\{\xi:|(\xi-\xi_{i,\mbq}^{t_l})\cdot
{\mathrm N}^{t_l}(\xi_{i,\mbq}^{t_l})|\le &C\delta,\,\,
|(\xi-\xi_{i,\mbq}^{t_l})\cdot v_i|\le C\sqrt\delta,\,\, i=1,\dots,
k-1,\,
  \\  &\, |(\xi-\xi_{i,\mbq}^{t_l})\cdot u_i|\le C\delta, \,\,i=1,\dots d-k \big\}
  \end{aligned}
 \]
and
\[\begin{aligned}
\mathbf P_{i,\mbq}^{t_l} =\big\{\xi:|\xi\cdot {\mathrm
N}^{t_l}(\xi_{i,\mbq}^{t_l})|\le &C,\,\,
 |\xi\cdot v_i|\le C\sqrt \delta,\, i=1,\dots, k-1,\,|\xi\cdot u_i|\le C, \,i=1,\dots d-k \big\}\end{aligned}\]
with a sufficiently large $C>0$. Then  $\mathcal F(
G_{i,\mbq}(\cdot, t))$, $t\in [t_l, t_{l+1}]$ is supported in
$\mbr_{i,\mbq}^{t_l}$.

The rest of proof is similar to that of Proposition \ref{multisq}, so we shall be brief.  Let  $\fm_{i,\mbq}^{t_l}$  be a smooth function naturally  adapted to   $\mbr_{i,\mbq}^{t_l}$  such that  $\fm_{i,\mbq}^{t_l}\sim 1$ on $\mbr_{i,\mbq}^{t_l}$ and $\mathcal F^{-1}(\fm_{i,\mbq}^{t_l})$ is supported in  $R\mathbf  P_{i,\mbq}^{t_l}$. This can be done by using $\rho$ and composition with it an appropriate affine map (for example, see \eqref{flocal}).   As before we define $ G_{i,\mbq}^l(\cdot, t)$
by \eqref{tildeg} and let $K_{i,\mbq}^{t_l}=\mathcal F^{-1}(\fm_{i,\mbq}^{t_l})$
so that $
 G_{i,\mbq}^l(\cdot,t)=  G_{i,\mbq}^l(\cdot,t)\ast K_{i,\mbq}^{t_l}$ if
$t\in I_{i,\mbq}^l$. Hence, $\sum_\mbq  G_{i,\mbq}=\sum_{\mbq,l}
 G_{i,\mbq}^l(\cdot,t)\ast K_{i,\mbq}^{t_l}$, 
$|K_{i,\mbq}^{t_l}|\lesssim  | R\mathbf
P_{i,\mbq}^{t_l}|^{-1} \chi_{R\mathbf
P_{i,\mbq}^{t_l}}$.  Let us set  $\widetilde {\mathbf P}_{i,\mbq}^{t_l}=R^{1+\eps}  \mathbf
P_{i,\mbq}^{t_l}$.  
Hence, from the same lines of inequalities  as in
\eqref{squarehalf} and repeating the similar argument in the proof of Proposition \ref{multisq}  we have, for  $x\in B(y,R^{1/2+\eps})$,  
\[ \prod_{i=1}^k
\Big(\sum_{\mbq} \|
G_{i,\mbq}\ltwo^2(x)\Big)\lesssim  R^{ c\eps} \prod_{i=1}^k  \sum_{\mbq,i} \| G_{i,\mbq}^l(\cdot,t)\|_{L^2(I_\mbq^l)}^2\ast
(\frac{\chi_{{\widetilde{\mathbf P}}_{i,\mbq}^{t_l}}}{|{\widetilde{\mathbf P}}_{i,\mbq}^{t_l}|})(y).\] 
Now,  we use the lines of argument from \eqref{easy1} to \eqref{easy2}, and  combine this with \eqref{half22}  to get 
\begin{align*}
&\Big \| \prod_{i=1}^k \|G_i\|_{L^2_t(I)}\Lp {\frac 2{k-1}}{\ball
{0}{R}}
\lesssim R^{c\eps+\frac\alpha 2}  \Big \| \prod_{i=1}^k \Big(&\sum_{\mbq,\,l}\| G_{i,\mbq}^l(\cdot,t)\|_{L^2(I_\mbq^l)}^2\ast(\frac{\chi_{{\widetilde{\mathbf P}}_{i,\mbq}^{t_l}}}{|{\widetilde{\mathbf P}}_{i,\mbq}^{t_l}|})\Big)^\frac12 \Lp {\frac
2{k-1}}{\ball {0}{ 2R}} +\cE. 
\end{align*}
Since $\sum_{\mbq,\,l}\|\|\widetilde
G_{i,\mbq}\|_{L^2_t(I_\mbq^l)}\|_2^2\sim \sum_{\mbq}\|\|G_{i,\mbq}\|_{L^2_t(I_\mbq)}\|_2^2\sim
\|G_i\|_{L^2_{x,t}}$, the proof is completed if we show
\begin{align*}
\Big \| \prod_{i=1}^k \Big(\sum_{\mbq,\,l}f_{\mbq,\,l}&\ast
\frac{\chi_{{\widetilde{\mathbf
P}}_{i,\mbq}^{t_l}}}{|{\widetilde{\mathbf P}}_{i,\mbq}^{t_l}|}\Big)
\Big \|_{L^{\frac 2{k-1}}({\ball {0}{ 2R}})}\le C
R^{c\eps}\sigma^{-1} R^{-d} \prod_{i=1}^k
\Big(\sum_{\mbq,\,l}\|f_{\mbq,\,l} \|_1\Big).
\end{align*}

Finally, to show the above inequality we may repeat the argument in
the last part in the proof of Proposition \ref{confined}.  In fact, we need only to show the associated Kakeya estimate (for example, see \eqref{multi-kakeya0}, \eqref{multi-kakeya00}).  Using the
coordinates $(u,v)\in \Pi\times \Pi^\perp=\mathbb R^d$, it is sufficient to
show that the longer sides of ${\mathbf P}_{i,\mbq}^{t_l}$ are
transverse to $\Pi$. More precisely, if $\epsilon_\circ$ is sufficiently
 small and $N$ is large enough,  there is a constant $c>0$, independent of $\psi\in \fgee$, such
 that,  for $w\in \big(T_{\xi_{i,\mbq}^{t_l}}(\mathrm N^{-1} (\Pi))\oplus \text{span}\{\mathrm
  N(\xi_{i,\mbq}^{t_l})\}\big)^\perp$, \eqref{angle} holds.  Since
$\psi(\zeta,t)=\frac12|\zeta|^2+t+  \mathcal R$ with $\|\mathcal
R\|_{C^N(I^d\times I)}\le \epsilon_\circ$, by the same perturbation
argument it is sufficient to consider
$\psi(\zeta,t)=\frac12|\zeta|^2+t$. For this case \eqref{angle}
clearly holds for $w\in \big(T_{\xi_{i,\mbq}^{t_l}}(\mathrm N^{-1}
(\Pi))\oplus \text{span}\{\mathrm
  N(\xi_{i,\mbq}^{t_l})\}\big)^\perp$ because  translation by $t$ doesn't have any effect. The same argument works without modification. This
completes the proof.
\end{proof}

\subsection{Multi-scale decomposition for $S_\delta f$}
In this section we obtain multi-scale decomposition for the square function,  which is to be
combined with multilinear square function estimates to prove Proposition
\ref{localf2}. This is will be carried out in the similar way that
we obtain the decomposition in Section \ref{multiplier} though we
need to take care of the additional $t$ average.

Let $0<\epsilon_\circ\ll 1$, $1\ll N$, $\psi\in \fgee$, $\eta\in
\cE(N)$, and $\sdel$ be given by \eqref{srdef}. Let $\mathrm N^t$,
$\mbn$ be given by Definition \ref{nvector}. Let
$\kappa=\kappa(\epsilon_\circ, N)$ be the number given in
Proposition \ref{rescalesquare} so that \eqref{rescaless} holds
whenever $0<\eps\le \kappa$, $\psi\in \fgee$, and $\eta\in \cE(N)$.
As before, let $\sigma_1, \dots, \sigma_m$, and $M_1, \dots, M_m$  be
dyadic numbers such that \Be\label{dyadicsss}\delta\ll
\sigma_{d-1}\ll \dots \ll \sigma_1\ll \min(\kappa,1), \quad M_i=1/\sigma_i.\Ee
We assume that $f$ is Fourier supported in $\frac12 I^d$. We keep
using the same notation as in Section \ref{multi-scale}.
In particular, $\{\qqs i\}$, $\{\fQ^i\}$ are the collection of (closed)
dyadic intervals of sidelength $2\sigma_i$, $2 M_i$,
respectively,  so that \eqref{icube}  and \eqref{spatialcube} holds .

\subsubsection{Decomposition by normal vector sets} Let $\{\theta^{\,i}\}$ be a discrete subset of $\sphere$
of which elements are separated by distance $\sim\sigma_i$. Let $\fd^i$ be disjoint  subsets of $\{\fq^i\}$ which satisfies, for some $\theta^{\,i}$,   \Be\label{Thetai} \fd^i\subset \{\fq^i: \dist(\mbn(\fq^i), \theta^i)\le C\sigma_i\}\Ee
and 
\Be \label{Thetai1}\bigcup_{\fd^i} \fd^i=\{\fq^i\},  \quad i=1, \dots, m.\Ee
Obviously, such a partitioning of  $\{\fq^i\}$ is possible.  Disjointness between 
$\fd^i$ will be useful later for decomposing the square function. Then we also define an auxiliary operator by
\[{\sctp {} i f} = \Big(\sum_{\fq^i\in \fd^i}  |\sdel f_{\fq^i}|^2\Big)^\frac12.\]
Similarly, as before, $\tti,$ $\tti_\ast,$ $\tti_j$,
and $\tti_{j\ast}$ denote the elements in $\{ \tti\}$  for the rest of this section.
\begin{defn}  We define
$\mbn(\fd^i)$ to be a vector\footnote{Possibly, there are more than one $\theta$. In the case we simply choose one of them. Ambiguity of the definition does not  cause any problem in what follows. }  $\theta\in \{\theta^{\,i}\}$  such that $\dist(\mbn(\fq^i), \theta)\le C \sigma_i$ whenever
$\fq^i\in \fd^i$.  Particularly, we may set  $\mbn(\fd^i)=\theta^i$ if \eqref{Thetai} holds. 
\end{defn}

Since the map $\mathrm N^t$ is injective for each $t$, the elements of $\fd^i$ are contained in a $O(\sigma_i)$ neighborhood of the curve
$\{\xi: \mbn(\xi)=\theta^i\}$ with $\theta^i=\mbn(\fd^i)$.
From $\eqref{ell2}$
we observe that for any interval $J$ of length $\sigma_i$
there are as many as $O(1)$ $\fq^i\in \fd^i$
such that $\phi\big(\frac{D_d-\psi(D',t)}{\delta}\big) \ffqqs i\not=0 $
if $t\in J$. Hence, dividing $I$ intervals of length $\sim \sigma_i$ and taking integration in $t$ we see that
\Be\label{lps} \sdel(\sum_{\fq^i\in \fd^i} f_{\fq^i}) \lesssim
\Big(\sum_{\fq^i\in \fd^i}  |\sdel
f_{\fq^i}|^2\Big)^\frac12=\sctp {}i f \Ee with the implicit
constant independent of $\fd^i$. Since ${S_\delta} f\le
\sum_{\fd^i}\sdel(\sum_{\fq^i\in \fd^i} f_{\fq^i}) , \ i=1,\dots,
m$, we also have
\begin{equation}\label{sectorialbounds}
{S_\delta} f \lesssim \sum_{\fd^i}\Big(\sum_{\fq^i\in \fd^i}
|\sdel f_{\fq^i}|^2\Big)^\frac12=\sum_{\fd^i}\sctp {}i f.
\end{equation}

\subsubsection{$\sigma_1$-scale decomposition}
Decomposition at this stage  is similar with that of $T_\delta$ in
Section 2. So, we shall be brief.   Fix $x\in \mathbb R^d$ and let us denote by $\fd^{1}_\ast\in \{\fd^1\}$ such that
\[\fS^{}_{\fd^{1}_{\ast}}f (x)=\max_{\fd^1
}\fS^{}_{\fd^{1}}f (x).\]
 Considering the cases
$\sum_{\ti1}\fS^{}_{\tii{1}{}}f (x)\le 100^d \fS^{}_{\tii{1}{\ast}}f (x)$ and $
\sum_{\ti1}\fS^{}_{\tii{1}{}}f (x)>  100^d \fS^{}_{\tii{1}{\ast}}f(x)$ separately,
we have
\begin{align*}
\sdel f(x)
&\lesssim \sum_{\ti1}\fS^{}_{\ti{1}}f (x)
 \lesssim \fS^{}_{\tii{1}{\ast}}f (x)
 +\sigma_1^{1-d} \max_{\substack{\fd^1: |\mbn(\tii1\ast)-\mbn(\ti1)|\gtrsim
\sigma_1}} ( \fS^{}_{\tii{1}\ast}f (x) \fS^{}_{\ti{1}}f(x))^\frac12
\\
&\lesssim \fS^{}_{\tii{1}{\ast}}f (x)+\sigma_1^{1-d}\max_{\substack{{\tii11, \tii12}:
|\mbn(\tii11)-\mbn(\tii12)|\gtrsim
\sigma_1}}( \fS^{}_{\tii{1}1}f (x) \fS^{}_{\tii{1}2}f(x))^\frac12 \,.
\end{align*}
Since $\# \fd^i\lesssim \sigma_1^{-1}$ and $\sctffp 11 \sctffp  12 =
\big(\sum_{\vpi 11\in \tii 11, \vpi 12\in \tii 12}
 \big({S_\delta}\vppf11{S_\delta}\vppf12\big)^2\,\big)^{1/2}$,  \begin{align*}
\sdel f(x) &\lesssim \sigma_1^{\frac1p-\frac12}
\Big(\sum_{\fq^1\in \fd^1_\ast}  |\sdel^{}
f_{\fq^1}|^p\Big)^\frac1p
+\sigma_1^{-C}\Big(\sum_{\substack{{\tii11, \tii12}: 
|\mbn(\tii11)-\mbn(\tii12)|\gtrsim
\sigma_1}}\big({S_\delta}\vppf11{S_\delta}\vppf12\big)^\frac
p2\Big)^\frac1{p}.
\end{align*}
Taking $L^p$ norm on both side of the inequality yields
\begin{align*}
\|\sdel f\|_p &\lesssim  \sigma_1^{\frac1p-\frac12}
\Big(\sum_{\fq^1}  \|\sdel^{} f_{\fq^1}\|_p^p\Big)^\frac1p
+\sigma_1^{-C}\Big(\sum_{{\qqq11, \qqq12}: trans}\|{S_\delta}\vppf11{S_\delta}\vppf12\|_{\frac p2}^\frac
p2\Big)^\frac1{p}. 
\end{align*}
Hence, using Lemma  \ref{rescalesquare} and Lemma \ref{vector}, we have
\begin{align}\label{2linear}
\|\sdel f\|_p &\lesssim  \sigma_1^{\frac2p}
B_{p}(\sigma_1^{-2} \delta)\|f\|_p +\sigma_1^{-C}
\max_{{\qqq11, \qqq12}: trans}
\|{S_\delta}\vppf11{S_\delta}\vppf12\|_{\frac p2}^\frac 12\,.
\end{align}
We proceed to decompose those terms appearing in  the bilinear expression.  

\subsubsection{$\sigma_k$-scale decomposition, $k\ge 2$}  Fixing $\sigma$, for $l\in \sigma^{-1}\mathbb Z^d$, let $A_l$ and $\tau_l$ be given by \eqref{A-tau}.  The following is a slight modification of Lemma \ref{scattered}.
   
\begin{lem} \label{smod}Let  $\fd$ be a subset of $\{\fq^i\}$. Set
$\fS_\fd^{}f=\big(\sum_{\fq^i\in \fd}  |\sdel^{} f_{\fq^i}|^{2}\big)^{1/2}$, and set 
\[[\sctp {} {} f]  = \sum_{l\in M_i\mathbb Z^d}
A_l^\frac12 \sctp {} {} {(\tau_{l}f)},\,\,\,\, |\![\sctp {} {} f]\!|  = \sum_{l,l'\in M_i  \mathbb   Z^d}
(A_lA_{l'})^\frac12 \sctp {} {} {(\tau_{(l+l')}f)}.\]
If $x,$ $x_0\in \fQ^i$, the following inequality holds with the implicit constants independent of $\fd$:
\Be\label{smodul}\sctp {} {} f (x)  \lesssim  [\sctp {} {} f ](x_0) \lesssim
|\![\sctp {} {} f ]\!|(x).\Ee
\end{lem}

\begin{proof} Note that $\fq^i$ is a cube of sidelength $2\sigma_i$. Since $x,$ $x_0\in \fQ^i$,
using \eqref{fexpansion} and Cauchy-Schwarz inequality, we get
\begin{align*}
\Big|\phi\Big(&\frac{D_d-\psi(D',t)}{\delta}\Big)f_{\fq^i}(x)\Big|^2\lesssim
\sum_{l\in  M_i\mathbb Z^d} A_l
\Big|\phi\Big(\frac{\eta(D,t)(D_d-\psi(D',t))}{\delta}\Big)\tau_{l}f_{\fq^i}(x_0)\Big|^2.
\end{align*}
By taking integration in $t$ we get
\begin{align}\label{xx}
(\sdel f_{\fq^i}(x))^2\lesssim
\sum_{l\in  M_i\mathbb Z^d} A_l(\sdel(\tau_{l}f_{\fq^i})(x_0))^2.
\end{align}
Summation in
$\fq^i\in \fd$ gives
\[\Big(\sum_{\fq^i\in \fd}(\sdel f_{\fq^i}(x))^2\Big)^\frac12\lesssim
\sum_{l\in M_i\mathbb Z^d} A_l^\frac12 \Big(\sum_{\fq^i\in \fd} (\sdel(\tau_{l}f_{\fq^i})(x_0))^2\Big)^\frac12,\]
by which we get the first inequality of \eqref{smodul}.
By interchanging the roles of $x$ and $x_0$ in \eqref{xx} and summation in  $\fq^i\in \fd$  it follows that
\[\sum_{\fq^i\in \fd} (\sdel(\tau_{l}f_{\fq^i})(x_0))^2\lesssim \sum_{l\in M_i\mathbb Z^d}
A_{l'}\sum_{\fq^i\in \fd} (\sdel(\tau_{(l+l')}f_{\fq^i})(x))^2 \]
Putting this in the right hand side of the above inequality and repeating the same argument,
we get the second inequality of \eqref{smodul}.
\end{proof}

Now we have bilinear decomposition \eqref{2linear} on which we build higher degree of multilinear decomposition. 

\subsubsection{From $k$-transversal to $k+1$-transversal,  $2\le k\le m$}   Let us be given  cubes $\vpp {k-1} 1,\vpp {k-1}2,\dots,
\vpp {k-1}{k}$ of sidelength $\sigma_{k-1}$ which satisfy
\eqref{trans-k}. Though we use the same notations as in the multiplier estimate  case, it should be noted that the normal vector field  $\mathbf n$ is defined on $I^{d-1}\times CI$ (see Definition  \ref{nvector}).   As before, we denote by $\{\vpp k i\}$ the collection of
dyadic cubes of sidelength $\sigma_k$ contained in $\vpp {k-1} i$ (see \eqref{k-cube}), which  are partitioned  into the subsets of $\{\tii{k}i\}$ so
that
\[\bigcup_{\tii ki}\Big(\bigcup_{\vpp ki\in  \tii ki} \vpp ki\Big)= \vpp{k-1}i, \, i=1,\dots, k.\]
So,
we can write
\[
\prod_{i=1}^k {S_\delta}(\sum_{\vpp ki\subset \vpp {k-1}i} \vppff
{k-1}i ) =  \prod_{i=1}^k {S_\delta}(\sum_{\tii{k}i}\sum_{\vpp ki\in
\tii ki} \vppff ki )
\]
and recall the definition
$ \sqfrq ki{k-1} := \big(\sum_{\fq^k_i\in \fd^k_j}  |\sdel F_{\fq^k_i}|^2\big)^{1/2}.$
Fix $\fQ^k$ and let $x_0$ be the center of $\fQ^k$.
Let  $\tii {k}{i\ast}\in \{\tii ki\}$ be an angular partition such that
\[\sqfr {k}{i\ast}{k-1}i(x_0)=\max_{\tii k i} \sqfr  k i{k-1}i(x_0). \]

Let us set 
\Be\label{majorks}
\overline{\Lambda_i^k}=\big\{\tii k{i}: [\sqfr ki{k-1}i](x_0)> (\sigma_k)^{kd}
\max_{1\le j\le k}[\sqfr {k}{j\ast}{k-1}j](x_0)\big\}, \quad 1\le i \le k\,. \Ee
We split the sum to get
\Be\label{split1}\begin{aligned}
 \prod_{i=1}^k 
 {S_\delta}(\sum_{\tii{k}i}\sum_{\vpp ki\in \tii ki} \vppff ki )
\le  
\prod_{i=1}^k {S_\delta}(\sum_{\tii{k}i\in \overline{\Lambda^k_i}\,}
\sum_{\vpp ki\in \tii ki} \vppff ki )   +  \sum_{(\tii k1,
\dots, \tii kk)\not\in \prod_{i=1}^k \overline{\Lambda^k_i} } \prod_{i=1}^k
{S_\delta}(\sum_{\vpp ki\in \tii ki} \vppff ki ).
\end{aligned}
\Ee 
Thus, if $x\in \mathfrak Q^k$, by \eqref{smodul} and \eqref{lps}  the second term in the right hand side
is bounded by 
\Be\label{mins}
\begin{aligned} 
&\sum_{(\tii k1, \dots, \tii kk)\not\in \prod_{i=1}^k\overline{\Lambda^k_i}\, }\, 
     \prod_{i=1}^k {S_\delta}(\sum_{\vpp ki\in \tii ki} \vppff ki )(x)
         \lesssim \sum_{(\tii k1, \dots, \tii kk)\not\in \prod_{i=1}^k 
                 \overline{\Lambda^k_i}\, } \,\prod_{i=1}^k
                    [\sqfr ki{k-1}i](x_0)
\\
&\lesssim
\big(\max_{1\le j\le k}[\sqfr{k}{j\ast} {k-1}j](x_0)\big)^k
\lesssim 
\big(\max_{1\le j\le k} [ \fS_{\tii{k}{j\ast}  }   f  ](x_0)\big)^k
\lesssim
(\max_{\tii
k{}}|\![\fS_{\tii{k}{}  }   f]\!|(x))^k. 
\end{aligned}
\Ee
Here $\{\tii k{}\}=\bigcup_{1\le i\le k}  \{\tii ki\}$ and the third inequality follows from the definition of 
$\fS_{\tii{k}{j}} f$ because $\vpp k i\subset \vpp {k-1}i$ . 
  Since \eqref{mins} holds for each $\mathfrak Q^k$, 
  integrating over all $\mathfrak Q^k$, using Lemma \ref{smod}, 
  Proposition \ref{rescalesquare} and
Lemma \ref{vector}, we get 
\Be
\begin{aligned}\label{maxs} 
&\Big\|\sum_{(\tii k1, \dots,    \tii kk)\not\in \prod_{i=1}^k\overline{\Lambda^k_i}\, }\, 
     \prod_{i=1}^k {S_\delta}(\sum_{\vpp ki\in \tii ki} \vppff ki )\Big\|_{\frac pk}^\frac1k
\lesssim 
     \|\max_{\tii k{}}|\![\fS_{\tii{k}{}  }   f]\!|\|_p  
\lesssim 
   \sup_{h}  \|\max_{\tii k{}}   
                          \fS_{\tii{k}{}  }  ( \tau_{h} f)\|_p   
     \\&
     \lesssim   \sup_{h} \Big( \sum_{\tii k{}}   \|  
                          \fS_{\tii{k}{}  }  ( \tau_{h} f)\|_p^p \Big)^\frac1p
  \lesssim \sup_{h} \sigma_k^{(\frac1p-\frac12)}
\Big(\sum_{\fq^k_i} \|\sdel  \tau_h f_{\fq^k_i}\|^p_p\Big)^{1/p}
\lesssim
\sigma_k^{\frac2p} B_{p}(\sigma_k^{-2}
\delta)\|f\|_p. 
\end{aligned}
\Ee
The inequality before the last one follows from the definition of  $\fS_{\tii{k}{}  } f$ and H\"older's inequality 
since there are as many as  $O(\sigma^{-1}_k)$ $\vpp k{}\subset \tii{k}{}$.

We note that  vectors $\mbn(\fd^{k}_{1\ast}),$ $ \dots, \mbn(\fd^{k}_{k\ast})$ are linearly independent  because $\vpp {k-1} 1,\vpp {k-1}2,\dots,
\vpp {k-1}{k}: trans$. We also  
denote by $ \Pi_*^k=\Pi_*^k(\vpp{k-1}{1},\dots, \vpp{k-1}{k}, \fQ^k)$ the $k$ plane  spanned by  the vectors 
$\mbn(\fd^{k}_{1\ast}),$ $ \dots, \mbn(\fd^{k}_{k\ast})$.  Let us set 
\[ \overline{{\fN}}=\overline{{\fN}}(\vpp{k-1}{1},\dots, \vpp{k-1}{k}, \fQ^k)=  \{\fd^{k}:
\dist(\mbn(\fd^{k}), \Pi_*^k)\le C\sigma_k\}.\]
We   split the sum and  use the triangle inequality so that
\begin{equation}\label{split2}\begin{aligned}
 \prod_{i=1}^k {S_\delta}
 (\sum_{\tii{k}i\in \overline{\Lambda^k_i}}\sum_{\vpp ki\in \tii ki} \vppff ki )
&\le \prod_{i=1}^k {S_\delta}(\!\!\!\! \sum_{ \substack{\tii k i\in
\overline{\Lambda^k_i} : \, 
 \tii ki\in \overline{{\fN}} \,\, }} \sum_{\vpp ki\in \tii ki} \vppff ki )
+\!\!\!\!
\sum_{\substack{\tii k i\in  \overline{\Lambda^k_i}:\, 
\tii ki\not\in \overline{{\fN}} \text{ for some } i}} \prod_{i=1}^k
{S_\delta}(\sum_{\vpp ki\in \tii ki} \vppff ki ).
\end{aligned}
\end{equation}
For the $k$-tuples
$(\tii
k1,\dots,\tii k k)$ appearing in the second summation of the right hand side,
there is a $\tii ki$ for which $\mbn(\tii ki)$ is not contained in
$\Pi_*^k+O(\sigma_k)$. In  particular, suppose that $\mbn(\tii k1)\not\in \Pi_*^k+O(\sigma_k)$.
Then,  by  \eqref{smodul} and  \eqref{majorks}  we have
\begin{align*}
 \prod_{i=1}^k {S_\delta}(\sum_{\vpp ki\in \tii ki} \vppff ki )(x)
\lesssim
\prod_{i=1}^k [\sqfr{k}{i}{k-1}i](x_0)
\le \sigma_k^{-C} ([\sqfr {k}{1}{k-1}{1}](x_0))^\frac k{k+1}\prod_{i=1}^k ([\sqfr {k}{i\ast}{k-1}{i}](x_0))^\frac k{k+1}.
\end{align*}

Recall that $V\!ol(\mbn(\xi_1),\mbn(\xi_2),\dots, \mbn(\xi_k))\gtrsim \sigma_1\dots\sigma_{k-1}$ if
$\xi_i\in \qqq{k-1}i$, $i=1,\dots, k$. From the definition of $\overline{\fN}$ it follows that
$\dist (\mbn(\fq^k), \Pi_*^k)\gtrsim \sigma_{k}$ if $\fq^k\in \fd^k$ and $\mbn(\fd^k)\not \in  \overline  \fN$. Hence
$V\!ol(\mbn(\xi_1),\mbn(\xi_2), \dots, \mbn(\xi_k),$ $ \mbn(\xi_{k+1}))\gtrsim \sigma_1\dots \sigma_{k}$
if $\xi_i\in  \fq^k_i$ and $\fq^k_i\in \tii k{i\ast}$, $i=1,\dots,k$, 
and $\xi_{k+1}\in \fq_{k+1}^k$ and $\fq_{k+1}^k\in \tii k1$. Hence these cubes are transversal.
Since there are only $O(\sigma_k^{-C})$ $\sigma_k$-scale cubes, by \eqref{smodul} and H\"older's inequality
\begin{align*}
 &
 \qquad\prod_{i=1}^k {S_\delta}
 (\sum_{\vpp ki\in \tii ki} \vppff ki )(x)
\lesssim 
\sigma_k^{-C} (|\![\sqfr {k}{1}{k-1}1]\!|(x))^\frac k{k+1}\prod_{i=1}^k
(|\![\sqfr {k}{i\ast}{k-1}i]\!|(x))^\frac k{k+1}
\\
\lesssim &\sigma_k^{-C}\sum_{l_1, l_1', \dots, l_{k+1}, l_{k+1}'\in M_k \mathbb Z^d}
 \prod_{i=1}^{k+1} \widetilde A_{l_i}\widetilde A_{l_i'}
\Big(\sum_{\qqq k1, \dots,\qqq k{k+1}:trans} \Big(\prod_{i=1}^{k+1}
{S_\delta}(\tau_{(l_i+l_i')}\vppff{k}i)(x)\Big)^\frac p{k+1}\,
\Big)^\frac{k}{p}.
\end{align*}
Here $\widetilde A_{l_i}, \widetilde A_{l_i'}$ are rapidly decaying sequences.  
The same is true for any $\tii k1,\dots, \tii kk$ satisfying   $\tii k i\in \overline{\Lambda^k_i}$, $1\le i\le k$, 
and $ \tii ki\not\in\overline{{\fN}}$ for some $i$ and
this  holds regardless of $\fQ^k$. So, we have, for any $x$,
\Be\label{split3}\begin{aligned}
 &\sum_{\substack{\tii k i\in \overline{\Lambda^k_i}: 
\tii ki\not\in \overline{{\fN}} \text{ for some } i}}
\,\,\,\prod_{i=1}^k S_\delta(\sum_{\vpp ki\in \tii ki} \vppff ki )(x)
\\
\lesssim &\sigma_k^{-C}\sum_{l_1, l_1', \dots, l_{k+1}, l_{k+1}'}
 \prod_{i=1}^{k+1} \widetilde A_{l_i}\widetilde A_{l_i'}
\Big(\sum_{\qqq k1, \dots,\qqq k{k+1}:trans} \Big(\prod_{i=1}^{k+1}
{S_\delta}(\tau_{(l_i+l_i')}\vppff{k}i)(x)\Big)^\frac p{k+1}\,
\Big)^\frac{k}{p}.
\end{aligned}
\end{equation}
Since  $\widetilde A_{l_i}, \widetilde A_{l_i'}$ are rapidly decaying,  taking $L^{p/k}$ norm  and a simple manipulation give
 \Be\label{obss}\begin{aligned}
\Big\|\sum_{\substack{\tii k i\in \overline{\Lambda^k_i}:
\\
\tii ki\not\in \overline{{\fN}} \text{ for some } i}}
\prod_{i=1}^k S_\delta(\sum_{\vpp ki\in \tii ki} \vppff ki ) \Big\|_{\frac pk}
\lesssim \sigma_k^{-C}\sup_{\tau_1,\dots, \tau_{k+1}}
\max_{\qqq k1, \dots,\qqq k{k+1}:trans} \Big\|\prod_{i=1}^{k+1}
{S_\delta}(\tau_{i}\vppff{k}i)\Big\|_{\frac
p{k+1}}^\frac{k}{k+1}.
\end{aligned}
\Ee

We now combine the inequalities  \eqref{split1}, \eqref{mins}, \eqref{split2}, \eqref{split3} to get
\[
\begin{aligned}
\prod_{i=1}^k& {S_\delta}(\sum_{\tii{k}i}\sum_{\vpp ki\in \tii ki} \vppff ki ) 
\lesssim 
(\max_{\tii
k{}}|\![\fS_{\tii{k}{}  }   f]\!|(x))^k
+ \chi_{\fQ^k}\prod_{i=1}^k {S_\delta}( \sum_{ \substack{\tii k
i\in \overline{\Lambda^k_i}: \, \tii ki\in\overline{{\fN}}  }}\,\, \sum_{\vpp ki\in \tii ki} \vppff ki )
\\
+&\sigma_k^{-C}\sum_{l_1, l_1', \dots, l_{k+1}, l_{k+1}'}
 \prod_{i=1}^{k+1} \widetilde A_{l_i}\widetilde A_{l_i'}
\Big(\sum_{\qqq k1, \dots,\qqq k{k+1}:trans} \Big(\prod_{i=1}^{k+1}
{S_\delta}(\tau_{(l_i+l_i')}\vppff{k}i)(x)\Big)^\frac p{k+1}\,
\Big)^\frac{k}{p}.
\end{aligned}
\]
Here $\overline\fN$ depends on $\vpp{k-1}{1},\dots, \vpp{k-1}{k}, \fQ^k$. 
By taking $1/k$-th power, integrating on $\mathbb R^{d}$ and using \eqref{maxs} and \eqref{obss} we get
\Be\label{scale-k} \begin{aligned}
 \Big\|\Big(\prod_{i=1}^k S_\delta(&\sum_{\vpp ki\subset \vpp {k-1}i} \vppff {k}i )\Big)^\frac 1k\Big\|_{p}
    \lesssim  
  \sigma_k^{\frac 2p} B_{p}(\sigma_k^{-2} \delta)\|f\|_p
  +\sigma_k^{-C}\sup_{\tau_1,\dots, \tau_{k+1}}
\max_{\substack{\qqq k1, \dots,\qqq k{k+1}:\\ trans}} \Big\|\prod_{i=1}^{k+1}
{S_\delta}(\tau_i\vppff{k}i)\Big\|_{\frac
p{k+1}}^\frac{k}{k+1}
\\
  &\qquad +\Big(\sum_{\fQ^k}
\Big\|\prod_{i=1}^k  {S_\delta} ( \sum_{\substack{\tii ki: \tii ki\in \\ [\overline{{\fN}}](\vpp{k-1}{1},\dots,
\vpp{k-1}{k}, \fQ^k)}} \,\,\sum_{\substack{\vpp ki\in \tii ki:\\ \vpp
ki\subset \vpp {k-1}i}}\vppff ki ) \Big\|_{L^\frac pk(\fQ^k)}^{\frac
pk}\Big)^\frac 1p, 
\end{aligned}
\Ee
where $[\overline{{\fN}}](\vpp{k-1}{1},\dots,
\vpp{k-1}{k}, \fQ^k)$ denotes a subset of $\overline{{\fN}}(\vpp{k-1}{1},\dots,
\vpp{k-1}{k}, \fQ^k)$ which depends on $\vpp{k-1}{1},\dots,$ $
\vpp{k-1}{k}, \fQ^k$.

\subsubsection{Multi-scale decomposition} For $k=2,\dots, m$, let us set
\[ \overline{\fM^{k}}\!f
        = \sup_{\tau_1,\dots,\tau_k} \max_{\qqq {k-1}1, \dots, \qqq {k-1}k: trans}
                \Big(\sum_{\fQ^k} \Big\|\prod_{i=1}^k  {S_\delta} (
                     \sum_{\substack{\tii ki: \tii ki\in \\
                         [\overline{{\fN}}](\vpp{k-1}{1},\dots, \vpp{k-1}{k}, \fQ^k)}} \sum_{\substack{\vpp
                                   ki\in \tii ki:\\ \vpp ki\subset \vpp {k-1}i}} \tau_i\vppf ki
)\Big\|_{L^\frac pk(\fQ^k)}^{\frac pk}\Big)^\frac 1p.\]
 Here $[\overline{{\fN}}](\vpp{k-1}{1},\dots, \vpp{k-1}{k}, \fQ^k)$ also depends on $\tau_1,
\dots, \tau_k$ but this doesn't affect the overall bound.  Starting from \eqref{2linear} we successively apply
\eqref{scale-k} to $k$-scale transversal products (given by $\qqq
{k-1}1, \dots,$ $\qqq{k-1}{k}:trans$). After decomposition up to $m$-th scale we get
\Be\label{scale-2}
\begin{aligned}
\|\sdel f\|_p \lesssim
\sum_{k=1}^m \sigma_{k-1}^{-C} \sigma_k^{\frac2p} & B_{p}(\sigma_k^{-2} \delta)\|f\|_p
+ \sum_{k=2}^m\sigma_{k-1}^{-C} \overline{\fM^k} f\\
&+
\sigma_m^{-C}\sup_{\tau_1, \dots, \tau_{m+1}}\,\,\, \max_{\qqq{m}1,\dots\qqq m{m+1}: trans}
\Big \| \prod_{i=1}^{m+1} \sdel \tau_i f_{\qqq {m}{i}}\Big\|_{L^\frac p{m+1}}^{\frac 1{m+1}}.
\end{aligned}
\Ee

\subsection{Proof of Proposition  \ref{localf2}}\label{pf-sq}  
We may assume 
$d\ge 9$ since $p_s\ge  {2(d+2)}/d$ for $d<9$ and the sharp bound for $p\ge {2(d+2)}/d$ is verified in \cite{lrs}. So, we have $p_s(d)\ge\frac{2(d-1)}{d-2}$.  
The proof is similar to that of  Proposition
\ref{localfrequency}. 
 Let $\beta>0$ and we aim to show that $\cB^\beta(s)\le C$ for $0<s\le
1$ if $p\ge p_s(d)$.  We choose $\epsilon>0$ such that $(100d)^{-1}\beta\ge \epsilon$. 
 Fix  $\epsilon_\circ>0$ and $N=N(\epsilon)$  such that Corollaries 
\ref{multi-squarekkk}, \ref{sq} and \ref{squarefunt222} hold
uniformly for $\psi \in \fgee$.

Let $s< \delta\le 1$. Obviously,
$(\sigma_k^{-2}\delta)^{\frac{d-2}2-\frac dp+\beta}
B(\sigma_k^{-2}\delta) \le  \mathcal B^\beta(s)+\sigma_k^{-C}$ because
$s\le \sigma_k^{-2}\delta$ and $B(\delta)=B_{p}(\delta)
\le C$ for $\delta\gtrsim 1$. Hence,  it follows that
\Be
\label{est1}
\begin{aligned}
\sigma_k^{\frac2p} B(\sigma_k^{-2}\delta)
\lesssim \sigma_k^{2(\frac{d-2}2-\frac {d-1}p)+2\beta}
\delta^{-\frac{d-2}2+\frac dp-\beta}( \cB^\beta(s)+\sigma_k^{-C}).
\end{aligned}
\Ee 
 We first consider  the $(m+1)$-product in \eqref{scale-2}. By Corollary
\ref{multi-squarekkk} we have, for $p\ge 2(m+1)/m$, \Be \label{est2}
\sup_{\tau_1, \dots, \tau_{m+1}} \max_{\qqq{m}1,\dots\qqq m{m+1}:
trans} \Big \| \prod_{i=1}^{m+1} \sdel \tau_i f_{\qqq
{m}{i}}\Big\|_{L^\frac p{m+1}}^{\frac 1{m+1}} \le C_\epsilon
\sigma_m^{-C} \delta^{-\frac{d-2}2+\frac dp-\epsilon} \|f\|_p. \Ee

For $\overline{\fM^k}$, as before  we have two types of estimates. The first one follows
from Corollary \ref{multi-squarekkk} while the second one is a
consequence of the square function estimates in  Corollary
\ref{squarefunt222}. From the definition  of $\overline{\fM^k}$, we note that  $\vpp {k}1,\vpp {k}2,\dots,\vpp {k}{k}$
are  contained, respectively, in $\vpp {k-1}1,\vpp {k-1}2,\dots,\vpp
{k-1}{k}$ which are transversal. Hence, we have 
\[ \prod_{i=1}^k{S_\delta}
( \sum_{\tii ki\in
[\overline{{\fN}}](\vpp{k-1}{1},\dots, \vpp{k-1}{k}, \fQ^k)}
\sum_{\substack{\vpp ki\in \tii ki:\\ \vpp ki\subset \vpp {k-1}i}}
\tau_i\vppf ki )(x) \le \sum_{\vpp {k}1,\vpp {k}2,\dots,\vpp
{k}{k}:trans}\prod_{i=1}^k{S_\delta} (\tau_i\vppf ki )(x). \]
Here $\vpp {k}1,\vpp {k}2,\dots,\vpp
{k}{k}:trans$ means
$V\!ol(\mbn(\xi_1),\dots,\mbn(\xi_k))\ge\sigma_1\dots\sigma_{k-1}$
provided $\xi_i\in \vpp k i$, $i=1\dots, k$. Since  there are as many
as $O(\sigma_{k-1}^{-C})$ $\vpp {k-1}1,\dots,\vpp {k-1}k$ and the above holds regardless of $\fQ^k$, by
Corollary  \ref{squarefunt222} we have, for $p\ge 2k/(k-1)$,
\Be\begin{aligned} \label{est3r}
\overline{\fM^k}  f\lesssim \sigma_{k}^{-C}
\sup_{\tau_1,\dots, \tau_k}\sum_{\vpp {k}1,\vpp {k}2,\dots,\vpp
{k}{k}:trans} \Big\|\prod_{i=1}^k{S_\delta} (\tau_i\vppf ki
)\Big\|_{\frac pk} \lesssim \sigma_{k}^{-C}\delta^{-\frac{d-2}2+\frac dp-\epsilon}
\|f\|_p.\end{aligned}
\Ee

\subsubsection*{Estimates for $\overline{\fM^k}$ via  Corollary
\ref{sq}} By fixing  ${\tau_1,\dots,\tau_k}$,  and  $(\qqq {k-1}1, \dots, \qqq {k-1}k)$ satisfying  $\qqq {k-1}1, \dots, \qqq {k-1}k: trans$,  we first handle  the integral over $\fQ^k$  which is  in the definition of  $\overline{\fM^k}$.  For $i=1,\dots, k$, set
\[f_i=\sum_{\tii ki\in
[\overline{{\fN}}](\vpp{k-1}{1},\dots, \vpp{k-1}{k}, \fQ^k)}\,\,
\Big(\sum_{\substack{\vpp ki\in \tii ki:  \vpp ki\subset \vpp {k-1}i}} \tau_i\vppf ki\Big).\]
Since  $\vpp {k-1}1,\dots,\vpp {k-1}k:trans$,
\eqref{transverse1}
holds with $\sigma=\sigma_1\dots\sigma_{k-1}$ whenever
$\xi_i\in \supp \widehat f_i+O(\sigma_k)$, $i=1,2,\dots,k$.
Also note that $\mbn(\tii k1),\dots,$ $\mbn(\tii k k)
\subset \Pi_\ast^k(\vpp {k-1}1,\dots,\vpp {k-1}k,$ $ \fQ^{k})$. Hence,
it follows that \eqref{angleconf} holds with $\widetilde \sigma=\sigma_k$.
Let us set 
\[  \cQ(\vpp{k-1}{1},\dots,
\vpp{k-1}{k}, \fQ^k)=\big \{  \vpp k{}:    
\mathbf n(\vpp k{}) \in [\overline{\fN}](\vpp {k-1}1,\dots,\vpp {k-1}k, \fQ^{k})\big\}\] 
Let write $\fQ^k=\fq(z,1/\sigma_k)$. Then,  by Corollary
\ref{squarefunt222}  we have, for $2\le p\le
2k/(k-1)$, 
\begin{align*}
\Big\|\Big(\prod_{i=1}^k S_\delta  f_i 
\Big)^\frac1k\Big\|_{L^p(\fQ^k)}^p
&\lesssim  \sigma_{k-1}^{-C_\epsilon}\sigma_{k}^{-\epsilon}
\prod_{i=1}^k 
\Big\|\Big(\sum_{\substack{\vpp {k}i\in \vpp i{k-1}:\, \vpp {k}i\in \cQ(\vpp{k-1}{1},\dots,
\vpp{k-1}{k}, \fQ^k)}} |\sdel \tau_i \vppf k{}|^2\Big)^\frac12\rho_{B(z,\frac C{\sigma_k})}\Big\|_{L^p}^\frac pk.
\end{align*}

The dyadic cubes of sidelength $\sigma_k$ in $\cQ(\vpp{k-1}{1},\dots,
\vpp{k-1}{k}, \fQ^k)$ are contained in $O(\sigma_k)$-neighborhood
of $\mbn^{-1}(\Pi_\ast^k)$ which is a smooth $k$-dimensional surface. Thus, 
$\# \{\vpp {k}i\subset \vpp i{k-1}: \vpp {k}i\in \cQ(\vpp{k-1}{1},\dots,
\vpp{k-1}{k}, \fQ^k)\} $ $ \le C\sigma^{-k}_k.$
Now, by H\"older's inequality we get 
\begin{align*}
\Big\|\Big(\prod_{i=1}^k S_\delta  f_i 
\Big)^\frac1k\Big\|_{L^p(\fQ^k)}^p
&\lesssim \sigma_{k-1}^{-C_\epsilon}\sigma_{k}^{-\epsilon-k(\frac p2-1)}
 \prod_{i=1}^k \Big\|\Big(\sum_{\vpp {k}i\subset \vpp {k-1}i}
|\sdel \tau_i\vppf k{i}|^p\Big)^\frac1p \rho_{\fQ^k}\Big\|_{L^p}^\frac pk\,. 
\end{align*}
Summation along $\fQ^k$ using rapid decay of Schwartz function
$\rho$ gives 
\begin{align*}
\Big\|\Big(\prod_{i=1}^k S_\delta  f_i 
\Big)^\frac1k\Big\|_{L^p}
\lesssim \sigma_{k-1}^{-C_\epsilon}\sigma_{k}^{-\epsilon-k(\frac 12-\frac1p)}
 \prod_{i=1}^k \Big\|\Big(\sum
|\sdel \tau_i\vppf k{i}|^p\Big)^\frac1p\Big\|_p^\frac1k.
\end{align*}
Hence, using  Proposition \ref{rescalesquare}, Lemma \ref{vector}, and \eqref{est1},  
for $2\le p\le \frac{2k}{k-1}$, we have 
\begin{align*}
&\qquad\Big\|\Big(\prod_{i=1}^k S_\delta  f_i  \Big)^\frac1k\Big\|_{L^p}
   \lesssim \sigma_{k-1}^{-C}  \sigma_{k}^{-\epsilon-\frac{k-1}2+\frac{k+1}p}
B(\sigma_k^{-2}\delta) \|f\|_p  \lesssim    \sigma_{k-1}^{-C} \delta^{-\frac{d-2}2+\frac dp -\beta} \times
\\
&  \sigma_{k}^{\beta+\frac{2d-k-3}2-\frac{2d-k-1}p}   
 \big(\sigma_k^{-C}+
\cB^\beta(s)\big) \|f\|_p
 \lesssim \sigma_{k-1}^{-C} \delta^{-\frac{d-2}2+\frac dp -\beta}
\big(\sigma_k^{-C}+\sigma_{k}^{\alpha}
\cB^\beta(s)\big) \|f\|_p\,
\end{align*}
with some $\alpha>0$ if $p\ge \frac{2(2d-k-1)}{2d-k-3}$. Here we have used 
$(100d)^{-1}\beta\ge \epsilon$.  We note that  the right hand side of the above is independent of 
${\tau_1,\dots,\tau_k}$ and  there are only 
$O(\sigma_{k-1}^{-C})$ many $k$-tuples  $(\qqq {k-1}1, \dots, \qqq {k-1}k)$ satisfying 
 $\qqq {k-1}1, \dots, \qqq {k-1}k: trans$. Thus, recalling the definition of $\overline{{\fM^{k}}\!f}$, 
  we have for $2\le p\le \frac{2k}{k-1}$ 
\[ 
\overline{{\fM^{k}}\!f} 
\lesssim \sigma_{k-1}^{-C} \delta^{-\frac{d-2}2+\frac dp -\beta}
\Big(\sigma_k^{-C}+\sigma_{k}^{\alpha}
\cB^\beta(s)\Big) \|f\|_p\,
\] 
with some $\alpha>0$ provided that $p\ge \frac{2(2d-k-1)}{2d-k-3}$.
Combining  this and \eqref{est3r}  we have,  for some $\alpha>0$, 
\begin{align}\label{est4s}
\overline{\fM^k} f
\le C\delta^{-\frac{d-2}2+\frac dp-\beta}\Big(\sigma_k^{-C}+\sigma_{k}^{\alpha}
\cB^\beta(s)\Big) \|f\|_p\,.
\end{align}
provided that $ p\ge\min\Big(\frac{2(2d-k-1)}{2d-k-3},
\frac{2k}{k-1}\Big).$

\subsubsection*{Closing induction}  Let us set \[p(m)= \max \Big( \max_{1\le k\le m}\min\big(\frac{2(2d-k-1)}{2d-k-3},
\frac{2k}{k-1}\big), \,\,\frac{2(m+1)}{m}\Big).\]
Since $p\ge p_s> \frac{2(d-1)}{d-2}$ and $(100d)^{-1}\beta\ge \epsilon$,  we have  $\sigma_k^{\frac2p} B(\sigma_k^{-2}\delta)
\lesssim \sigma_k^{\alpha}
\delta^{-\frac{d-2}2+\frac dp-\beta}( \cB^\beta(s)+\sigma_k^{-C})$  for some $\alpha>0$. Using \eqref{scale-2}, we  combine  the estimates \eqref{est1}, \eqref{est2},
and \eqref{est4s} to get
\begin{align*}
\|\sdel f\|_p\le C\sum_{k=1}^{m} \big(\sigma_{k-1}^{-C} +\sigma_{k-1}^{-C}
\sigma_{k}^{\alpha} \cB^\beta(s)\big)
\delta^{-\frac{d-2}2+\frac dp-\beta}  \|f\|_p+ C\sigma_{m}^{-C}
\delta^{-\frac{d-2}2+\frac dp-\beta} \|f\|_p
\end{align*}
for some $\alpha>0$ as long as $p\ge p(m)$.  The rest of proof is similar to that in Section \ref{closinginduction}. So, we intend to be brief. 
By using stability of the estimates along $\psi\in\fgee$, $\eta\in\mathcal E(N)$,
multiplying by $\delta^{\frac{d-2}2-\frac dp+\beta}$ on both sides and taking supremum along $\psi$, $\eta$ and $f$,   
and  taking supremum along $\delta$, 
$s< \delta\le 1$,  we get
\[\cB^\beta(s)
 \le C\Big(\sum_{k=1}^{m}\sigma_{k-1}^{-C}
\sigma_{k}^{\alpha} \Big)\cB^\beta(s)+ C \sum_{k=1}^{m}\sigma_{k}^{-C}\]
for some $\alpha>0$ provided that $p\ge p(m)$.
Choosing $\sigma_1,\dots, \sigma_{m-1}$ so that 
$C\Big(\sum_{k=1}^{m-1}\sigma_{k-1}^{-C}
\sigma_{k}^{\alpha} \Big)\le 1/2, $ 
gives  $\cB^\epsilon(\delta)\le C\sigma_m^{-C}$ for $p\ge p(m)$. Therefore, to complete the proof  we need only to  check that  the minimum of $p(m)$, $2\le m \le d-1$, is  $p_s$. This can be done by a simple computation.


\begin{thebibliography}{9}

\bibitem{amarco} M.  Annoni, \emph{Almost everywhere convergence for modified Bochner-Riesz means at the critical index for $p\ge 2$},  Pacific J. Math. {\bf 286} (2017), 257--275. 

 \bibitem{beje}  I. Bejenaru, \emph{The optimal trilinear restriction estimate for a class of hypersurfaces with curvature,} Adv. Math. {\bf 307} (2017), 1151--1183.

\bibitem{beje1} \bysame,  \emph{Optimal multilinear restriction estimates for a class of surfaces with curvature},  arXiv:1606.02634.

 \bibitem{bennett} J. Bennett, \emph{Aspects of multilinear harmonic analysis related to transversality,} Harmonic analysis and partial differential equations, 1--28, Contemp. Math., 612, Amer. Math. Soc., Providence, RI, 2014.
 
\bibitem{bbfl}  J. Bennett, N. Bez, T. C. Flock, and S. Lee,  \emph{Stability of the Brascamp-Lieb constant and applications}, to appear in American Journal of Mathematics, arXiv:1508.07502. 


\bibitem{becata} J. Bennett, A. Carbery, and T. Tao, \emph{On the multilinear restriction and Kakeya conjectures,}
Acta Math. {\bf 196} (2006), 261--302.

\bibitem{b1} J. Bourgain,
\textit{Besicovitch type maximal operators and applications to
Fourier analysis}, Geom. Funct. Anal. \textbf{1} (1991), 147--187.

\bibitem{b4} \bysame, \emph{$L^p$-estimates for oscillatory integrals in several
variables}, Geom. Funct. Anal. {\bf 1} (1991), 321--374.


\bibitem{b2} \bysame,
\textit{On the restriction and multiplier problems in $\mathbb
R\sp 3$}, In \textit{Geometric aspects of functional
analysis-seminar} 1989-90, Lecture Notes in Math. vol. 1469,
Springer-Berlin, (1991), p. 179-191.

\bibitem{bo-point} \bysame, \emph{On the Schrödinger maximal function in higher dimension,}   Tr.  Mat. Inst. Steklova   280  (2013),   translation in   Proc.   Steklov   Inst.   Math.   280   (2013),    46--60

\bibitem{bd} J. Bourgain and C.  Demeter, \emph{The proof of the 
$l^2$  decoupling conjecture},  Ann. of Math.  (2)   {\bf 182}  (2015),  351--389.


\bibitem{bdg}  J. Bourgain, C.  Demeter, and L. Guth, \emph{Proof of the main conjecture in Vinogradov's mean value theorem for degrees higher than three,}  Ann. of Math. (2) {\bf 184}   (2016),  633--682. 

\bibitem{bogu} J. Bourgain and L. Guth, \emph{Bounds on oscillatory integral operators based on multilinear estimates},
Geom. Funct. Anal. {\bf 21} (2011), 1239-1295.

\bibitem{ca} A. Carbery,
\textit{The boundedness of the maximal Bochner-Riesz operator on $L\sp{4}(R\sp{2})$},
 Duke Math. J. \textbf{50}  (1983), 409-416.

\bibitem{caesc} \bysame,
\emph{Radial Fourier multipliers and associated maximal functions,}  Recent progress in Fourier analysis (El Escorial, 1983),  49--56, North-Holland Math. Stud., 111, North-Holland, Amsterdam, 1985.



\bibitem{cgt} A. Carbery, G. Gasper,  W. Trebels,
\emph{Radial Fourier multipliers of $L\sp{p}(R\sp{2})$,}  Proc. Nat. Acad. Sci. U.S.A.  {\bf 81}  (1984),  no. 10, Phys. Sci., 3254--3255.



\bibitem{caruve} A. Carbery, J.L. Rubio de Francia and  L. Vega,
\emph{Almost everywhere summability of Fourier integrals},
J. London Math. Soc. (2) {\bf 38} (1988),  513--524.



\bibitem{cv} A. Carbery and S. Valdimarsson, \textit{The endpoint multilinear Kakeya theorem via the Borsuk-Ulam theorem}
{\bf 264} (2013), 1643--1663.

 \bibitem{cs} L. Carleson and P. Sj\"olin, \textit{Oscillatory integrals and a
multiplier problem for the disc},  Studia Math. \textbf{44}
(1972), 287-299.


\bibitem{ch}
M. Christ, \textit{On almost everywhere convergence of Bochner-Riesz means in higher dimensions},
 Proc. Amer. Math. Soc. \textbf{95} (1985), 16-20 .


\bibitem{christ-wtBR} \bysame,
\emph{Weak type endpoint bounds for Bochner--Riesz multipliers,} Rev. Mat. Iberoamericana {\vbf 3} (1987),  25--31.

\bibitem{christ-rough} \bysame,
\emph{Weak type (1,1) bounds for rough operators,} Ann. of Math. (2) {\vbf 128} (1988), 19--42.

\bibitem{li-guth}  X. Du, L. Guth, and X. Li, \emph{A sharp Schrodinger maximal estimate in $\mathbb R^2$ 
}, arXiv:1612.08946.

\bibitem{fe1}
C. Fefferman, \textit{Inequalities for strongly singular
convolution operators},  Acta Math. \textbf{124} (1970),  9--36.

\bibitem{fe2}
\bysame, \textit{The multiplier problem for the ball}, Annals of
Math. \textbf{94} (1971), 330--336.


\bibitem{fef2} \bysame, \emph{A note on spherical summation multipliers}, Israel J. Math. {\bf 15} (1973), 44--52.  






\bibitem{gu} L. Guth, \textit{The endpoint case of the Bennett-Carbery-Tao multilinear Kakeya conjecture
Acta Math.} {\bf 205} (2) (2010), 263--286.

\bibitem{gu1} \bysame, \emph{A restriction estimate using polynomial partitioning},  J. Amer. Math. Soc. {\bf 29} (2016),  371--413.  
  arXiv:1603.04250. 
  
\bibitem{gu2} \bysame, \emph{Restriction estimates using polynomial partitioning II}, arXiv:1603.04250.
 
 



\bibitem{hor}L. H\"ormander, \textit{Oscillatory integrals and multipliers on $FL^p$,} Ark. Mat. {\bf 11}, 1--11.
(1973).


\bibitem{lee1} S. Lee, \textit{Improved bounds for Bochner-Riesz and maximal Bochner-Riesz operators,}
Duke Math. J. {\bf 122} (2004), 205--232.

\bibitem{lee2} \bysame,
\textit{Linear and bilinear estimates for oscillatory integral operators related to restriction to hypersurfaces,}  J. Funct. Anal. {\bf 241} (2006), 56--98.

\bibitem{lee3} \bysame, \emph{ On pointwise convergence of the solutions to Schrödinger equations in $\mathbb  R^2$,}  Int. Math. Res. Not. 2006, Art. ID 32597, 21 pp.

\bibitem{lrs} S. Lee, K. Rogers and A. Seeger, \emph{Improved bounds for Stein's square functions,}
Proc. Lond. Math. Soc. (3) \textbf{104} (2012), 1198--1234.

\bibitem{lrs1} \bysame, \emph{On space-time estimates for the Schr\"odinger operator}, J. Math.
Pures Appl. (9) {\bf 99} (2013), 62--85.

\bibitem{lrs2} \bysame, \emph{Square functions and maximal
operators associated with radial Fourier multipliers,} Advances in Analysis: The Legacy of Elias M. Stein,
Princeton University Press,  2014, pp. 273--302.

 \bibitem{lee-seeger} S. Lee and A. Seeger, \emph{On radial Fourier multipliers and almost everywhere convergence,}  J. Lond. Math. Soc. {\bf 91} (2015),  105--126. 
 
 \bibitem{l-v0} S. Lee and A. Vargas, \emph{Sharp null form estimates for the wave equation,} Amer. J. Math. {\bf 130} (2008), 1279--1326.
 
\bibitem{l-v}  \bysame,   \emph{Restriction estimates for some surfaces with vanishing curvatures},
J. Funct. Anal. {\bf 258} (2010),  2884--2909.

\bibitem{l-v1} \bysame, \emph{On the cone multiplier in $\mathbb R^3$},  J. Funct. Anal. {\bf 263} (2012),  925--940.


\bibitem{ou-wang} {Y. Ou and H. Wang}, \emph{A cone restriction estimate using polynomial partitioning}, 
arXiv:1704.05485

\bibitem{jramos} J.  Ramos,  \emph{The trilinear restriction estimate with sharp dependence on the transversality},  arXiv:1601.05750.

\bibitem{ru0} J.L. Rubio de Francia, \emph{A Littlewood-Paley inequality for arbitrary intervals,}
Rev. Mat. Iberoamericana {\bf 1} (1985), 1--14.


\bibitem{se0} A. Seeger,
\emph{On quasiradial Fourier multipliers and their maximal
functions,} J. Reine Angew. Math. {\bf 370} (1986), {61--73}.

\bibitem{se1} \bysame,
\emph {Endpoint inequalities for Bochner--Riesz multipliers in the
plane,}  Pacific J. Math.  {\vbf 174}  (1996),   543--553.

\bibitem{shayya} B. Shayya, \emph{Weighted restriction estimates using polynomial partitioning}, arXiv:1512.03238.

\bibitem{stein58} E. M. Stein, \emph{Localization and summability of multiple Fourier series,}
Acta Math. {\bf 100} (1958),  93--147.

\bibitem{stein84} \bysame, \emph{Oscillatory integrals in Fourier analysis}, Beijing lectures in harmonic analysis (Beijing, 1984), 307--355, Ann. of Math. Stud., 112, Princeton Univ. Press, Princeton, NJ, 1986.


\bibitem{st2} \bysame, \textit{Harmonic analysis: real-variable methods, orthogonality, and
oscillatory integrals},   Princeton University Press, Princeton,
(1993).

\bibitem{su1} G. Sunouchi, \emph{On the Littlewood-Paley function $g\sp*$ of
multiple Fourier integrals and Hankel multiplier transformations,} 
 T\^ohoku Math. J. (2)  {\bf 19}  (1967), 496--511.

\bibitem{tao-weak0}  T. Tao, \emph{Weak-type endpoint bounds for Riesz means},  Proc. Amer. Math. Soc. {\bf 124} (1996), 2797--2805.
 

\bibitem{tao-weak}   \bysame, 
\textit{The weak-type endpoint Bochner-Riesz conjecture and related
topics}, Indiana Univ. Math. J. \textbf{47} (1998), 1097--1124.

\bibitem{t2} \bysame,
\textit{The Bochner-Riesz conjecture implies the restriction
conjecture}, Duke Math. J. \textbf{96}  (1999), 363-375.


\bibitem{tao-maximal} \bysame, \emph{On the maximal Bochner-Riesz conjecture in the plane for
$p<2$},  Trans. Amer. Math. Soc. {\bf 354} (2002), 1947--1959.


\bibitem{t5} \bysame, \emph{A sharp bilinear restrictions estimate for
paraboloids},  Geom. Funct. Anal. {\bf 13} (2003), 1359--1384.

\bibitem{tv1} T. Tao and A. Vargas,
\textit{A bilinear approach to cone multipliers I. Restriction
estimates}, Geom. Funct. Anal. \textbf{10} (2000), 185--215.

\bibitem{tv2} \bysame,
\textit{A bilinear approach to cone multipliers II. Applications},
Geom. Funct. Anal. \textbf{10}  (2000), 216--258.

\bibitem{tvv} T. Tao,  A. Vargas and L. Vega, 
\textit{A bilinear approach to the restriction and Kakeya
conjecture}, J. Amer. Math. Soc. \textbf{11} (1998), 967--1000.

\bibitem{temur} F. Temur, \emph{A Fourier restriction estimate for surfaces of positive curvature in  $\mathbb R^6$},  Rev. Mat. Iberoam.     {\bf 30}     (2014),   1015--1036.

\bibitem{tomas} P. Tomas,
\emph{A restriction theorem for the Fourier transform},  Bull. Amer. Math. Soc. {\bf 81} (1975), 477--478.

\bibitem{wise} L. Wisewell, \emph{Kakeya sets of curves},  Geom. Funct. Anal. {\bf 15} (2005), 1319--1362.


\bibitem{w3} T. Wolff,
\textit{ Local smoothing type estimates on $L\sp p$ for large $p$},
Geom. Funct. Anal. \textbf{10} (2000), 1237--1288.

\bibitem{w2} \bysame,  \textit{ A sharp cone restriction
estimate},  Annals of Math. \textbf{153} (2001),  661--698.

\bibitem{rzhang} R. Zhang, \emph{The Endpoint Perturbed Brascamp-Lieb Inequality with Examples}, 
arXiv:1510.09132

\end{thebibliography}
\end{document}